\documentclass[10pt]{article}

\usepackage{amsmath, amsfonts, amsthm, amssymb, enumerate, url, stmaryrd, esint, float}
\usepackage[all,arc,color]{xy}
\usepackage{mathrsfs}
\usepackage{mathtools}
\usepackage[utf8]{inputenc}
\usepackage{tikz-cd}

\usepackage[toc,page]{appendix}

\begingroup
 \makeatletter
 \@for\theoremstyle:=definition,remark,plain\do{%
 \expandafter\g@addto@macro\csname th@\theoremstyle\endcsname{%
 \addtolength\thm@preskip\parskip
 }
 }
 \endgroup
 
 \usepackage[colorlinks = true,
 linkcolor = blue,
 urlcolor = blue,
 citecolor = blue,
 anchorcolor = blue]{hyperref}
 
 \usepackage[parfill]{parskip}
 \usepackage{anysize}
 \marginsize{1in}{1in}{1in}{1in}
 \usepackage{thmtools}

\declaretheorem[name=Theorem,numberwithin=section]{thm}
\declaretheorem[name=Proposition,numberlike=thm]{prop}
\declaretheorem[name=Lemma,numberlike=thm]{lemma}
\declaretheorem[name=Corollary,numberlike=thm]{cor}

\declaretheorem[name=Definition,style=definition,qed=$\blacktriangle$,numberlike=thm]{defn}
\declaretheorem[name=Example,style=definition,qed=$\blacktriangle$,numberlike=thm]{ex}
\declaretheorem[name=Remark,style=definition,qed=$\blacktriangle$,numberlike=thm]{rmk}

\declaretheorem[name=Assumption,style=definition,qed=$\blacktriangle$,numberlike=thm]{ass}

\newcommand{\R}{\mathbb{R}}

\newcommand{\euc}{\mathrm{euc}}
\newcommand{\loc}{\mathrm{loc}}
\newcommand{\round}{\mathrm{round}}
\newcommand{\inj}{\mathrm{inj}}
\newcommand{\diam}{\mathrm{diam}}
\newcommand{\Ric}{\mathrm{Ric}}
\newcommand{\dist}{\mathrm{dist}}
\newcommand{\sgn}{\mathrm{sgn}}
\newcommand{\real}{\mathrm{Re}}

\newcommand{\pa}[2]{\frac{\partial #1}{\partial #2}}
\newcommand{\paop}[1]{\pa{}{#1}}
\newcommand{\supp}{\mathrm{supp}}
\newcommand{\esssup}{\mathrm{ess \, sup}}
\newcommand{\essinf}{\mathrm{ess \, inf}}

\newcommand{\ep}{\varepsilon}
\newcommand{\perm}{\epsilon}
\newcommand{\current}[1]{\llbracket #1 \rrbracket}

\newcommand{\NN}{\mathbb{N}}
\newcommand{\cT}{\mathcal{T}}
\newcommand{\cH}{\mathcal{H}}
\newcommand{\cN}{\mathcal{N}}
\newcommand{\cS}{\mathcal{S}}

\newcommand{\tr}{\operatorname{tr}}
\newcommand{\vol}{\mathsf{vol}}
\newcommand{\G}{\mathrm{G}_2}
\newcommand{\Spin}[1]{\mathrm{Spin}(#1)}
\newcommand{\spa}{\operatorname{span}}
\newcommand{\im}{\operatorname{im}}

\newcommand{\smith}{\eth}
\newcommand{\crit}{\operatorname{crit}}
\newcommand{\Div}{\operatorname{div}}

\begin{document}

\title{Bubble Tree Convergence of \\ Conformally Cross Product Preserving Maps}

\author{Da Rong Cheng \\ \textit{Department of Mathematics, University of Chicago} \\ \tt{chengdr@uchicago.edu} \and Spiro Karigiannis \\ \textit{Department of Pure Mathematics, University of Waterloo} \\ \tt{karigiannis@uwaterloo.ca} \and Jesse Madnick \\ \textit{Department of Mathematics and Statistics, McMaster University} \\ \tt{madnickj@mcmaster.ca}}

\date{October 2, 2019}

\maketitle

\begin{abstract}
We study a class of weakly conformal $3$-harmonic maps, called associative Smith maps, from $3$-manifolds into $7$-manifolds that parametrize associative $3$-folds in Riemannian $7$-manifolds equipped with $\mathrm{G}_2$-structures. Associative Smith maps are solutions of a conformally invariant nonlinear first order PDE system, called the Smith equation, that may be viewed as a $\mathrm{G}_2$-analogue of the Cauchy--Riemann system for $J$-holomorphic curves.

In this paper, we show that associative Smith maps enjoy many of the same analytic properties as $J$-holomorphic curves in symplectic geometry. In particular, we prove: (i) an interior regularity theorem, (ii) a removable singularity result, (iii) an energy gap result, and (iv) a mean-value inequality. While our approach is informed by the holomorphic curve case, a number of nontrivial extensions are involved, primarily due to the degeneracy of the Smith equation.

At the heart of above results is an $\varepsilon$-regularity theorem that gives quantitative $C^{1,\beta}$-regularity of $W^{1,3}$ associative Smith maps under a smallness assumption on the $3$-energy. The proof combines previous work on weakly $3$-harmonic maps and the observation that the associative Smith equation demonstrates a certain ``compensation phenomenon'' that shows up in many other geometric PDEs.

Combining these analytical properties and the conformal invariance of the Smith equation, we explain how sequences of associative Smith maps with bounded $3$-energy may be conformally rescaled to yield bubble trees of such maps. When the $\mathrm{G}_2$-structure is closed, we prove that both the $3$-energy and the homotopy are preserved in the bubble tree limit. This result may be regarded as an associative analogue of part of Gromov's Compactness Theorem in symplectic geometry.
\end{abstract}

\tableofcontents

\section{Introduction}

\subsection{Motivation} \label{sec:motivation}

In symplectic geometry, moduli spaces of holomorphic curves play an important role~\cite{Gro, MS}. Starting with a symplectic manifold $(M, \omega)$, one chooses an almost complex structure $J_M$ that is compatible with $\omega$ in some sense. A map $u \colon (\Sigma, J_\Sigma) \to (M, \omega, J_M)$ from a Riemann surface $(\Sigma, J_\Sigma)$ is then called a $(J_\Sigma, J_M)$-holomorphic map (or simply a holomorphic curve) if it solves the Cauchy--Riemann system
\begin{equation} \label{eq:CR}
J_M \circ du = du \circ J_\Sigma.
\end{equation}
Various moduli spaces of solutions to~\eqref{eq:CR} can then be studied. A priori, these moduli spaces are almost never compact, largely due to their conformal invariance (see property (i) below); their compactifications are described by Gromov's Compactness Theorem. From the compactified moduli spaces, one can derive powerful invariants of the original symplectic manifold $(M,\omega)$, independent of the original choice of $J_M$. Here we emphasize that in its full generality, Gromov's Compactness Theorem allows the complex structure $J_{\Sigma}$ to vary in such way that the Riemann surface $(\Sigma, J_{\Sigma})$ degenerates. However, for the purposes of the present paper we are primarily interested in the case where $J_{\Sigma}$ is fixed.

Underpinning Gromov's Compactness Theorem in this case are a litany of crucial algebraic and analytic properties enjoyed by solutions of the Cauchy--Riemann system~\eqref{eq:CR}. To recall these, let $u \colon \Sigma^2 \to M^{2n}$ be a smooth map, and equip both $(\Sigma, J_\Sigma)$ and $(M, J_M, \omega)$ with compatible Riemannian metrics. Define the $2$-energy of $u$ on a measurable set $A \subset \Sigma$ by
\begin{align*}
E(u; A) & = \frac{1}{2} \int_A |du|^2\,\vol_\Sigma
\end{align*}
where $\vol_\Sigma$ is the volume form on $\Sigma$, and write $\overline{\partial} u := \textstyle \frac{1}{2}\left( du + J_M \circ du \circ J_\Sigma \right)$. Then we have
\begin{enumerate}[(i)]
\item (\emph{Conformal Invariance}.) If $u$ satisfies~\eqref{eq:CR} and $f \colon \Sigma \to \Sigma$ is a conformal diffeomorphism, then $u \circ f$ satisfies~\eqref{eq:CR}.
\item (\emph{Weak Conformality}.) Every solution of~\eqref{eq:CR} is weakly conformal.
\item (\emph{Calibrated Image}.) If $u$ is an immersion satisfying~\eqref{eq:CR}, then $u(\Sigma) \subset M$ is calibrated by $\omega$.
\item (\emph{Energy Identity}.) Every smooth map $u \colon \Sigma \to M$ satisfies $\frac{1}{2} |du|^2\,\vol_\Sigma = |\overline{\partial}u|^2\,\vol_\Sigma + u^*\omega$. Therefore
$$E(u; \Sigma) = \int_\Sigma |\overline{\partial} u|^2\,\vol_\Sigma + \int_\Sigma u^*\omega.$$
In particular, $E(u; \Sigma) \geq \int_\Sigma u^*\omega$, with equality if and only if $\overline{\partial}u = 0$. Thus, by Stokes's theorem, holomorphic curves minimize $2$-energy in their homology class. Moreover, if $\partial \Sigma = \O$, then the minimum $2$-energy attained is $\int_\Sigma u^*\omega = [\omega] \cdot u_*[\Sigma]$, which depends only on the topological data $[\omega] \in H^2(M; \R)$ and $u_*[\Sigma] \in H_2(M; \R)$.\end{enumerate}

Note that Properties (i)--(iv) may all be proved using linear algebra. Significantly less trivial are the following five analytic properties of holomorphic curves. In the sequel, we let $B(r) \subset \R^2$ denote the open ball of radius $r > 0$ centered at the origin.
 \begin{enumerate}[(i)] \setcounter{enumi}{4}
 \item (\emph{Mean value inequality}.) There exist constants $C, \ep_0 > 0$ such that every holomorphic curve $u \colon B(2r) \to M$ with $E(u; B(2r)) < \ep_0$ satisfies
$$\sup_{B(r)} |du|^2 \leq \frac{C}{r^2}\, E(u; B(2r)).$$
\item (\emph{Interior Regularity}.) If $u \in W^{1,2}(\Sigma; M)$ satisfies~\eqref{eq:CR} almost everywhere, then $u$ is $C^\infty$.
\item (\emph{Removable Singularities}.) If $u \colon B(1) \setminus \{0\} \to M$ is a smooth holomorphic curve with
\begin{equation*}
E(u; B(1) \setminus \{0\}) < \infty,
\end{equation*}
then $u$ extends to a smooth holomorphic curve $u \colon B(1) \to M$.
\item (\emph{Energy Gap}.) There exists a constant $\ep_0 > 0$ such that every holomorphic curve $u \colon S^2 \to M$ with $E(u; S^2) < \ep_0$ is constant.
\item (\emph{Compactness Modulo Bubbling}.) Let $u_n \colon (\Sigma, J_\Sigma) \to (M, J_M)$ be a sequence of holomorphic curves with $E(u_n; \Sigma) \leq C$. Then there exists a holomorphic curve $u_\infty \colon \Sigma \to M$ and a (possibly empty) finite set $\mathcal{S} = \{x_1, \ldots, x_q\} \subset \Sigma$ such that, after passing to a subsequence of $\{u_n\}$:
\begin{enumerate}[(a)]
\item We have $u_n \to u_\infty$ in $C^1_{\loc}$ on $\Sigma \setminus \mathcal{S}$.
\item We have $|du_n|^2\,d\mu \to |du_\infty|^2\,d\mu + \sum_{i=1}^q m_i \delta(x_i)$ as Radon measures on $\Sigma$, where $\delta(x_i)$ is the Dirac delta measure at $x_i$. Also, each $m_i \geq \frac{1}{2} \ep_0$ with $\ep_0$ as in (viii).
\item If $\| du_n \|_{L^p} \leq C$ for some $p \in (2, \infty]$, then $\mathcal{S} = \varnothing$.
\end{enumerate}
\end{enumerate}
Note that by (iv), any sequence of holomorphic curves representing the same homology class will satisfy a uniform $2$-energy bound, and hence property (ix) may be applied to such a sequence. For a more detailed discussion of (i)--(ix), we refer the reader to~\cite{MS, PW, Ye}.
 
In fact, several important conformally-invariant geometric PDE systems, such as the Yang--Mills equation on $4$-manifolds~\cite{DK, Fe} and the harmonic map equation on surfaces~\cite{P}, satisfy analogues of properties (i)--(ix) with respect to an appropriate conformally invariant energy functional. In each of these cases, properties (i)--(ix) can be used to construct ``bubble trees" of such objects, which may be regarded as comprising the boundary of the corresponding compactified moduli space. In favorable cases, such as that of holomorphic curves~\cite{PW} or harmonic maps from surfaces~\cite{P}, the bubble tree enjoys two additional properties, which can be loosely stated as follows:
\begin{enumerate}[(i)] \setcounter{enumi}{9}
\item The 2-energy is preserved in the bubble tree limit.
\item The bubble tree has no necks. Consequently, homotopy is preserved in the bubble tree limit.
\end{enumerate}
Precise statements of (x) and (xi) may be found in~\cite{MS, PW, Ye}.

\textbf{In this paper, we demonstrate that analogues of Properties (i)--(xi) hold for a class of weakly conformal maps that parametrize associative $3$-folds in Riemannian $7$-manifolds equipped $\G$-structures.} We are motivated in part by the larger project of ``counting" associative submanifolds in order to obtain invariants of $\G$-manifolds, which we briefly discuss.

Recall that a $\G$-structure on a $7$-manifold $M$ is a $\G$-subbundle of the frame bundle of $M$. Equivalently, it is a choice of a $3$-form $\varphi \in \Omega^3(M)$ with the property that the bilinear form
$$B_\varphi \in \Gamma(\text{Sym}^2(T^*M) \otimes \Lambda^7(T^*M) )$$
given by $B_\varphi(X,Y) = \iota_X\varphi \wedge \iota_Y \varphi \wedge \varphi$ is definite. Since $\G \leq \text{SO}(7)$, a $\G$-structure on $M$ naturally induces a Riemannian metric $h_\varphi$ on $M$, although the correspondence $\varphi \mapsto h_\varphi$ is not injective.

Two classes of submanifolds are of particular interest in $\G$ geometry: the associative $3$-folds, which are semi-calibrated by $\varphi$, and the coassociative $4$-folds, which are semi-calibrated by $\ast\varphi$. It is often useful to regard associative $3$-folds and coassociative $4$-folds, respectively, as $\G$ analogues of holomorphic curves and special Lagrangian submanifolds. One may also consider gauge-theoretic objects on (bundles over) $7$-manifolds with $\G$-structures, such as $\G$-monopoles and $\G$-instantons, by analogy with familiar objects in $3$ dimensions and $4$ dimensions. For background on $\G$ geometry, we refer the reader to~\cite{J-book}.
 
In~\cite[$\S$3]{DT}, Donaldson--Thomas described an analogy between $\G$-instantons on $7$-manifolds and flat connections on $3$-manifolds, suggesting that $\G$-instantons may be used to define $7$-dimensional analogues of the Casson invariant and Floer homology. Then, Tian~\cite{Ti} showed that sequences of $\G$-instantons may bubble along associative submanifolds, complicating the compactification of the moduli space. To account for this, Donaldson--Segal~\cite[$\S$6.2]{DS} conjectured that a count of $\G$-instantons, appropriately weighted by a count of associative submanifolds, ought to yield a deformation invariant of $\G$-manifolds.

In~\cite{HW} and~\cite{Ha}, Haydys and Walpuski suggested that these weights might be described by considering generalized Seiberg--Witten monopoles on the associatives. Doan--Walpuski~\cite{DW} explored this idea further, describing a Floer homology group generated by both associatives and ADHM Seiberg--Witten monopoles that could lead to a deformation invariant. A rather different perspective is offered by Joyce in~\cite{J}, who puts forth a conjectural notion of $\G$ quantum cohomology. (Some other earlier work on the analysis of counting associative submanifolds in a particular special case (``thin'') was done by Leung--Wang--Zhu~\cite{LWZ, LWZ-2}.)

One of the main challenges confronting each of these proposed deformation invariants is that $1$-parameter families of associative immersions may degenerate to a map of the form $u \circ \pi$, where $u \colon \Sigma^3 \to M^7$ is an associative immersion and $\pi \colon \widetilde{\Sigma}^3 \to \Sigma^3$ is a branched cover. Thus, we are led to seek a class of maps whose immersions parametrize associative submanifolds, but is sufficiently large to allow for nonempty critical loci.

Indeed, in \cite[page 90, item (i)]{LWZ-2}, Leung--Wang--Zhu point out that the conformality of the Cauchy--Riemann equation plays a key role for holomorphic curves, and suggest that an analogous conformally invariant PDE system for associative submanifolds would be desirable.

This is where the work of Smith~\cite{Sm} enters the picture. In his 2011 PhD thesis, Smith introduced a class of maps that generalize holomorphic maps of Riemann surfaces into almost Hermitian manifolds. Namely, a map $u \colon (\Sigma^n, g, P_\Sigma) \to (M^m, h, P_M)$ between Riemannian manifolds $(\Sigma^n, g)$ and $(M^m, h)$ endowed with compatible $k$-fold vector cross products $P_\Sigma \colon \Gamma(\Lambda^{k}T\Sigma) \to \Gamma(T\Sigma) $ and $P_M \colon \Gamma(\Lambda^kTM) \to \Gamma(TM)$ is said to be $k$-\textit{Smith} (or \textit{Smith} when $k$ is clear from context) if $du$ preserves the vector cross products up to a scale factor $\lambda = \lambda(du)$ depending only on $du$. That is,
\begin{equation} \label{eq:Gen-Smith}
P_M \circ \Lambda^kdu = \lambda(du)\,du \circ P_\Sigma.
\end{equation}
Equation~\eqref{eq:Gen-Smith} is henceforth referred to as the Smith equation. In brief, Smith maps are ``conformally cross product preserving." Smith's original terminology~\cite{Sm} for such maps was ``multiholomorphic'', which we do not use.

Vector cross products are rather special geometric structures. Indeed, a $k$-fold vector cross product $P \colon \Gamma(\Lambda^kTX) \to \Gamma(TX)$ on a Riemannian manifold $(X, g)$ induces a semi-calibration $\alpha \in \Omega^{k+1}(X)$ by raising an index:
\begin{equation}
\alpha(v_1, \ldots, v_{k+1}) = g( P(v_1, \ldots, v_k), v_{k+1} ). \label{eq:induced-cal}
\end{equation}
Moreover, by the work of Brown--Gray~\cite{BG}, $k$-fold vector cross products are classified into four families, and each is equivalent to particular reduction of the structure group of $X$. We summarize their classification in the following table:
\begin{center}
 \begin{tabular}{ | c | c | l | }
 \hline
 $k$ & $\dim(X)$ & Equivalent Structure \\ \hline
 $k$ & $k+1$ & Orientation \\ 
 $1$ & $2r$ & Almost Hermitian Structure \\ 
 $2$ & $7$ & G$_2$-Structure \\ 
 $3$ & $8$ & Spin$(7)$-Structure \\ \hline
 \end{tabular}
\end{center}
In particular, $2$-fold vector cross products can only occur in dimensions $3$ and $7$, while $3$-fold vector cross products can only occur in dimensions $4$ and $8$.

Smith primarily considered (as do we) the case in which the domain $\Sigma^n$ and target $M^m$ carry $(n-1)$-fold vector cross products. In that case, it can be shown that nonconstant solutions of~\eqref{eq:Gen-Smith} must satisfy $\lambda(du) = \frac{1}{(\sqrt{n})^{n-2}}|du|^{n-2}$, so that~\eqref{eq:Gen-Smith} reads
\begin{equation} \label{eq:Spec-Smith}
P_M \circ \Lambda^{n-1}du = \frac{1}{(\sqrt{n})^{n-2}} |du|^{n-2}\,du \circ P_\Sigma.
\end{equation}
In particular, when $n = 2$, equation~\eqref{eq:Spec-Smith} is the Cauchy--Riemann system~\eqref{eq:CR} for holomorphic curves.

Note that~\eqref{eq:Spec-Smith} is a nonlinear first order PDE system involving $m$ equations on $m$ unknown functions of $n$ variables. When $n = 2$, the system is elliptic, but when $n > 2$, it degenerates at points of the critical set
\begin{equation*}
\crit_u = \{x \in \Sigma \mid (du)_x = 0 \}.
\end{equation*}

Smith proved~\cite{Sm} that solutions of~\eqref{eq:Spec-Smith} satisfy the analogues of properties (i)--(iv). To state his result, let $\alpha \in \Omega^n(M)$ denote the semi-calibration induced by $P_M$ as in~\eqref{eq:induced-cal}, and denote the $n$-energy of a map $u \in W^{1,n}(\Sigma; M)$ on a measurable set $A \subset \Sigma$ by
\begin{align*}
E(u; A) & = \frac{1}{(\sqrt{n})^n} \int_A |du|^n\,\vol_\Sigma.
\end{align*}

\begin{thm}[Smith~\cite{Sm}] \label{thm:Smith-Four-Parts}
Let $u \colon \Sigma^n \to M^m$ be a smooth map, where $n < m$.
\begin{enumerate}[(i)]
\item If $u$ satisfies~\eqref{eq:Spec-Smith} and $f \colon \Sigma \to \Sigma$ is a conformal diffeomorphism, then $u \circ f$ satisfies~\eqref{eq:Spec-Smith}.
\item Every solution of~\eqref{eq:Spec-Smith} is weakly conformal.
\item If $u$ is an immersion satisfying~\eqref{eq:Spec-Smith}, then $u(\Sigma) \subset M$ is semi-calibrated by $\alpha$.
\item We have
$$ E(u; \Sigma) \geq \int_\Sigma u^*\alpha. $$
Moreover, equality holds if and only if $u$ satisfies~\eqref{eq:Spec-Smith}. Thus, if $d\alpha = 0$, then Smith maps minimize $n$-energy in their homology class. If, in addition, $\partial \Sigma = \varnothing$, then the minimum $n$-energy attained can simply be expressed as $\int_\Sigma u^*\alpha = [\alpha] \cdot u_*[\Sigma]$, which depends only on topological data.
\end{enumerate}
\end{thm}

Our interest is in the exceptional cases $(n,m) = (3,7)$ and $(n,m) = (4,8)$. We shall refer to $2$-Smith maps from a $3$-manifold into a $7$-manifold as \textit{associative Smith maps}, and refer to $3$-Smith maps from a $4$-manifold into an $8$-manifold as \textit{Cayley Smith maps}. The corresponding cases of~\eqref{eq:Spec-Smith} are called the \textit{associative Smith equation} and \textit{Cayley Smith equation}, respectively.

In light of Smith's result, it is natural to ask whether analogues of (v)--(xi) for holomorphic curves also hold true for associative and Cayley Smith maps. \textbf{The purpose of the present paper is to establish all seven of these properties.} In this present paper, our main focus is on the associative Smith equation, although analogous results hold for the Cayley Smith equation.

\begin{rmk}
Smith's theorem indicates that solutions of~\eqref{eq:Spec-Smith} are attractive parametrizations of calibrated submanifolds. Fortunately, associative (respectively, Cayley) Smith maps are at least as abundant as associative (respectively, Cayley) submanifolds. Indeed, given an associative immersion $u \colon \Sigma^3 \to (M^7, h)$ into a $7$-manifold $M$ with a $\G$-structure and a smooth function $F \colon \Sigma \to \R$, equipping the domain $\Sigma$ with the Riemannian metric $g = e^{F} u^*h$ makes $u$ into an associative Smith map.
\end{rmk}

\subsection{Methods and main results} \label{sec:results-methods}

\textbf{Methods.} We emphasize that many of the techniques used in establishing properties of holomorphic curves ($n = 2$) \textit{do not} carry over \textit{mutatis mutandis} to our situation ($n = 3$ or 4). This is primarily due to the following two reasons, which are themselves related:\begin{enumerate}[(i)]
\item The second order elliptic system obtained by differentiating~\eqref{eq:Spec-Smith} is of ``$n$-harmonic type'' and thus degenerate when $n \neq 2$, as opposed to being uniformly elliptic when $n = 2$. Consequently, the methods used in proving regularity and compactness results for holomorphic curves do not directly apply to our case.
\item Solutions to~\eqref{eq:Spec-Smith} for $n \neq 2$ need not be ``immersions except for a discrete set" as they are when $n = 2$. In fact, it is not known whether a unique continuation theorem holds when $n \neq 2$ (see~\cite[Theorem 2.3.2]{MS} for the case $n = 2$). Thus, in extending properties (x) and (xi), one runs into complications when trying to apply results about minimal submanifolds that prove useful in the study of holomorphic curves, such as the isoperimetric inequality or the monotonicity formula.
\end{enumerate}

Of the two issues above, the second is less of a problem as we simply avoid it by working directly with the $n$-energy of our maps, as in~\cite{MS}, instead of applying minimal surface results (or more generally geometric measure theory) to their images, as in~\cite{PW}. The reader interested in the distinction between these two approaches is suggested to compare, for example,~\cite[Theorem 4.4.1(i)]{MS} and~\cite[Lemma 3.1]{PW}.

As for the first issue, one expects previous work on $n$-harmonic maps to help in dealing with the degeneracy of the equation, and indeed we are aided greatly by the results of Uhlenbeck~\cite{U}, Giaquinta--Modica~\cite{GiMa}, Duzaar--Fuchs~\cite{DF} and Duzaar--Mingione~\cite{DM}. However, there is one crucial aspect that existing literature on $n$-harmonic maps does not fully cover, namely the continuity of $W^{1, n}$-weak solutions (on $n$-dimensional domains), which has only been established in full generality for the case $n = 2$~\cite{He}. Analogous results for $n \neq 2$ all require additional assumptions (see for example~\cite{Ha-Li, T, TW}) that do not generally fit with our setting. This is where the special structure of the Smith equation~\eqref{eq:Spec-Smith} comes in. Specifically, we are able to show that equation~\eqref{eq:Spec-Smith} demonstrates a ``compensation phenomenon'' which has been observed in many other geometric PDEs, including $n$-harmonic maps into special targets. See $\S$\ref{sec:int-reg-sub}, in particular Remarks~\ref{hardyspacermk} and~\ref{rmk:compensation}, for details. Finally, we point out that results on minimizers of the $n$-energy, such as those obtained by Hardt--Lin~\cite{Ha-Li}, do not apply to our situation, as Smith maps are only minimizing within a fixed homology class.

\textbf{Main results.} We summarize the main results of the paper. In this section, the statements of the theorems are slightly simplified from their precise forms.

 Let $(\Sigma^3, g, \ast)$ denote an oriented Riemannian $3$-manifold equipped with the $2$-fold vector cross product $\ast = P_\Sigma \colon \Gamma(\Lambda^2T\Sigma) \to \Gamma(T\Sigma)$ given by the Hodge operator.

Let $(M^7, h, J)$ denote a closed oriented Riemannian $7$-manifold equipped with a $2$-fold vector cross product $J = P_M \colon \Gamma(\Lambda^2TM) \to \Gamma(TM)$. Note that the data $(h,J)$ induces a definite $3$-form $\varphi \in \Omega^3(M)$, and hence a $\G$-structure on $M$, by raising an index:
$$\varphi(x,y,z) = h(J(x,y), z).$$
For the moment, we do not assume that $\varphi$ is closed or coclosed.

For technical convenience, we isometrically embed $M$ in a Euclidean space $\R^d$. The Sobolev space $W^{1, p}(\Sigma; M)$ is then by definition
\[
W^{1, p}(\Sigma; M) = \{ u \in W^{1, p}(\Sigma; \R^{d})\ |\ u(x) \in M \text{ for a.e. }x \in \Sigma \}.
\]
Throughout, the manifolds $\Sigma$ and $M$, the embedding $M \to \R^d$, and all of the relevant geometric data $g, h, \ast, J$, are assumed to be smooth.

The associative Smith equation~\eqref{eq:Spec-Smith} for maps $u \colon (\Sigma^3, g, \ast) \to (M^7, h, J)$ reads
$$J \circ \Lambda^2du = \tfrac{1}{\sqrt{3}}|du|\,du \circ \ast.$$
In the case $u \in W^{1, 3}(\Sigma; M)$, the weak derivative $(du)_{x}$ is a linear map $T_{x}\Sigma \to T_{u(x)}M$ for almost every $x \in \Sigma$, and at such points both sides of the above equation make sense. Hence we say that a map $u \in W^{1, 3}(\Sigma; M)$ is an associative Smith map if the weak derivative $du$ satisfies the above equation almost everywhere on $\Sigma$.

As above, we let $E(u; A) = \int_A |du|^3\,\vol_\Sigma$ denote the $3$-energy of $u$ on a measurable set $A \subset \Sigma$.

Our first result is an $\ep$-regularity theorem for weak solutions of~\eqref{eq:Spec-Smith}, which gives $C^{1,\beta}$-regularity with \emph{a priori} estimates under a smallness assumption on the $3$-energy. Two of the main ingredients for the proof are the regularity of $n$-harmonic functions due to Uhlenbeck~\cite{U}, and a deep result from harmonic analysis due to Fefferman--Stein~\cite{FS}. Note that because of the degeneracy of~\eqref{eq:Spec-Smith}, in general we do not expect the solution to be better than $C^{1, \beta}$. However, on the set where $du$ is nonzero we do get smoothness. For convenience we state the $\ep$-regularity theorem in local terms and assume the domain is $B(2) \subset \R^{3}$. In this case we may identify the (weak) derivative $du$ of a map $u \in W^{1, 3}(B(2); \R^{d})$ with an element of $L^{3}(B(2); \R^{3\times d})$, which we denote by $Du$.

\textbf{Theorem~\ref{thm:ep-reg}.} ($\ep$-regularity.)
\emph{There exist $\ep_{0}, \beta > 0$ depending on $M, J$ and the embedding $M \to \R^{d}$ such that if $g$ is a Riemannian metric on $B(2)$ with $|g - g_{\text{euc}}|_{0; B(2)} + |Dg|_{0; B(2)} < \ep_0$ and $u \in W^{1, 3}(B(2); M)$ is an associative Smith map with respect to $g$ satisying $E(u; B(2)) < \ep_{0}$, then the following hold:
\begin{enumerate}[(a)]
\item The map $u$ belongs to $C^{1, \beta}(B(1); M)$ and the norm $|u|_{1, \beta; B(1)}$ can be estimated through $M, J$, the embedding $M \to \R^{d}$, and $E(u; B(2))$.
\item In addition, $u$ is smooth on the open set $\{x \in B(1)\ |\ Du(x) \neq 0\}$.
\end{enumerate}
}

On a number of occasions, especially in $\S$\ref{sec:zero-neck-length}, we need a more explicit gradient estimate than the one obtained in part (a) of Theorem~\ref{thm:ep-reg}. The following result, based on an adaptation of~\cite{DF}, provides such an estimate, again under a smallness assumption on the energy. As in Theorem~\ref{thm:ep-reg} we suppose that the domain $\Sigma$ is the ball $B(2) \subset \R^3$.

\textbf{Theorem~\ref{thm:MVI}.} (Mean value inequality.)
\emph{Suppose that the Ricci curvature of the domain metric $g$ is bounded, in the sense that
$$\left| \text{Ric}_{g(x)}(v,v) \right| \leq K \left| v \right|^2_{g(x)} \quad \text{ for all } x \in B(2), \, v \in \R^3.$$
Then there exist $C > 0$ and $\ep_1 > 0$ such that if $|g - g_{\text{euc}}|_{0; B(2)} + |Dg|_{0; B(2)} \leq \ep_1$ and $u \in W^{1,3}(B(2); M)$ is an associative Smith map with $E(u; B(2)) < \ep_1$, then
$$\sup_{B(\frac{1}{2})} |Du(x)|^3 \leq C E(u; B(2))$$
where both $C$ and $\ep_1$ depend only on $(M,J)$, the embedding $M \to \R^d$, and the constant $K$.
}

Covering the domain with small enough balls to which Theorem~\ref{thm:ep-reg} is applicable, we obtain the following interior regularity result without a smallness assumption on the energy. We remark that, as opposed to Theorem~\ref{thm:ep-reg}, the $C^{1, \beta}$-regularity given below is only qualitative as Theorem~\ref{thm:int-reg} does not provide an a priori estimate of the $C^{1, \beta}$-norm and thus leaves open the possibility of a sequence of Smith maps bubbling.

\textbf{Theorem~\ref{thm:int-reg}.} (Interior regularity.)
\emph{Suppose that $g$ is a smooth Riemannain metric on $B(2)$, and that $u$ is a $W^{1, 3}$-Smith map on $B(2)$ with respect to $g$. Then $u$ has H\"older continuous first derivatives on $B(1)$. Moreover, $u$ is $C^\infty$ on the open set $\{x \in B(1) \colon Du(x) \neq 0\}$.
}

As another corollary of Theorem~\ref{thm:ep-reg}, we get that isolated singularities of associative Smith maps are removable.

\textbf{Theorem~\ref{thm:rem-sing}.} (Removable singularity.)
\emph{Suppose that $g$ is a smooth Riemannian metric on $B(2)$, and that $u \in C^{1}_{\loc}(B(2)\setminus\{0\}; M)$ is a Smith map with respect to $g$, satisfying
\begin{equation*}
\int_{B(2)}|Du|^{3}dx < \infty.
\end{equation*}
Then in fact $u$ extends to a $C^{1}$-Smith map on all of $B(2)$.}

The next result can be viewed as a global version of Theorem~\ref{thm:MVI}. Specifically, as in~\cite[$\S$4.1]{MS}, a short argument using Theorem~\ref{thm:MVI} and the conformal invariance of the Smith equation~\eqref{eq:Spec-Smith} establishes the following.

\textbf{Proposition~\ref{prop:energy-gap}.} (Energy gap.)
\emph{There exists a constant $\ep_0 > 0$, depending only on $(M,J)$ and the embedding $M \to \R^d$, such that every $C^1$ associative Smith map $u \colon (S^3, g_{\round}) \to (M^7, h)$ with $E(u; S^3) < \ep_0$ is constant.
}

We then turn to sequences of associative Smith maps with uniformly bounded $3$-energy. The following theorem is central to the construction of the bubble tree limit. Specifically, from Theorem~\ref{thm:ep-reg} and Theorem~\ref{thm:rem-sing}, together with an argument using the weak-$*$ compactness of Radon measures which is by now standard, we deduce $C^{1}$-convergence away from a finite set of points to a $C^{1, \beta}$-limit which is again a Smith map.

\textbf{Proposition~\ref{prop:conv-mod-bubbling}.} (Compactness modulo bubbling.)
\emph{Let $(\Sigma^3, g)$ be a closed Riemannian $3$-manifold. Let $\Omega \subset \Sigma$ be an open set, and let $\{\Omega_n\}$ be a sequence of open sets that exhaust $\Omega$. Let $g_n$ be a Riemannian metric on $\Omega_n$ such that $g_n \to g$ smoothly on compact subsets of $\Omega$.}

\emph{Let $u \colon (\Omega_n, g_n) \to (M^7, h)$ be a sequence of $C^1$ associative Smith maps satisfying a uniform $3$-energy bound
$$E(u_n; \Omega_n) \leq E_0.$$
Then there exists a finite set $\mathcal{S} = \{x_1, \ldots, x_q\} \subset \Sigma$ and an associative Smith map $u \in C^1_{\loc}(\Omega; M)$ with $\int_\Omega |du|_g^3\,d\mu_g \leq E_0$ such that, after passing to a subsequence:
\begin{enumerate}[(a)]
\item We have $u_n \to u$ in $C^1_{\loc}(\Omega \setminus \mathcal{S})$.
\item We have
$$\left|du_n\right|^3_{g_n} d\mu_{g_n} \to \left| du \right|^3_g d\mu_g + \sum_{i=1}^q m_i \delta(x_i) \quad \text{ as Radon measures on $\Omega$},$$
where each $\delta(x_i)$ is a Dirac measure. Moreover, each $m_i \geq \frac{1}{2}\ep_0$, where $\ep_0$ is as in Theorem~\ref{thm:ep-reg}.
\item If $\| du_n \|_{p; \Omega_n} \leq C$ for some $p \in (3, \infty]$, then $\mathcal{S} = \varnothing$.
\end{enumerate}
}

\begin{rmk} \label{rmk:uv-initial-rmk}
One can construct associative submanifolds of the form $\Sigma^2 \times S^1$ in $Y^6 \times S^1$, where $Y^6$ is a Calabi-Yau $3$-fold. Because of this, one might think that, since holomorphic curves bubble at points, associative Smith maps should generically bubble along curves. The issue is that bubbling of holomorphic curves corresponds to concentration of $2$-energy, whereas bubbling of associative Smith maps corresponds to concentration of $3$-energy. This is all very carefully explained in \S\ref{sec:uv-stuff} using relations between different classes of Smith maps we establish in \S\ref{sec:smith-relations}.
\end{rmk}

In $\S$\ref{sec:bubble-tree}, we study the behavior of sequences of associative Smith maps with uniformly bounded $3$-energy near the bubble points $x_i \in \mathcal{S}$. Indeed, using Theorem~\ref{prop:conv-mod-bubbling}, the conformal invariance of the Smith equation, and the conformal invariance of the $3$-energy functional $E$, we will see that such a sequence gives rise to a bubble tree of associative Smith maps.

In outline, bubble trees arise in the following way. Beginning with a sequence $u_n \colon (\Sigma^3, g) \to (M^7, h)$ of associative Smith maps with $E(u_n; \Sigma) \leq E_0$, Theorem~\ref{prop:conv-mod-bubbling} shows that a subsequence of $u_n$ converges in $C^1_{\loc}$ to an associative Smith map $u_\infty \colon \Sigma \to M$ off of a finite set $\mathcal{S}$ of bubble points in $\Sigma$. 

If $\mathcal{S} \neq \varnothing$, then for each bubble point $x_i \in \mathcal{S}$, we may conformally rescale the sequence $u_n$ in such a way that a subsequence of the rescaled maps converges in $C^1_{\loc}$ to an associative Smith map $\widetilde{u}_{\infty, i} \colon S^3 \to M^7$ off of a finite set $\mathcal{S}_i \subset S^3 \setminus \{p^-\}$, where $p^-$ is the south pole. If any of the sets $\mathcal{S}_i \neq \varnothing$, then for each first-level bubble point $x_{ij} \in \mathcal{S}_i$, we may again conformally rescale $u_n$ so that a subsequence of the rescaled maps converge in $C^1_{\loc}$ to an associative Smith map $\widetilde{u}_{\infty, ij} \colon S^3 \to M$ off of a finite set $\mathcal{S}_{ij} \subset S^3 \setminus \{p^-\}$, and so on.

In $\S$\ref{sec:bubble-construction}, we use the energy gap of Proposition~\ref{prop:energy-gap} to show that this process does, in fact, terminate after a finite number of iterations. The result is a tree of associative Smith maps whose base vertex corresponds to the base map $u_\infty$, whose higher vertices correspond to the bubble maps $\widetilde{u}_{\infty, I}$, and whose edges correspond to the bubble points $x_I$, where here $I = (i_1, \ldots, i_k)$ is a multi-index.

Having constructed the bubble tree, we turn to the analogues of Properties (x) and (xi). For this, we need to assume that the $\G$-structure $\varphi$ on $M$ is \emph{closed}. In $\S$\ref{sec:no-energy-loss} and $\S$\ref{sec:zero-neck-length}, respectively, we show that both $3$-energy and homotopy are preserved in the ``bubble tree limit".

\textbf{Theorem~\ref{thm:no-energy-loss}.} (No energy loss.)
\emph{Suppose $d\varphi = 0$. If $u_n \colon \Sigma \to M$ is a sequence of associative Smith maps with uniformly bounded $3$-energy, then
$$\lim_{n \to \infty} E(u_n) = E(u_\infty) + \sum_I E(\widetilde{u}_{\infty, I}).$$
}

\textbf{Theorem~\ref{thm:zero-neck-length}.} (Zero neck length.)
\emph{Suppose $d\varphi = 0$. If $u_n \colon \Sigma \to M$ is a sequence of associative Smith maps with uniformly bounded $3$-energy, then its bubble tree has no necks. Therefore, for each multi-index $I$, we have
$$\widetilde{u}_{\infty, I}(x_{Ij}) = \widetilde{u}_{\infty, Ij}(p^-).$$
In particular, the set
$$u_\infty(\Sigma) \cup \bigcup_I \widetilde{u}_{\infty, I}(S^3)$$
is connected.
}

The precise meaning of ``no necks" is explained in both $\S$\ref{sec:bubble-overview} and $\S$\ref{sec:zero-neck-length}.

Our proofs of Theorems~\ref{thm:no-energy-loss} and~\ref{thm:zero-neck-length} both depend on a more detailed understanding of the $3$-energies of associative Smith maps on $3$-dimensional annuli, which we study in $\S$\ref{sec:energy-annuli}. In turn, this requires an energy gap result of Brian White~\cite{Wh} as well as the homologically energy-minimizing property of Smith maps (see Smith's result (iv) above), the latter of which requires that the $\G$-structure be closed.

\begin{rmk} Here is a simple explicit example of bubbling, adapted from \cite[Example 1.1]{Li}. Let $\iota \colon (S^3, g_{\text{round}}) \to (M^7, h)$ be an isometric associative immersion of a round $3$-sphere into a $\text{G}_2$-manifold $(M^7, h)$. For example, we could take $(M^7, h)$ to be the spinor bundle $\slash\!\!\!\mathcal{S}(S^3)$ equipped with the Bryant--Salamon metric~\cite{BS}, and take $\iota$ to be the inclusion map of the zero section. Let
$$ \sigma \colon (S^3 \setminus \{p^-\}, g_\text{round}) \to (\mathbb{R}^3, g_{\text{euc}}) $$
denote stereographic projection. Then the maps $u_n \colon (\mathbb{R}^3, g_{\text{euc}}) \to (M^7, h)$ given by $u_n(x) := \iota(\sigma^{-1}(nx))$ form a sequence of associative Smith maps with bounded $3$-energy. Moreover, $u_n \to u_\infty$ in $C^1_{\text{loc}}(\mathbb{R}^3 \setminus \{0\})$, where $u_\infty(x) \equiv \iota(p^-)$ is a constant map, while
$$ |du_n|^3\,d\mu \to m\delta({\{0\}}) $$
as Radon measures on $\mathbb{R}^3$ for some $m > 0$.
\end{rmk}

\subsection{Organization and notation}

The paper is organized as follows. In $\S$\ref{sec:linear}, we present a comprehensive treatment of the linear algebra of Smith maps, including vector cross products and calibrations. In $\S$\ref{sec:smith-maps-manifolds}, we discuss Smith maps between manifolds, including the energy identity and the relation to $n$-harmonic maps. In $\S$\ref{sec:int-reg-cpt}, we establish many of the analytical results, including interior regularity, the mean value inequality, removable singularities, compactness modulo bubbling, and the energy gap. In $\S$\ref{sec:bubble-tree}, we give a detailed explanation of bubbling and the bubble tree, and prove that there is no energy loss and zero neck length when the $\G$-structure is closed. Finally, Appendix~\ref{sec:harmonic-analysis-appendix} collects some needed results from harmonic analysis, and Appendix~\ref{sec:regularity-appendix} presents the proofs of two of the results from $\S$\ref{sec:int-reg-cpt}.

We employ the following notation throughout the paper:
\begin{itemize} \setlength\itemsep{-1mm}
\item We use $g_{\euc}$ to denote the Euclidean metric on $\R^n$ for any $n$.
\item On a Riemannian manifold we use $|\cdot|$ for the pointwise norm on tensors induced from the metric. We use $\inj_M$ to denote the injectivity radius of $M$ and $\diam(A)$ to denote the diameter of a set $A$. In most cases the metric is clear from the context, but it is indicated by subscripts when necessary.
\item When $A$ is a set, we use the following notation for norms:
\begin{align*}
& \text{$| \cdot |_{k; A}$ on $C^k(A)$}, \qquad \, \, \text{$| \cdot |_{k, \alpha; A}$ on $C^{k,\alpha}(A)$}, \\
& \text{$\| \cdot \|_{p; A}$ on $L^p (A)$}, \qquad \text{$\| \cdot \|_{1, p; A}$ on $W^{1,p} (A)$}.
\end{align*}
Moreover, $[ \cdot ]_{\alpha; A}$ denotes the H\"older semi-norm with exponent $\alpha$ on $A$.
\item In a Riemannian manifold, we use $B(x; r)$ to denote the open ball at $x$ of radius $r > 0$. When local coordinates are chosen, we write $B(r) := B(0;r)$ for brevity.
\item When $u : \Sigma \to M$ is a $C^1$ map, then $du \in \Gamma( T^* \Sigma \otimes u^* TM )$ denotes the differential, so $(du)_x : T_x \Sigma \to T_{u(x)} M$ for all $x \in \Sigma$. By the embedding $M \to \R^{d}$, we may also view $(du)_x$ as a linear map $T_{x}\Sigma \to \R^{d}$ with image contained in $T_{u(x)}M$. The latter viewpoint extends to maps in $W^{1, 3}(\Sigma; M)$, but only holds for almost every $x \in \Sigma$.
\item If $\Sigma\subset\R^{3}$ and $u\in W^{1, 3}(\Sigma; \R^{d})$, the derivative $du$ may be identified with a map in $L^{3}(\Sigma; \R^{3 \times d})$, which we denote by $Du$.
\item Given a function $F : A \to \R$, we use $(f)_A$ to denote its average value over $A$, namely
\begin{equation*}
(f)_A = \frac{1}{\int_A \vol} \int_A F \, \vol = \fint_A F \, \vol.
\end{equation*}
\end{itemize}

Other notation is introduced and defined when it is needed.

\textbf{Acknowledgements.} The authors gratefully acknowledge useful conversations with Ali Aleyasin, Gavin Ball, Benoit Charbonneau, Aleksander Doan, Andriy Haydys, Jason Lotay, Andr\'{e} Neves, Rick Schoen, Kyler Siegel, and Thomas Walpuski. The second author is supported by an NSERC Discovery Grant. The third author also thanks McKenzie Wang for his support and encouragement.

\textbf{Remark.} After this work was completed, the authors learned that Mou--Wang~\cite{MW} had studied a conformally invariant $n$-harmonic type system which the Smith maps satisfy when $n = 3$ or $4$, and obtained a no energy loss result for bubble trees of solutions using methods quite different from ours. Indeed our proof of Theorem~\ref{thm:no-energy-loss} relies on an isoperimetric-type estimate coming from the geometric properties of Smith maps (see Section~\ref{sec:energy-annuli}), whereas in~\cite{MW} the authors exploit again the compensation phenomenon mentioned in Section~\ref{sec:results-methods} (see~\cite[Lemma 3.3 and pages 363--364]{MW}). We refer the reader to the introduction of our Section~\ref{sec:int-reg-sub} for comments on other similarities and contrasts between our work and~\cite{MW}.

\section{The linear algebra of Smith maps} \label{sec:linear}

Vector cross products were introduced by Brown--Gray~\cite{BG} and were further studied much later by Lee--Leung~\cite{LL}. The notion of a vector cross product preserving map was introduced by Gray in~\cite{Gr-CR}. We refer to such maps as \textit{Gray maps}. A generalized notion, which can reasonably be called conformally vector cross product preserving, was introduced by Smith in~\cite{Sm}, where they were called \textit{multiholomorphic maps}. We refer to such maps as \textit{Smith maps}. Calibrations were introduced in the seminal paper of Harvey--Lawson~\cite{HL}. In this section we discuss the linear algebraic aspects of vector cross products, calibrations, and Smith maps.

A few of the results in $\S$\ref{sec:VCP} and in $\S$\ref{sec:VCP-maps} are at least implicit in~\cite{Gr-CR} for the case of Gray maps. We adapt the proofs to the more general case of Smith maps and flesh out several details that are missing from~\cite{Gr-CR}, which is in any case somewhat difficult to access. We also present many more results in $\S$\ref{sec:VCP-maps} that do not appear to be in the literature.

The main results in this section are that any Smith map is weakly conformal, proved in Theorem~\ref{thm:smith-conformal}, the relations between Smith maps and calibrations in $\S$\ref{sec:smith-calib}, and the \textit{generalized calibration inequality} established in Theorem~\ref{thm:generalized-calib}. These results are also implicit in the unpublished preprint~\cite{Sm}, although their precise statements are somewhat obscured. In the present paper we significantly clarify both of these results, first by distilling them to their minimal hypotheses (which is entirely linear algebraic), and second by parcelling out the various required components into separate lemmas and propositions. We hope the careful exposition in this section will be useful to a wider audience.

\subsection{Preliminaries} \label{sec:prelim}

We use the term \textit{Euclidean space} to denote a finite-dimensional real vector space equipped with a positive definite inner product. Here we collect several preliminary results on linear maps between Euclidean spaces and the exterior powers of such maps. In particular we require multiple versions of \textit{Hadamard's inequality}. We collect them here with proofs for completeness.

Let $(V, \langle \cdot, \rangle)$ be an $n$-dimensional Euclidean space and let $(W, \langle \cdot, \rangle)$ be an $m$-dimensional Euclidean space. Let $A : V \to W$ be a linear map. Let $A^* : W \to V$ be the adjoint map. We define the \textit{matrix norm}, also called the Frobenius norm or Hilbert--Schmidt norm, of $A$ by $|A|^2 = \tr (A^* A)$. With respect to orthonormal bases $\{ v_1, \ldots, v_n \}$ of $V \cong \R^n$ and $\{ w_1, \ldots, w_m \}$ of $W \cong \R^m$, this norm is given by
\begin{equation} \label{eq:matrix-norm}
\begin{aligned}
|A|^2 & = \tr (A^* A) = \sum_{l=1}^n \sum_{i=1}^m A_{\, l}^i A_{\, l}^i \\
& = \sum_{l=1}^n \sum_{i,j=1}^m \langle A_{\, l}^i w_i, A_{\, l}^j w_j \rangle = \sum_{l=1}^n \langle Av_l, Av_l \rangle = \sum_{l=1}^n |Av_l|^2.
\end{aligned}
\end{equation}
We need to use both the $\tr (A^* A)$ and the $\sum_{l=1}^n \langle Av_l, Av_l \rangle$ expressions for $|A|^2$.

\begin{lemma} \label{lemma:conf-isom}
The map $A : V \to W$ is called a \textit{conformal injection} if $\langle A v_1, A v_2 \rangle = \lambda^2 \langle v_1, v_2 \rangle$ for some $\lambda > 0$. This is equivalent to $A^* A = \lambda^2 I$ and also equivalent to $A = \lambda \hat A$ where $\hat A$ is an isometric injection. That is, $\langle \hat A v_1, \hat A v_2 \rangle = \langle v_1, v_2 \rangle$. When $\lambda = 1$ then $A = \hat A$ is an isometric injection. Moreover, we necessarily have $\lambda^2 = \frac{1}{n} |A|^2$ where $n=\dim V$. 
\end{lemma}
\begin{proof}
From $\langle A^* A v_1, v_2 \rangle = \langle Av_1, Av_2 \rangle$, we have that $\langle Av_1, Av_2 \rangle = \lambda^2 \langle v_1, v_2 \rangle$ if and only if $A^* A = \lambda^2 I$. Let $\hat A = \lambda^{-1} A$. Then $A^* A = \lambda^2 I$ is equivalent to $\hat A^* \hat A = I$. The last statement follows from~\eqref{eq:matrix-norm}, since $|A|^2 = \tr(A^* A) = \tr(\lambda^2 I) = n \lambda^2$.
\end{proof}

\begin{cor} \label{cor:conf-isom-comp}
Let $n = \dim V$. Suppose $A : V \to W$ is a conformal injection, and let $B : V \to V$ be a conformal isomorphism. Then $|AB| = \frac{1}{\sqrt{n}} |A| \, |B|$.
\end{cor}
\begin{proof}
By Lemma~\ref{lemma:conf-isom} we have $A^* A = \lambda^2 I$ and $B^* B = \mu^2 I$, where $n \lambda^2 = |A|^2$ and $n \mu^2 = |B|^2$. Moreover since $B$ is invertible we have $B^* B = B B^*$. Thus we have
\begin{align*}
|AB|^2 & = \tr \big( (AB)^* (AB) \big) = \tr (B^* A^* A B) = \tr (A^* A B B^*) \\
& = \tr (\lambda^2 \mu^2 I) = n \lambda^2 \mu^2 = \tfrac{1}{n} |A|^2 |B|^2
\end{align*}
as claimed.
\end{proof}

For $1 \leq r \leq \dim V$, let $\Lambda^r A : \Lambda^r V \to \Lambda^r W$ be the $r$th exterior power of $A$, defined by
\begin{equation*}
(\Lambda^r A)(v_1 \wedge \cdots \wedge v_r) = (Av_1) \wedge \cdots \wedge (Av_r)
\end{equation*}
on decomposable elements and extended linearly to all of $\Lambda^r V$. It is immediate that $(\Lambda^r A^*) = (\Lambda^r A)^*$ and that $\Lambda^r (A_1 A_2) = (\Lambda^r A_1)(\Lambda^r A_2)$.

\begin{lemma} \label{lemma:eigen}
Let $B : V \to V$ be a positive self-adjoint linear map. This means that $\langle B v, w \rangle = \langle v, B w \rangle$ and $\langle B v, v \rangle \geq 0$ with equality if and only if $v = 0$. Let $0 < r < n = \dim V$.

Suppose that $\Lambda^r B : \Lambda^r V \to \Lambda^r V$ is an isometry on decomposable elements. This means that
\begin{equation} \label{eq:eigen-hyp}
\langle Bv_1 \wedge \cdots \wedge Bv_r, v_1 \wedge \cdots \wedge v_r \rangle = |v_1 \wedge \cdots \wedge v_r|^2 \qquad \text{ for all $v_1, \ldots, v_r \in V$.}
\end{equation}
Then $B$ is is the identity map on $V$.
\end{lemma}
\begin{proof}
By the spectral theorem, there is an orthonormal basis $\{ v_1, \ldots, v_n \}$ of $V$ consisting of eigenvectors of $B$, where the eigenvalues $\mu_1, \ldots, \mu_n$ are \textit{strictly positive real numbers}. Let $1 \leq i_1 < \cdots < i_r \leq n$. From $B v_j = \mu_j v_j$, we have that $(Bv_{i_1}) \wedge \cdots \wedge (Bv_{i_r}) = (\mu_{i_1} \cdots \mu_{i_r}) v_{i_1} \wedge \cdots \wedge v_{i_r}$. By hypothesis~\eqref{eq:eigen-hyp} it follows that $\mu_{i_1} \cdots \mu_{i_r} = 1$ for all such strictly increasing multi-indices. The positivity of the $\mu_j$'s now implies that $\mu_j = 1$ for all $j$, so $Bv = v$ for all $v \in V$. Note that we definitely need $0 < r < n$ to obtain the conclusion.
\end{proof}

\begin{lemma}[Hadamard's inequality on $\Lambda^r A$] \label{lemma:hadamard}
Let $1 < r \leq n = \dim V$. Let $A : V \to W$ be nonzero. Then we have
\begin{equation} \label{eq:hadamardonA}
|\Lambda^r A|^2 \leq n^{-r} \binom{n}{r} |A|^{2r},
\end{equation}
with equality if and only if $A$ is a conformal injection in the sense of Lemma~\ref{lemma:conf-isom}.
\end{lemma}
\begin{proof}
Let $B = A^* A$, which is a nonnegative self-adjoint linear map. By the spectral theorem, there is an orthonormal basis $\{ v_1, \ldots, v_n \}$ of $V$ consisting of eigenvectors of $B$, where the eigenvalues $\mu_1, \ldots, \mu_n$ are nonnegative. Then $\{ v_{i_1} \wedge \cdots \wedge v_{i_r} : 1 \leq i_1 < \cdots < i_r \leq n \}$ is an orthonormal basis of $\Lambda^r V$ and $(\Lambda^r B)(v_{i_1} \wedge \cdots \wedge v_{i_r}) = (\mu_{i_1} \cdots \mu_{i_r}) v_{i_1} \wedge \cdots \wedge v_{i_r}$. Applying~\eqref{eq:matrix-norm} to $\Lambda^r A$ gives
\begin{align*}
| \Lambda^r A |^2 & = \tr\big( (\Lambda^r A)^* (\Lambda^r A) \big) = \tr \big(\Lambda^r (A^*A) \big) = \tr (\Lambda^r B) \\
& = \sum_{1 \leq i_1 < \cdots < i_r \leq n} (\mu_{i_1} \cdots \mu_{ir}).
\end{align*}
McLaurin's inequality, which is a generalization of the arithmetic-geometric mean inequality to other symmetric polynomials, says that if $\mu_1, \ldots, \mu_n$ are nonnegative, then
\begin{equation*}
\Bigg( \frac{1}{\binom{n}{r}} \sum_{1 \leq i_1 < \cdots < i_r \leq n} (\mu_{i_1} \cdots \mu_{ir}) \Bigg)^{\frac{1}{r}} \leq \Bigg( \frac{1}{n} \sum_{j=1}^n \mu_j \Bigg) \qquad \text{with equality iff $\mu_1 = \cdots = \mu_n$.}
\end{equation*}
Using this inequality and $\tr B = \tr (A^* A) = |A|^2$, we get that
\begin{equation*}
| \Lambda^r A |^2 \leq \binom{n}{r} \Bigg( \frac{1}{n} \sum_{j=1}^n \mu_j \Bigg)^r = \binom{n}{r} \Big( \frac{1}{n} \tr B \Big)^r = \binom{n}{r} n^{-r} |A|^{2r}
\end{equation*}
with equality if and only if $B = A^* A = \mu I$ for some $\mu \geq 0$, where $n \mu = |A|^2$. Since $A$ is nonzero, we deduce that $\mu > 0$. The result now follows by Lemma~\ref{lemma:conf-isom}.
\end{proof}

\begin{cor} \label{cor:hadamard-n}
Let $A : V \to W$ be nonzero, and let $n = \dim V$. Then we have
\begin{equation*}
|\Lambda^n A| \leq \frac{1}{(\sqrt{n})^n} |A|^{n},
\end{equation*}
with equality if and only if $A$ is a conformal injection in the sense of Lemma~\ref{lemma:conf-isom}.
\end{cor}
\begin{proof}
This is obtained by taking square roots of both sides of~\eqref{eq:hadamardonA} in the special case $r=n$.
\end{proof}

\begin{cor} \label{cor:hadamard-vectors}
Let $n = \dim V$ and let $v_1, \ldots, v_n$ be nonzero vectors in $V$. Then we have
\begin{equation} \label{eq:hadamard-vectors}
| v_1 \wedge \cdots \wedge v_n | \leq |v_1| \cdots |v_n| \qquad \text{with equality iff $v_1, \ldots, v_n$ are orthogonal}.
\end{equation}
\end{cor}
\begin{proof}
By dividing both sides by $|v_1| \cdots |v_n|$, we can assume that $v_1, \ldots, v_n$ are all \textit{unit vectors}, and we need to prove that
\begin{equation*}
| v_1 \wedge \cdots \wedge v_n |^2 \leq 1 \qquad \text{with equality iff $v_1, \ldots, v_n$ are orthonormal}.
\end{equation*}
Let $\{ u_1, \ldots, u_n \}$ be an orthonormal basis of $V$ and define a linear map $A : V \to V$ by $Au_i = v_i$. By~\eqref{eq:matrix-norm} we have
\begin{equation*}
|A|^2 = \sum_{i=1}^n | Au_i |^2 = \sum_{i=1}^n | v_i |^2 = n.
\end{equation*}
 Then we have
\begin{equation*}
|v_1 \wedge \cdots \wedge v_n|^2 = |(Au_1) \wedge \cdots (Au_n)|^2 = |(\Lambda^n A)(u_1 \wedge \cdots \wedge u_n)|^2.
\end{equation*}
Since the single element $u_1 \wedge \cdots \wedge u_n$ is an orthonormal basis for $\Lambda^n V$, equation~\eqref{eq:matrix-norm} applied to $\Lambda^n A$ gives $|(\Lambda^n A)(u_1 \wedge \cdots \wedge u_n)|^2 = |\Lambda^n A|^2$. Using this and Corollary~\ref{cor:hadamard-n} we have
\begin{equation*}
|v_1 \wedge \cdots \wedge v_n|^2 = |\Lambda^n A|^2 \leq \frac{1}{n^n} |A|^{2n} = \frac{1}{n^n} n^n = 1,
\end{equation*}
with equality if and only if $A$ is a conformal injection. Note that by Lemma~\ref{lemma:conf-isom}, since $|A|^2 = n$, the conformal factor must be $1$ and thus $A$ is an isometric injection. But by the definition of $A$ this means equality occurs if and only if $v_1, \ldots, v_n$ are orthonormal.
\end{proof}

\subsection{Vector cross products} \label{sec:VCP}

Let $(V, \langle \cdot, \cdot \rangle)$ be an $n$-dimensional Euclidean space.

\begin{defn} \label{defn:VCP}
Let $1 \leq k \leq n-1$. A $k$-fold \textit{vector cross product} $P$ on $V$ is an element $P \in \Lambda^k (V^*) \otimes V$ that satisfies the following two properties:
\begin{equation} \label{eq:Porthogonal}
P(v_1 \wedge \cdots \wedge v_k) \quad \text{is orthogonal to $v_1, \ldots, v_k$},
\end{equation}
and
\begin{equation} \label{eq:Pmetric}
| P(v_1 \wedge \cdots \wedge v_k) |^2 = |v_1 \wedge \cdots \wedge v_k|^2.
\end{equation}
Note that equation~\eqref{eq:Porthogonal} is equivalent to the statement that the covariant $(k+1)$-tensor $\alpha_P$ defined by
\begin{equation} \label{eq:alphaskew}
\alpha_P(v_1, \ldots, v_k, v_{k+1}) = \langle P(v_1 \wedge \cdots \wedge v_k), v_{k+1} \rangle \qquad \text{is totally skew-symmetric}.
\end{equation}
The $(k+1)$-form $\alpha_P$ is called the \textit{calibration form} associated to $P$. It is discussed in $\S$\ref{sec:calibrations}.
\end{defn}

\begin{ex} \label{ex:complex}
Consider the case $k=1$. Conditions~\eqref{eq:Porthogonal} and~\eqref{eq:Pmetric} say that $P : V \to V$ satisfies $\langle Pv, v \rangle = 0$ and $|Pv|^2 = |v|^2$. Polarizing these two equations gives $\langle Pv, w \rangle = - \langle v, Pw \rangle$, so $P$ is skew-adjoint, and $\langle Pv, Pw \rangle = \langle v, w \rangle$, so $P$ is an isometry. Thus we have
\begin{equation*}
\langle P^2 v, w \rangle = - \langle Pv, Pw \rangle = - \langle v, w \rangle. 
\end{equation*}
Since this holds for all $v,w$, we have $P^2 = - I$ and thus $P$ is an \textit{orthogonal complex structure} on $V$. In particular then $n = \dim V$ must be even.
\end{ex}

\begin{ex} \label{ex:hodgestar}
Consider the case $k=n-1$. Let $u_1, \ldots, u_n$ be an orthonormal basis for $V$. Conditions~\eqref{eq:Porthogonal} and~\eqref{eq:Pmetric} say that $P : \Lambda^{n-1} V \to V$ satisfies $\langle P(u_1 \wedge \cdots \wedge u_{n-1}), u_i \rangle = 0$ for $1 \leq i \leq n-1$, so $P(u_1 \wedge \cdots \wedge u_{n-1})$ must be a multiple of $u_n$, and $|P(u_1 \wedge \cdots \wedge u_{n-1})|^2 = |u_1 \wedge \cdots \wedge u_{n-1}|^2 = 1$, so $P(u_1 \wedge \cdots \wedge u_{n-1}) = \pm u_n$. Let $\vol = \pm u_1 \wedge \cdots \wedge u_n$ be the orientation for $V$ induced by the ordered orthonormal basis $\{ u_1, \ldots, u_{n-1}, \pm u_n \}$. Then $P = \ast$ is the Hodge star operator on $\Lambda^{n-1} V$ corresponding to the inner product and this orientation. The Hodge star is an isometry.
\end{ex}

In general, condition~\eqref{eq:Pmetric} says that $P : \Lambda^k (V) \to V$ is length-preserving on the \textit{decomposable} elements of $\Lambda^k (V)$. When $k=1$ or $k=n-1$, any $k$-vector is decomposable, so the vector cross $P$ is an honest isometry in the cases of Examples~\ref{ex:complex} and~\ref{ex:hodgestar}.

The fundamental identities~\eqref{eq:Porthogonal} and~\eqref{eq:Pmetric} relating a vector cross product to the inner product are very strong constraints. In fact, there are only four possible types of vector cross products, the two we have already discussed and two \textit{exceptional} types. This is the \textit{Brown--Gray classification}~\cite{BG} given in Table~\ref{table:VCP}. The two exceptional vector cross products are the most interesting, as they are related to \textit{associative submanifolds of $\G$-manifolds} and to \textit{Cayley submanifolds of $\Spin{7}$-manifolds}, respectively.

\begin{table}[H]
\begin{center}
\begin{tabular}{|c||c|c|c|c||c|l|} \hline
Type & $n$ & $k$ & $P$ & $\alpha$ & Name of structure \\ \hline
I & $n$ & $n-1$ & $\ast$ & $\vol$ & Orientation \\
II & $2m$ & $1$ & $J$ & $\omega$ & Orthogonal complex structure \\
III & $7$ & $2$ & $\times$ & $\varphi$ & $\G$-structure \\
IV & $8$ & $3$ & $P$ & $\Phi$ & $\Spin{7}$-structure \\ \hline
\end{tabular}
\end{center}
\caption{The Brown--Gray classification of vector cross products.} \label{table:VCP}
\end{table}

Despite this simple classification, almost all the properties of vector cross products and of Smith maps that we establish in the rest of this section and in $\S$\ref{sec:VCP-maps} are consequences of the defining properties~\eqref{eq:Porthogonal} and~\eqref{eq:Pmetric}, so the statements and proofs are identical for all four types of vector cross products.

\begin{prop} \label{prop:VCP-identity}
Let $P$ be a $k$-fold vector cross product on $(V, \langle \cdot, \cdot \rangle)$. Let $u_1, \ldots, u_{k-1}$ be \textit{linearly independent} vectors in $V$ and let $w \in V$ be arbitrary. Let $U$ be the $(k-1)$-dimensional subspace of $V$ spanned by $u_1, \ldots, u_{k-1}$. Then the following identity holds:
\begin{equation} \label{eq:VCP-identity}
P \big(u_1 \wedge \cdots \wedge u_{k-1} \wedge P(u_1 \wedge \cdots \wedge u_{k-1} \wedge w) \big) = - |u_1 \wedge \cdots \wedge u_{k-1}|^2 \pi_{U^\perp} w,
\end{equation}
where $\pi_{U^{\perp}}$ is orthogonal projection onto the subspace $U^{\perp}$.
\end{prop}
\begin{proof}
We first polarize equation~\eqref{eq:Pmetric} in $v_k$ to obtain
\begin{equation} \label{eq:Pmetric-polarized}
\langle P(u_1 \wedge \cdots \wedge u_{k-1} \wedge v), P(u_1 \wedge \cdots \wedge u_{k-1} \wedge w) \rangle = \langle u_1 \wedge \cdots \wedge u_{k-1} \wedge v, u_1 \wedge \cdots \wedge u_{k-1} \wedge w \rangle.
\end{equation}
Let $U = \spa \{u_1, \ldots, u_{k-1} \}$. Since $w = \pi_U w + \pi_{U^{\perp}} w$ and $u_1 \wedge \cdots \wedge u_{k-1} \wedge (\pi_U w) = 0$, we can replace $w$ with $\pi_{U^{\perp}} w$ in the right hand side of equation~\eqref{eq:Pmetric-polarized}. Thus the right hand side is
\begin{align} \nonumber
\langle u_1 \wedge \cdots \wedge u_{k-1} \wedge v, u_1 \wedge \cdots \wedge u_{k-1} \wedge (\pi_{U^{\perp}} w) \rangle & = \det \begin{pmatrix} \langle u_i, u_j \rangle & \langle u_i, \pi_{U^{\perp}} w \rangle \\ \langle v, u_i \rangle & \langle v, \pi_{U^{\perp}} w \rangle \end{pmatrix} \\ \nonumber
& = \det \begin{pmatrix} \langle u_i, u_j \rangle & 0 \\ \langle v, u_i \rangle & \langle v, \pi_{U^{\perp}} w \rangle \end{pmatrix} \\ \label{eq:VCP-identity-RHS}
& = |u_1 \wedge \cdots \wedge u_{k-1}|^2 \langle \pi_{U^{\perp}} w, v \rangle.
\end{align}
Using the definition and skew-symmetry of the calibration form $\alpha_P$ from~\eqref{eq:alphaskew} the left hand side of~\eqref{eq:Pmetric-polarized} becomes
\begin{align} \nonumber
\alpha_P \big( u_1, \ldots, u_{k-1}, v, P(u_1 \wedge \cdots \wedge u_{k-1} \wedge w) \big) & = - \alpha_P \big( u_1, \ldots, u_{k-1}, P(u_1 \wedge \cdots \wedge u_{k-1} \wedge w), v \big) \\ \label{eq:VCP-identity-LHS}
& = - \langle P\big( u_1 \wedge \cdots \wedge u_{k-1} \wedge P(u_1 \wedge \cdots \wedge u_{k-1} \wedge w) \big), v \rangle.
\end{align}
Since~\eqref{eq:VCP-identity-LHS} and~\eqref{eq:VCP-identity-RHS} agree for all $v \in V$, the identity~\eqref{eq:VCP-identity} holds.
\end{proof}

\begin{rmk} \label{rmk:k=1}In the particular case when $k=1$, Proposition~\ref{prop:VCP-identity} simply says $P^2 (u) = -u$ for all $u \in V$, which we already knew from the discussion in Example~\ref{ex:complex}. But for $k>1$ it is nontrivial.
\end{rmk}

Proposition~\ref{prop:VCP-identity} has two important corollaries.

\begin{cor} \label{cor:VCP-surjective}
Let $v, u_1 \in V$ be orthogonal vectors. Then there exists a decomposable element of $\Lambda^k V$ of the form $u_1 \wedge \cdots \wedge u_k$ such that $v = P(u_1 \wedge \cdots \wedge u_k)$.
\end{cor}
\begin{proof}[Proof of Corollary~\ref{cor:VCP-surjective}]
By the linearity of $P$, we can without loss of generality assume that $v$ and $u_1$ both have unit length. Choose orthonormal vectors $u_2, \ldots, u_{k-1}$ in $V$ that are orthogonal to both $v$ and $u_1$. We can always do this because $(k-2)+2 = k < n = \dim V$. Let $U = \operatorname{span}\{ u_1, \ldots, u_{k-1} \}$. By construction we have $\pi_{U^{\perp}} v = v$. Thus the fundamental identity~\eqref{eq:VCP-identity} gives
\begin{equation*}
P \big(u_1 \wedge \cdots \wedge u_{k-1} \wedge P(u_1 \wedge \cdots \wedge u_{k-1} \wedge v) \big) = - v.
\end{equation*}
Taking $u_k = - P(u_1 \wedge \cdots \wedge u_{k-1} \wedge v)$ completes the proof.
\end{proof}

\begin{rmk} \label{rmk:VCP-surjective}
Corollary~\ref{cor:VCP-surjective} says that not only is the linear map $P : \Lambda^k V \to V$ always surjective, but that we can in fact always choose a preimage of $v$ that is of the decomposable form $u_1 \wedge \cdots \wedge u_k$, where $u_1$ is any nonzero vector orthogonal to $v$.
\end{rmk}

\begin{cor} \label{cor:VCP-identity}
Consider the same hypotheses as in Proposition~\ref{prop:VCP-identity}. Then the following identity holds:
\begin{equation} \label{eq:VCP-identity2}
\begin{aligned}
& P \Big( u_1 \wedge \cdots \wedge u_{k-1} \wedge P \big(u_1 \wedge \cdots \wedge u_{k-1} \wedge P(u_1 \wedge \cdots \wedge u_{k-1} \wedge w) \big) \Big) \\
& \qquad = - |u_1 \wedge \cdots \wedge u_{k-1}|^2 P(u_1 \wedge \cdots \wedge u_{k-1} \wedge w).
\end{aligned}
\end{equation}
\end{cor}
\begin{proof}
Let $y = \pi_{U^{\perp}} w$. Then $v - y = \pi_U w \in U$. Thus, by multlinearity and the property~\eqref{eq:Porthogonal} of vector cross products, we have
\begin{equation} \label{eq:VCP-identity-temp}
P(u_1 \wedge \cdots \wedge u_{k-1} \wedge y) = P(u_1 \wedge \cdots \wedge u_{k-1} \wedge w).
\end{equation}
Take the wedge product of both sides of~\eqref{eq:VCP-identity} with $u_1 \wedge \cdots \wedge u_{k-1}$ and then apply the vector cross product $P$ to both sides. The observation~\eqref{eq:VCP-identity-temp} therefore establishes~\eqref{eq:VCP-identity2}.
\end{proof}

\subsection{Calibration forms} \label{sec:calibrations}

Before we define Smith maps and Gray maps we review the closely related but somewhat more general notion of a \textit{calibration} on $V$. This is used crucially in the generalized calibration inequality of Theorem~\ref{thm:generalized-calib}. The theory of calibrations was introduced in~\cite{HL} by Harvey--Lawson. We emphasize that in this section we are only considering the linear algebraic (pointwise) aspects of the theory.

\begin{defn} \label{defn:calibration}
Let $\alpha$ be a $(k+1)$-form on $V$, where $k+1 \leq n = \dim V$. We say that $\alpha$ is a \textit{calibration} if
\begin{equation} \label{eq:comassone}
\alpha(v_1, \ldots v_{k+1}) \leq |v_1 \wedge \cdots \wedge v_{k+1}| \quad \text{for all $v_1, \ldots, v_{k+1} \in V$.}
\end{equation}
We say that the $(k+1)$-form $\alpha$ has \textit{comass one}.
\end{defn}

\begin{lemma} \label{lemma:calibration-equivalent}
The inequality~\eqref{eq:comassone} is equivalent to the inequality
\begin{equation} \label{eq:comassone-alt}
\alpha(u_1, \ldots u_{k+1}) \leq 1 \quad \text{for all \textit{orthonormal vectors} $u_1, \ldots, u_{k+1} \in V$.}
\end{equation}
\end{lemma}
\begin{proof}
Both sides of~\eqref{eq:comassone} vanish if $\{ v_1, \ldots, v_{k+1} \}$ is linearly dependent so we may assume it is linearly independent. Moreover, recall that $\Lambda^k V^* \cong (\Lambda^k V)^*$. That is, $\alpha(v_1, \ldots, v_{k+1}) = \alpha(v_1 \wedge \cdots \wedge v_{k+1})$. Hence, both sides of~\eqref{eq:comassone} depend only on the \textit{oriented} $(k+1)$-plane $u_1 \wedge \cdots \wedge u_{k+1} \in \Lambda^k V$ where $\{ u_1, \ldots, u_{k+1} \}$ is an ordered orthonormal basis for $\spa \{ v_1, \ldots, v_{k+1} \}$ inducing the same orientation. That is, $u_1 \wedge \cdots \wedge u_k = t v_1 \wedge \cdots \wedge v_{k+1}$ for some $t > 0$. Thus we have
\begin{align*}
\alpha(u_1, \ldots, u_{k+1}) = t \alpha(v_1, \ldots, v_{k+1}) & \leq t |v_1 \wedge \cdots \wedge v_{k+1}| \\
& = |u_1 \wedge \cdots \wedge u_{k+1}| = 1,
\end{align*}
which is what we wanted to show.
\end{proof}

\begin{defn} \label{defn:calibrated-plane}
A $(k+1)$-dimensional \textit{oriented} subspace $L$ of $V$ is called \textit{calibrated} with respect to $\alpha$ if the maximum in~\eqref{eq:comassone-alt} is achieved on $L$. That is, if $\alpha(u_1, \ldots, u_{k+1}) = 1$ for any \textit{oriented} orthonormal basis $\{ u_1, \ldots, u_{k+1} \}$ of $L$. Equivalently, if $L$ is an oriented $k$-dimensional subspace of $V$ with volume form $\vol_L$, then $L$ is calibrated with respect to $\alpha$ if and only if $\alpha|_L = \vol_L$.
\end{defn}

The set $\{ u_1 \wedge \cdots \wedge u_k \in \Lambda^k V : u_1, \ldots, u_k \text{ are orthonormal} \}$ of unit-length decomposable $k$-vectors in $V$ is compact, as it is precisely the Grassmanian $G^+(k, V)$ of oriented $k$-planes in $V$. Hence, any nonzero $(k+1)$-form can be suitably rescaled to satisfy the comass one condition~\eqref{eq:comassone} to obtain a calibration. However, not all calibrations admit a rich class of calibrated $k$-planes. One of the most important classes of such calibrations arise from vector cross products. In fact, any vector cross product $P$ on $V$ induces a calibration $\alpha_P$, originally introduced in~\eqref{eq:alphaskew}, as the following result demonstrates.

\begin{prop} \label{prop:VCP-calibration}
Let $P$ be a $k$-fold vector cross product on $V$. The element $\alpha_P \in \otimes^{k+1} V^*$ defined by
\begin{equation} \label{eq:alphaP}
\alpha_P(v_1, \ldots, v_k, v_{k+1}) = \langle P(v_1 \wedge \cdots \wedge v_k), v_{k+1} \rangle
\end{equation}
is a calibration on $V$. An oriented $(k+1)$-dimensional subspace $L$ of $V$ is calibrated with respect to $\alpha_P$ if and only if $P(u_1 \wedge \cdots \wedge u_k) = u_{k+1}$ whenever $\{ u_1, \ldots, u_{k+1} \}$ is an \textit{oriented} orthonormal basis for $L$.
\end{prop}
\begin{proof}
We first observe that by~\eqref{eq:Porthogonal}, the covariant $(k+1)$-tensor $\alpha_P$ is totally skew-symmetric and hence a $(k+1)$-form. Let $u_1, \dots, u_{k+1}$ be unit vectors in $V$. By property~\eqref{eq:Pmetric}, the Cauchy--Schwarz inequality, and Hadamard's inequality~\eqref{eq:hadamard-vectors}, we have
\begin{align*}
\alpha_P(u_1, \ldots, u_k, u_{k+1}) & \leq |P(u_1 \wedge \cdots \wedge u_k)| \, |u_{k+1}| \\
& = |u_1 \wedge \cdots \wedge u_k| \, |u_{k+1}| \\
& \leq |u_1| \cdots |u_k| \, |u_{k+1}| = 1.
\end{align*}
Thus $\alpha_P$ is indeed a calibration. Equality is achieved if and only if we have equality in \textit{both} the Cauchy--Schwarz and the Hadamard inequalities. This means that $\{ u_1, \ldots, u_k \}$ must be \textit{orthonormal} and that $u_{k+1} = P(u_1 \wedge \cdots \wedge u_k)$. In particular by~\eqref{eq:Porthogonal} the vector $u_{k+1}$ is also orthogonal to $u_1, \ldots, u_k$ and thus $\{ u_1, \ldots, u_{k+1} \}$ is an orthonormal basis for the $(k+1)$-dimensional subspace $L$ that it spans. Since $\alpha(u_1, \ldots, u_k, u_{k+1}) = 1$, the ordered orthonormal basis $\{ u_1, \ldots, u_{k+1} \}$ induces the given orientation on $L$. The converse is clear.
\end{proof}

The form $\alpha_P$ in~\eqref{eq:alphaP} is called the \textit{calibration form associated to to the vector cross product $P$}.

Despite Proposition~\ref{prop:VCP-calibration}, there are many more interesting calibrations, admitting a rich class of calibrated subspaces, that do not arise from vector cross products. The most well-studied calibrations are those introduced by Harvey--Lawson~\cite{HL} that include all those associated to vector cross products plus several others. They are summarized in Table~\ref{table:HL}.

\begin{table}[H]
\begin{center}
\begin{tabular}{|c|c|c|c||c|} \hline
$n$ & $k+1$ & $\alpha$ & Calibrated subspaces & Associated VCP \\ \hline
$n$ & $n$ & $\vol$ & Entire space & Hodge star \\
$2m$ & $2$ & $\omega$ & Complex lines & Orthogonal complex structure \\
$2m$ & $2r$ & $\tfrac{1}{r!} \omega^r$ & Complex $r$-planes & NONE \\
$2m$ & $m$ & $\real(e^{i \theta} \Omega)$ & Special Lagrangian $m$-planes & NONE \\
$7$ & $3$ & $\varphi$ & Associative $3$planes & $2$-fold VCP of a $\G$-structure \\
$7$ & $4$ & $\psi = \ast \varphi$ & Coassociative $4$-planes & NONE \\
$8$ & $4$ & $\Phi$ & Cayley $4$-planes & $3$-fold VCP of a $\Spin{7}$-structure \\ \hline
\end{tabular}
\end{center}
\caption{The calibrations discussed in Harvey--Lawson~\cite{HL}.} \label{table:HL}
\end{table}

Another very interesting class of calibrations~\cite{BH} is related to the quaternions, but appeared after~\cite{HL}.

It is interesting to note from Table~\ref{table:HL} the particular case when $V$ is equipped with an orthogonal complex structure $J$ and the associated $2$-form $\omega$ is defined by $\omega(v,w) = \langle Jv, w \rangle$. In this case, the complex $r$-planes in $V$ are calibrated by $\frac{1}{r!} \omega^r$, which is the classical Wirtinger inequality, but only the case $r=1$ is associated to a vector cross product.

\subsection{Smith maps and Gray maps} \label{sec:VCP-maps}

Let $P$ be a $k$-fold vector cross product on the $n$-dimensional Euclidean space $V$ with associated calibration $(k+1)$-form $\alpha_P$. Further, let $Q$ be a $k$-fold vector cross product on the $m$-dimensional Euclidean space $W$ with associated calibration $(k+1)$-form $\alpha_Q$, for \textit{the same $k$}. We \textit{do not assume} that $m = n$, although this is a special case.

Let $A : V \to W$ be a linear map. We consider a special type of such a linear map, which we call a Smith map. This is a linear map that is in a particular sense ``compatible'' with the vector cross products $P, Q$ on $V, W$ respectively.

\begin{defn} \label{defn:smith}
Let $\Lambda^k A : \Lambda^k V \to \Lambda^k W$ be the $k$th exterior power $A$. We say that $A$ is a \textit{Smith map} if is satisfies the equation
\begin{equation} \label{eq:smith}
Q (\Lambda^k A) = \lambda^{k-1} A P
\end{equation}
for some \textit{positive} constant $\lambda$. Note that both sides of~\eqref{eq:smith} are linear maps from $\Lambda^k V$ to $W$. Explicitly, $A$ is a Smith map iff for all $v_1, \ldots, v_k \in V$, we have $Q(Av_1 \wedge \cdots \wedge Av_k) = \lambda^{k-1} A\big( P(v_1 \wedge \cdots \wedge v_k) \big)$. Thus we also say that a Smith map is \textit{conformally vector cross product preserving}.
\end{defn}

\begin{rmk} \label{rmk:normA-determined}
We show in Theorem~\ref{thm:smith-conformal} and Proposition~\ref{prop:complex-conformal} that if $k>1$ or $(k=1,n=2)$ then the constant $\lambda$ in~\eqref{eq:smith} is not arbitrary if $A$ is nonzero.
\end{rmk}

\begin{defn} \label{defn:gray}
A special case of Smith maps corresponds to $\lambda = 1$. Then equation~\eqref{eq:smith} becomes $Q (\Lambda^k A) = A P$. Such a map is called a \textit{Gray map} and is \textit{vector cross product preserving}.
\end{defn}

We now establish the fundamental properties of Smith maps.

\begin{thm} \label{thm:smith-conformal}
Let $k>1$ and let $A : V \to W$ be a Smith map. Then either $A = 0$ or $A$ is a conformal injection in the sense of Lemma~\ref{lemma:conf-isom}. Thus if $A \neq 0$ then $\lambda$ is necessarily given by $\lambda = (\dim V)^{-\frac{1}{2}} |A|$.
\end{thm}
\begin{proof}
Suppose that $A$ is \textit{not} injective. Then there exists nonzero $u_1 \in V$ with $A u_1 = 0$. We want to show that $A v = 0$ for all $v \in V$. We can assume $v$ is nonzero and orthogonal to $u_1$. By Corollary~\ref{cor:VCP-surjective} we can find $u_2, \ldots, u_k \in V$ such that $v = P(u_1 \wedge \cdots \wedge u_k)$. Applying $\lambda^{k-1} A$ to both sides and using the fact that $A$ is a Smith map, we find that
\begin{equation*}
\lambda^{k-1} A v = \lambda^{k-1} A \big( P(u_1 \wedge \cdots \wedge u_k ) \big) = Q (A u_1 \wedge \cdots \wedge A u_k ) = 0.
\end{equation*}
Since $\lambda > 0$, we deduce that $Av = 0$. Because $v$ was arbitrary we conclude that $A = 0$ whenever $A$ is not injective.

From now on assume $A$ is nonzero and thus injective. Let $u_1, \ldots, u_{k-1}, v$ be linearly independent vectors in $V$. From Corollary~\ref{cor:VCP-identity} for $(V, P)$, we have that
\begin{equation*}
\begin{aligned}
& P \Big( u_1 \wedge \cdots \wedge u_{k-1} \wedge P \big(u_1 \wedge \cdots \wedge u_{k-1} \wedge P(u_1 \wedge \cdots \wedge u_{k-1} \wedge v) \big) \Big) \\
& \qquad = - |u_1 \wedge \cdots \wedge u_{k-1}|^2 P(u_1 \wedge \cdots \wedge u_{k-1} \wedge v).
\end{aligned}
\end{equation*}
We apply $\lambda^{3(k-1)} A$ to both sides of the above and use the Smith equation~\eqref{eq:smith} repeatedly. This gives
\begin{equation} \label{eq:smith-conformal-temp}
\begin{aligned}
& Q \Big( Au_1 \wedge \cdots \wedge Au_{k-1} \wedge Q \big(Au_1 \wedge \cdots \wedge Au_{k-1} \wedge Q(Au_1 \wedge \cdots \wedge Au_{k-1} \wedge Av) \big) \Big) \\
& \qquad = - |u_1 \wedge \cdots \wedge u_{k-1}|^2 \lambda^{2(k-1)} Q(Au_1 \wedge \cdots \wedge Au_{k-1} \wedge Av).
\end{aligned}
\end{equation}
Since $A$ is injective, the vectors $Au_1, \ldots Au_{k-1}, Av$ are linearly independent in $W$. From Corollary~\ref{cor:VCP-identity} for $(W, Q)$, we have that
\begin{equation} \label{eq:smith-conformal-temp2}
\begin{aligned}
& Q \Big( Au_1 \wedge \cdots \wedge Au_{k-1} \wedge Q \big(Au_1 \wedge \cdots \wedge Au_{k-1} \wedge Q(Au_1 \wedge \cdots \wedge Au_{k-1} \wedge Av) \big) \Big) \\
& \qquad = - |Au_1 \wedge \cdots \wedge Au_{k-1}|^2 Q(Au_1 \wedge \cdots \wedge Au_{k-1} \wedge Av).
\end{aligned}
\end{equation}
From~\eqref{eq:Pmetric} and the linear independence of $Au_1, \ldots, Au_{k-1}, Av$ we have $Q(Au_1 \wedge \cdots \wedge Au_{k-1} \wedge Av) \neq 0$. Comparing equations~\eqref{eq:smith-conformal-temp} and~\eqref{eq:smith-conformal-temp2} therefore gives
\begin{equation} \label{eq:smith-conformal-temp3}
|Au_1 \wedge \cdots \wedge Au_{k-1}|^2 = \lambda^{2(k-1)} |u_1 \wedge \cdots \wedge u_{k-1}|^2
\end{equation}
whenever $u_1, \ldots, u_{k-1}$ are linearly independent, and hence by multilinearity for all $u_1, \ldots, u_{k-1}$. Let $A^* : W \to V$ be the adjoint map, and let $B = \lambda^{-2} A^* A : V \to V$. The map $B$ is self-adjoint and positive. Equation~\eqref{eq:smith-conformal-temp3} says
\begin{align*}
|u_1 \wedge \cdots \wedge u_{k-1}|^2 & = \lambda^{-2(k-1)} \det \langle Au_i, Au_j \rangle \\
& = \det \langle \lambda^{-2} A^* A u_i, u_j \rangle = \langle Bu_1 \wedge \cdots \wedge Bu_{k-1}, u_1 \wedge \cdots u_{k-1} \rangle.
\end{align*}
We can thus apply Lemma~\ref{lemma:eigen} (this is where we need the hypothesis that $k>1$) to conclude that $B$ is the identity, so $A^* A = \lambda^2 I$. From Lemma~\ref{lemma:conf-isom} we conclude that $A$ is conformal with conformal factor $\lambda = (\dim V)^{-\frac{1}{2}} |A|$ as claimed.
\end{proof}

\begin{cor} \label{cor:gray}
For $k > 1$, the nonzero Gray maps are precisely the nonzero Smith maps that are also isometric injections.
\end{cor}
\begin{proof}
This is immediate from Theorem~\ref{thm:smith-conformal}.
\end{proof}

The case when $k=1$ is special, as it requires a restriction on the dimension of the domain $V$. It is well known but we include it for completeness. The precise statement is as follows.
\begin{prop} \label{prop:complex-conformal}
Let $k=1$. Let $A : V \to W$ be a Smith map. Then the conclusion of Theorem~\ref{thm:smith-conformal} always holds if and only if $\dim V = 2$.
\end{prop}
\begin{proof}
Let $\dim V = n$ and $\dim W = m$. By Example~\ref{ex:complex}, the maps $P = V \to V$ and $Q : W \to W$ are both orthogonal complex structures, and moreover $n = \dim V = 2r$ and $m = \dim W = 2s$. We can choose orthonormal bases of the form $\{ u_1, \ldots, u_p, Pu_1, \ldots, Pu_r \}$ and $\{ e_1, \ldots, e_s, Qe_1, \ldots, Qe_s\}$ of $V$ and $W$, respectively. With respect to such bases, the matrices for $P$ and $Q$ are
\begin{equation*}
P = \begin{pmatrix} 0_{r \times r} & -I_{r \times r} \\ I_{r \times r} & 0_{r \times r} \end{pmatrix}, \qquad Q = \begin{pmatrix} 0_{s \times s} & -I_{s \times s} \\ I_{s \times s} & 0_{s \times s} \end{pmatrix}.
\end{equation*}
The Smith equation~\eqref{eq:smith} in this case is $A P = Q A$. It follows easily from this equation that with respect to these bases, the $2s \times 2r$ matrix for $A$ must be of the block diagonal form
\begin{equation*}
A = \begin{pmatrix} B & -C \\ C & B \end{pmatrix}
\end{equation*}
where $B, C$ are $s \times r$ matrices. Because we are using orthonormal bases, the matrix of the adjoint $A^*$ is just the transpose. Hence we have
\begin{equation*}
A^* A = \begin{pmatrix} B^T & C^T \\ -C^T & B^T \end{pmatrix} \begin{pmatrix} B & -C \\ C & B \end{pmatrix} = \begin{pmatrix} B^T B + C^T C & C^T B - B^T C \\ -(C^T B - B^T C) & B^T B + C^T C \end{pmatrix}.
\end{equation*}
By Lemma~\ref{lemma:conf-isom}, if $A$ is nonzero then it is a conformal injection if and only if $A^* A = \lambda^2 I$ for some $\lambda > 0$. If $\dim V = 2$, so $r=1$, then $B, C$ are $s \times 1$ column vectors, and $B^T C$ is $1 \times 1$. Thus in this case
\begin{equation*}
A^T A = \begin{pmatrix} |B|^2 + |C|^2 & 0 \\ 0 & |B|^2 + |C|^2 \end{pmatrix} = (|B|^2 + |C|^2) I.
\end{equation*}
Thus a nonzero Smith map $A$ is indeed a conformal injection in this case. However, if $r > 1$ it is clear that there exist choices of $B, C$ which do not yield conformal injections.
\end{proof}

Next we investigate the properties of Smith maps under composition and inversion.

\begin{prop} \label{prop:smith-comp}
The composition of Smith maps is a Smith map. If a Smith map is invertible, then its inverse is a Smith map.
\end{prop}
\begin{proof}
Let $(V, P)$, $(W,Q)$, and $(U, R)$ be Euclidean spaces equipped with $k$-fold vector cross products $P, Q, R$, respectively. Let $A : V \to W$ and $B : W \to U$ be Smith maps. Then we have
\begin{equation} \label{eq:smithcomptemp}
Q (\Lambda^k A) = \lambda^{k-1} A P, \qquad \qquad R (\Lambda^k B) = \mu^{k-1} B Q
\end{equation}
for some $\lambda, \mu > 0$. Define $C = B A : V \to W$. Since $\Lambda^k(B A) = (\Lambda^k B) (\Lambda^k A)$, from~\eqref{eq:smithcomptemp} we have
\begin{align*}
R (\Lambda^k C) & = R (\Lambda^k B) (\Lambda^k A) = \mu^{k-1} B Q (\Lambda^k A) \\
& = \mu^{k-1} \lambda^{k-1} B A P = (\mu \lambda)^{k-1} C P = \rho^{k-1} C P
\end{align*}
where $\rho = \lambda \mu > 0$. Thus $C$ is a Smith map.
 
Now suppose that $A : V \to V$ is an invertible Smith map. The scaling of $A^{-1}$ is obviously inverse to the scaling of $A$. But we need to show that the inverse is still conformally cross-product preserving. From $(\Lambda^k A)^{-1} = (\Lambda^k A^{-1})$ we have
\begin{align*}
P (\Lambda^k A) = \lambda^{k-1} A P \quad & \Longleftrightarrow \quad P = \lambda^{k-1} A P (\Lambda^k A^{-1}) \\
& \Longleftrightarrow \quad (\lambda^{-1})^{k-1}A^{-1} P = P (\Lambda^k A^{-1}).
\end{align*}
Thus $A^{-1}$ is a Smith map.
\end{proof}

\begin{cor} \label{cor:gray-comp}
The composition of Gray maps is a Gray map. If a Gray map is invertible, then its inverse is a Gray map.
\end{cor}
\begin{proof}
This is immediate from Proposition~\ref{prop:smith-comp} and Corollary~\ref{cor:gray}.
\end{proof}

\subsection{Smith maps and calibrations} \label{sec:smith-calib}

In this section we investigate relations between Smith maps and calibrations. In a precise sense that we explain, the image of a Smith map is calibrated and conversely, any calibrated subspace is the image of a Smith map in many different ways. The consequences for manifolds are discussed in $\S$\ref{sec:smith-maps-manifolds}.

\begin{lemma} \label{lemma:smith-calib}
Let $A : V \to W$ be a Smith map. Then we have
\begin{equation} \label{eq:smith-calib}
A^* \alpha_Q = \lambda^{k+1} \alpha_P. 
\end{equation}
\end{lemma}
\begin{proof}
Let $v_1, \ldots, v_{k+1} \in V$ be arbitrary. We compute
\begin{align*}
A^* \alpha_Q (v_1, \ldots, v_{k+1}) & = \alpha_Q (Av_1, \ldots, Av_{k+1}) = \langle Q (\Lambda^k A) (v_1 \wedge \cdots \wedge v_k), Av_{k+1} \rangle \\
& = \langle \lambda^{k-1} A P(v_1 \wedge \cdots \wedge v_k), Av_{k+1} \rangle = \lambda^{k-1} \lambda^2 \langle P(v_1 \wedge \cdots \wedge v_k), v_{k+1} \rangle \\
& = \lambda^{k+1} \alpha_P (v_1, \ldots, v_{k+1})
\end{align*}
and hence~\eqref{eq:smith-calib} holds.
\end{proof}

We have seen that if $A$ is a nonzero Smith map and either $k>1$ or $(k=1,n=2)$ so that either Theorem~\ref{thm:smith-conformal} or Proposition~\ref{prop:complex-conformal} holds, then $A$ is a conformal injection with conformal factor $\lambda > 0$ where $\lambda^2 = \frac{1}{n} |A|^2$ and $n = \dim V$. Moreover, by Lemma~\ref{lemma:smith-calib}, any Smith map satisfies~\eqref{eq:smith-calib}. In fact, when the vector cross product $P$ on the domain $V$ is of Type I from Table~\ref{table:VCP}, then these two conditions together are equivalent to the Smith map equation~\eqref{eq:smith}. The precise statement is proved in Proposition~\ref{prop:smith-conf-calib} below. First we need some definitions.

\begin{defn} \label{defn:conformal-calib}
Let $V, W$ be Euclidean spaces with calibration $(k+1)$-forms $\alpha_V, \alpha_W$, respectively. Let $A : V \to W$ be a nonzero linear map. Suppose that the following two conditions both hold:
\begin{enumerate}[(i)]
\item $A$ is a conformal injection, necessarily with conformal factor $\lambda = \frac{1}{\sqrt{\dim V}} |A|$. Equivalently we write $A^* g_W = \lambda^2 g_V$ where $g_V, g_W$ are the Euclidean inner products on $V, W$, respectively;
\item $A^* \alpha_W = \lambda^{k+1} \alpha_V$.
\end{enumerate}
Then we say that $A$ is \textit{conformally calibrating}.
\end{defn}

\begin{rmk} \label{rmk:G2spin7}
In two particular cases, condition (ii) of Definition~\ref{defn:conformal-calib} automatically implies condition (i). These are the cases when $V = W$ and $\alpha_P = \alpha_Q$ is a calibration induced from a vector cross product $P=Q$ of either Type III or Type IV from Table~\ref{table:VCP}. The fact that $A^* \alpha_P = \lambda^{k+1} \alpha_P$ implies $A^* g_V = \lambda^2 g_V$ in these two cases is well-known. See~\cite[Sections 3.1 and 5.1]{K1}, for example. The reason this happens is that a $\G$-structure or a $\Spin{7}$-structure determines the Euclidean inner product uniquely.
\end{rmk}

\begin{ass} \label{assume:typeI}
From now on, we restrict to the case when the vector cross product $P$ on the domain $V$ is of Type I from Table~\ref{table:VCP}.
\end{ass}

\begin{defn} \label{defn:nomenclature}
An oriented Euclidean $n$-space $(V, g_V, \ast_V, \vol_V)$ is an $n$-dimensional Euclidean space $(V, g_V)$ equipped with an orientation which induces a Hodge star operator $\ast_V$. The associated calibration $n$-form is the volume form $\vol_V$. By Example~\ref{ex:hodgestar} this is the same as saying $(V, g_V)$ is equipped with a vector cross product of Type I from Table~\ref{table:VCP}. In~\cite{Sm} such a structure $(V, g_V, \ast_V, \vol_V)$ is called a `conformal $n$-triad' but we do not use this terminology.

An $(n-1)$-fold VCP space $(W, g_W, Q, \alpha_Q)$ is a Euclidean space $(W, g_W)$ equipped with an $(n-1)$-fold vector cross product $Q$ on $W$ and its associated calibration $n$-form $\alpha_Q$. Note that if $\dim W=n$ then an $(n-1)$-fold VCP space is an oriented Euclidean $n$-space by Table~\ref{table:VCP}.
\end{defn}

\begin{prop} \label{prop:smith-conf-calib}
Let $(V, g_V, \ast, \vol_V)$ be an oriented Euclidean $n$-space. Let $(W, g_W, Q, \alpha_Q)$ be an $(n-1)$-fold VCP space. Let $A : V \to W$ be a nonzero linear map. Then $A$ is a Smith map in the sense of equation~\eqref{eq:smith} if and only if $A$ is conformally calibrating in the sense of Definition~\ref{defn:conformal-calib}.
\end{prop}
\begin{proof}
We first note that necessarily we have either $k=n-1>1$ or else $(k=1,n=2)$. In these cases we have already shown that any Smith map is conformally calibrating. Thus assume that $A$ is conformally calibrating. That is,
\begin{equation} \label{eq:smith-conf-calib-temp-0}
A^* g_W = \lambda^2 g_V, \qquad A^* \alpha_Q = \lambda^{k+1} \alpha_P = \lambda^n \vol_V,
\end{equation}
and we need to prove that
\begin{equation} \label{eq:smith-conf-calib-temp}
Q (\Lambda^{n-1} A) (v_1 \wedge \cdots \wedge v_{n-1}) = \lambda^{n-2} A P (v_1 \wedge \cdots \wedge v_{n-1})
\end{equation}
for all $v_1, \ldots, v_{n-1} \in V$. Both sides of~\eqref{eq:smith-conf-calib-temp} vanish if $\{ v_1, \ldots, v_{n-1} \}$ is linearly dependent so we may assume it is linearly independent, so both sides of~\eqref{eq:smith-conf-calib-temp} depend only on the $(n-1)$-plane $u_1 \wedge \cdots \wedge u_{n-1} \in \Lambda^{n-1} V$ where $\{ u_1, \ldots, u_{n-1} \}$ is an orthonormal basis for $\spa \{ v_1, \ldots, v_{n-1} \}$. That is, $u_1 \wedge \cdots \wedge u_{n-1} = t v_1 \wedge \cdots \wedge v_{n-1}$ for some nonzero $t \in \R$. Thus~\eqref{eq:smith-conf-calib-temp} holds for all $v_1, \ldots, v_{n-1} \in V$ if and only if
\begin{equation} \label{eq:smith-conf-calib-temp-b}
Q (\Lambda^{n-1} A) (u_1 \wedge \cdots \wedge u_{n-1}) = \lambda^{n-2} A P (u_1 \wedge \cdots \wedge u_{n-1})
\end{equation}
holds for all \textit{orthonormal} $u_1, \ldots, u_{n-1} \in V$. By Lemma~\ref{lemma:conf-isom} we know $A = \lambda \hat A$ where $\lambda > 0$ and $\hat A : V \to W$ is an isometric injection. Therefore~\eqref{eq:smith-conf-calib-temp-0} becomes
\begin{equation} \label{eq:smith-conf-calib-temp-2}
\hat A^* g_W = g_V, \qquad \hat A^* \alpha_Q = \vol_V.
\end{equation}
Let $u_1, \ldots, u_{n-1}$ be orthonormal. Let $u_n = P (u_1 \wedge \cdots \wedge u_{n-1})$. Then $\{ u_1, \ldots, u_n \}$ is an oriented orthonormal basis for $V$ since $P = \ast$ is the Hodge star operator. Using~\eqref{eq:smith-conf-calib-temp-2} we compute
\begin{align} \nonumber
1 & = \vol_V (u_1, \ldots u_n ) = (\hat A^* \alpha_Q) (u_1, \ldots, u_n) \\ \nonumber
& = \alpha_Q( \hat A u_1, \ldots, \hat A u_{n-1}, \hat A u_n ) \\ \label{eq:smith-conf-calib-temp-3}
& = g_W \big( Q (\Lambda^{n-1} \hat A)(u_1 \wedge \cdots \wedge u_{n-1}), \hat A u_n ).
\end{align}
Since $\hat A$ is an isometric injection, $\hat A u_n$ and $(\Lambda^{n-1} \hat A)(u_1 \wedge \cdots \wedge u_{n-1})$ are unit vectors in $W$ and $\Lambda^{n-1} W$, respectively. Since $Q$ is a vector cross product, property~\eqref{eq:Pmetric} says that $Q (\Lambda^{n-1} \hat A)(u_1 \wedge \cdots \wedge u_{n-1})$ is also a unit vector in $W$. Then Cauchy--Schwarz applied to the equality~\eqref{eq:smith-conf-calib-temp-3} gives
\begin{equation} \label{eq:image-calib}
Q (\Lambda^{n-1} \hat A)(u_1 \wedge \cdots \wedge u_{n-1}) = \hat A u_n = \hat A P(u_1 \wedge \cdots \wedge u_{n-1}).
\end{equation}
Multiplying both sides by $\lambda^{n-1}$ gives equation~\eqref{eq:smith-conf-calib-temp-b} as required.
\end{proof}

The above result has a number of important corollaries.

\begin{cor} \label{cor:image-calib}
Let $(V, g_V, \ast, \vol_V)$ be an oriented Euclidean $n$-space. Let $(W, g_W, Q, \alpha_Q)$ be an $(n-1)$-fold VCP space. Let $A : V \to W$ be a nonzero Smith map. Then $A$ induces an orientation on its image $L = A(V)$ such that $L$ is calibrated with respect to $\alpha_Q$.
\end{cor}
\begin{proof}
Since $A = \lambda \hat A$ where $\hat A$ is an isometric injection, the image $L = A(V) = \hat A (V)$ is isometric to $V$ and thus inherits an induced orientation by declaring that $\{ \hat A u_1, \ldots, \hat A u_n \}$ is oriented (and necessarily orthonomal) whenever $\{ u_1, \ldots, u_n \}$ is an oriented orthonormal basis for $V$. Now equation~\eqref{eq:image-calib} in the proof of Proposition~\ref{prop:smith-conf-calib} says that
\begin{equation*}
Q \big( (\hat A u_1) \wedge \cdots \wedge (\hat A u_{n-1}) \big) = \hat A u_n.
\end{equation*}
Then Proposition~\ref{prop:VCP-calibration} applied with $P$ replaced by $Q$ and $u_i$ replaced by $\hat A u_i$ says that $L$ is calibrated with respect to $\alpha_Q$.
\end{proof}

\begin{cor} \label{cor:conformals}
Let $(V, g_V, \ast, \vol_V)$ be an oriented Euclidean $n$-space. A nonzero Smith map $A : V \to V$ is precisely an orientation preserving conformal isomorphism.
\end{cor}
\begin{proof}
Apply Proposition~\ref{prop:smith-conf-calib} with $W = V$ and $Q = P = \ast$. Then $A : V \to V$ is a nonzero Smith map if and only if it satisfies $A^* g_V = \lambda^2 g_V$ and $A^* \vol_V = \lambda^n \vol_V$, with $n \lambda = |A| > 0$. This precisely means that $A$ is an orientation preserving conformal isomorphism.
\end{proof}

\begin{cor} \label{cor:conformal-invariance}
Let $(V, g_V, \ast, \vol_V)$ be an oriented Euclidean $n$-space. Let $(W, g_W, Q, \alpha_Q)$ be an $(n-1)$-fold VCP space. Let $A : V \to W$ be a Smith map. Then for any orientation preserving conformal isomorphism $B : V \to V$, the composition $AB : V \to W$ is a Smith map.
\end{cor}
\begin{proof}
This is immediate from Corollary~\ref{cor:conformals} and Proposition~\ref{prop:smith-comp}.
\end{proof}

\begin{rmk} \label{rmk:conformal-invariance}
Corollary~\ref{cor:conformal-invariance} is very important, as it implies that the notion of a Smith map between appropriate Riemannian manifolds is a \textit{conformally invariant} notion. This is discussed in $\S$\ref{sec:smith-maps-defn}.
\end{rmk}

In Corollary~\ref{cor:image-calib} we showed that the image of a nonzero Smith map is calibrated. In fact, a kind of converse holds, which is Proposition~\ref{prop:smith-calibrated} below.

\begin{lemma} \label{lemma:gray-calibrated}
Let $(W, g_W, Q, \alpha_Q)$ be an $(n-1)$-fold VCP space. Let $L$ be an $n$-dimensional oriented subspace, and let $\iota : L \to W$ be the linear inclusion. Equip $L$ with the induced inner product $g_L = \iota^* g_W$ and volume form $\vol_L$. Then $L$ is calibrated with respect to $\alpha_Q$ if and only if $\iota : L \to W$ is a Gray map.
\end{lemma}
\begin{proof}
The map $\iota : L \to W$ is a Gray map if and only if it is a Smith map with $\lambda = 1$. By Proposition~\ref{prop:smith-conf-calib} we deduce that $\iota$ is a Gray map if and only if $\iota^* \alpha_Q = \vol_L$, which is precisely the condition that $L$ is calibrated with respect to $\alpha_Q$.
\end{proof}

More generally, we have the following result.

\begin{prop} \label{prop:smith-calibrated}
Let $(W, g_W, Q, \alpha_Q)$ be an $(n-1)$-fold VCP space. Let $V$ be a $n$-dimensional real vector space. Let $A : V \to W$ be a linear injection. Equip $V$ with the inner product $g_V = A^* g_W$. Then the following are equivalent:
\begin{enumerate}[(i)]
\item The image $L=A(V)$ admits an orientation for which $L$ is calibrated with respect to $\alpha_Q$.
\item The space $V$ admits an orientation such that $A : V \to W$ is a Gray map, with respect to $\ast_{g_V}$.
\end{enumerate}
\end{prop}
\begin{proof}
Let $L = A(V)$ and give $L$ the induced inner product $g_L = \iota^* g_W$ where $\iota : L \to W$ is the inclusion. Let $A_1$ denote the map $A$ with codomain $L = A(V)$. That is, $A_1 : V \to L$ is given by $A_1 (v) = A(v)$. Then $A_1 : V \to L$ is a linear isomorphism, and $A = \iota A_1$, so $A^* = A_1^* \iota^*$. Hence $g_V = A^* g_W = A_1^* g_L$, so $A_1 : V \to L$ is an isometry, and any orientation $\vol_L$ on $L$ compatible with $g_L$ corresponds to a unique orientation $A_1^* \vol_L$ on $V$ compatible with $g_V = A_1^* g_L$. Thus $A_1 : V \to L$ is an invertible Gray map with respect to the Hodge star operators $\ast_{g_V}, \ast_{g_L}$ on $V, L$, respectively.

By Lemma~\ref{lemma:gray-calibrated}, condition (i) is equivalent to the statement that $\iota : L \to W$ is a Gray map with respect to some orientation $\vol_L$ on $L$ compatible with $g_L$. Using $A = \iota A_1$ and the invertibility of $A_1$, we deduce from Corollary~\ref{cor:gray-comp} that $\iota$ is a Gray map if and only if $A$ is a Gray map.
\end{proof}

\begin{rmk} \label{rmk:smith-calibrated}
Suppose $V$ is an $n$-dimensional real vector space and $A : V \to W$ is a linear injection where $(W, g_W, Q, \alpha_Q)$ is an $(n-1)$-fold VCP space, and that the image $L = A(V)$ admits an orientation for which $L$ is calibrated with respect to $\alpha_Q$. Then Proposition~\ref{prop:smith-calibrated} gives the structure of an oriented Euclidean $n$-space on $V$ such that $A : V \to W$ is a Gray map. But then Corollaries~\ref{cor:conformals} and~\ref{cor:conformal-invariance} say that by precomposing $A$ with any orientation preserving conformal isomorphism $B$ of $(V, g_V, \ast, \vol_V)$, we obtain another Smith map $AB : V \to W$ whose image is calibrated.
\end{rmk}

\subsection{A generalized calibration inequality} \label{sec:generalized-calib}

We now establish the fundamental \textit{generalized calibration inequality} that is a crucial ingredient for the \textit{energy identity} in $\S$\ref{sec:energy-identity}.

\begin{thm}[Generalized calibration inequality] \label{thm:generalized-calib}
Let $(V, g_V, \ast, \vol_V)$ be an oriented Euclidean $n$-space. Let $(W, g_W, Q, \alpha_Q)$ be an $(n-1)$-fold VCP space. Let $A : V \to W$ be a linear map. Let $\{ u_1, \ldots, u_n \}$ be an oriented orthonormal basis for $V$. Then we have
\begin{equation} \label{eq:generalized-calib}
\left( \frac{1}{\sqrt{n}} \right)^n |A|^n - (A^* \alpha_Q)(u_1, \ldots, u_n) \geq 0,
\end{equation}
with equality if and only if $A : V \to W$ is a Smith map.
\end{thm}
\begin{proof}
If $A$ is zero then equality holds in~\eqref{eq:generalized-calib} trivially, and the zero map trivially satisfies~\eqref{eq:smith}. Thus assume that $A$ is nonzero. Let $u_1, \ldots, u_n$ be an oriented orthonormal basis of $V$. We compute
\begin{align} \nonumber
(A^* \alpha_Q)(u_1, \ldots, u_n) & = \alpha_Q (Au_1, \ldots, Au_n) & & \\ \label{eq:calib-ineq-1}
& \leq |(Au_1) \wedge \cdots \wedge (Au_n)| & & \text{by~\eqref{eq:comassone}, since $\alpha_Q$ is a calibration} \\ \nonumber
& = |(\Lambda^n A)(u_1 \wedge \cdots \wedge u_n)| & & \\ \nonumber
& = |\Lambda^n A| & & \text{by~\eqref{eq:matrix-norm}, since $\{ u_1 \wedge \cdots \wedge u_n \}$ is o.n.\ basis for $\Lambda^n V$} \\ \label{eq:calib-ineq-2}
& \leq \left( \frac{1}{\sqrt{n}} \right)^n |A|^n & & \text{by Hadamard's inequality~\eqref{eq:hadamardonA} with $r=n$}
\end{align}
Thus the inequality~\eqref{eq:generalized-calib} has been established. Equality holds if and only if equality holds in both~\eqref{eq:calib-ineq-1} and~\eqref{eq:calib-ineq-2}. By Lemma~\ref{lemma:hadamard}, equality holds in~\eqref{eq:calib-ineq-2} if and only if condition (i) of Definition~\ref{defn:conformal-calib} holds. But then $A = \lambda \hat A$ where $\hat A : V \to W$ is an isometric injection. Thus $\{ \hat A u_1, \ldots, \hat A u_n \}$ is an orthonormal basis for the $n$-dimensional subspace $\im A$ of $W$. By Corollary~\ref{cor:hadamard-vectors} applied to $\im A$ we have
\begin{equation*}
|(Au_1) \wedge \cdots \wedge (Au_n)| = \lambda^n |(\hat Au_1) \wedge \cdots \wedge (\hat Au_n)| = \lambda^n.
\end{equation*}
Thus, given equality in~\eqref{eq:calib-ineq-2}, equality also holds in~\eqref{eq:calib-ineq-1} if and only if $(A^* \alpha_Q)(u_1, \ldots, u_n) = \lambda^n$, and this holds if and only if $A^* \alpha_Q = \lambda^n \vol_V$, which is condition (ii) of Definition~\ref{defn:conformal-calib}. Thus $A : V \to W$ is conformally calibrating. Using Proposition~\ref{prop:smith-conf-calib}, we conclude that equality holds in~\eqref{eq:generalized-calib} if and only if $A$ is a Smith map.
\end{proof}

\begin{rmk} \label{rmk:generalized-calib}
We remark that Theorem~\ref{thm:generalized-calib} generalizes Proposition~\ref{prop:VCP-calibration} in two important ways:
\begin{itemize} \setlength\itemsep{-1mm}
\item It allows general injective linear maps $A : V \to W$. This is the passage from subspaces to maps.
\item It allows the freedom that $g_V$ is only conformal to $A^* g_W$, not necessarily isometric. This is the passage from Gray maps to Smith maps.
\end{itemize}
The consequences for manifolds are discussed in $\S$\ref{sec:smith-maps-manifolds}.
\end{rmk}

\section{Smith maps between manifolds} \label{sec:smith-maps-manifolds}

In this section we apply the linear algebraic results of $\S$\ref{sec:linear} to maps between manifolds equipped with the appropriate geometric structures. The most important result of this section is that, if $d \alpha = 0$, then a Smith map $u : (\Sigma^n, g, \ast, \vol) \to (M^m, h, Q, \alpha)$ is a minimizer of the $n$-energy in its homology class, and thus is also an $n$-harmonic map.

\subsection{Definition and basic properties of Smith maps between manifolds} \label{sec:smith-maps-defn}

Let $(\Sigma, g)$ be an oriented Riemannian $n$-manifold, with Riemannian volume form $\vol$. At each $x \in \Sigma$, the associated Hodge star operator $\ast_x$ is an $(n-1)$-fold vector cross product on $T_x \Sigma$ of Type I from Table~\ref{table:VCP} and $\vol_x$ is its associated calibration $n$-form. Thus $(T_x \Sigma, g_x, \ast_x, \vol_x)$ is an oriented Euclidean $n$-space for each $x \in \Sigma$, as in Definition~\ref{defn:nomenclature}.

Let $(M, h)$ be a Riemannian $m$-manifold equipped with an $(n-1)$-fold vector cross product $Q$ and its associated calibration $n$-form $\alpha$. This means that $Q$ is a smooth section of $\Lambda^{n-1} T^* M \otimes TM$ and $\alpha$ is a smooth $n$-form on $M$ such that, at each $y \in M$, the map $Q_y : \Lambda^{n-1} T_y M \to T_y M$ is a vector cross product with associated calibration $n$-form $\alpha_y$. Thus $(T_y M, h_y, Q_y, (\alpha_Q)_y)$ is an $(n-1)$-fold VCP space, as in Definition~\ref{defn:nomenclature}.

\begin{defn} \label{defn:smith-map-smooth}
Let $u : \Sigma \to M$ be a smooth map. We say that $u$ is a \textit{Smith map} if the differential $(d u)_x : T_x \Sigma \to T_{u(x)} M$ is a Smith map for all $x \in \Sigma$ in the sense of Definition~\ref{defn:smith}. Explicitly, a Smith map satisfies the equation
\begin{equation} \label{eq:smith-u}
Q (\Lambda^{n-1} d u) = \frac{1}{(\sqrt{n})^{n-2}} |d u|^{n-2} (d u) \ast
\end{equation}
where $\ast$ is the Hodge star operator on $\Sigma$. Both sides of~\eqref{eq:smith-u} are smooth sections of $\Lambda^{n-1} T^* \Sigma \otimes u^* TM$. In~\cite{Sm} the Smith maps are called \textit{multiholomorphic maps}, which is reasonable given Example~\ref{ex:smith-u-II} in $\S$\ref{sec:smith-examples}.
\end{defn}

A point $x \in \Sigma$ where $(d u)_x = 0$ is called a \textit{critical point} of $u$. We define
\begin{equation*}
\crit_u = \{ x \in \Sigma \mid (d u)_x = 0 \},
\end{equation*}
the set of all critical points of $u$, also called the \textit{critical locus} of $u$.

\begin{prop} \label{prop:smith-conf-calib-maps}
A smooth map $u : \Sigma \to M$ is a Smith map if and only if both of the following conditions hold:
\begin{itemize} \setlength\itemsep{-1mm}
\item $u$ is \textit{weakly conformal}. (This means that $u^* h = \frac{1}{n} |d u|^2 g$.)
\item $u^* \alpha = \frac{1}{(\sqrt{n})^{n}} |d u|^{n} \vol$.
\end{itemize}
We call such maps ``conformally calibrating''. Thus a map is Smith if and only if it is conformally calibrating.
\end{prop}
\begin{proof}
This is immediate from Proposition~\ref{prop:smith-conf-calib}.
\end{proof}

\begin{rmk} \label{rmk:conf-inv-alt}
The alternate characterization of Smith maps in Proposition~\ref{prop:smith-conf-calib-maps} makes it easy to see that \emph{precomposition} of a Smith map $u: \Sigma \to M$ by an orientation preserving weakly conformal map $F : \Sigma \to \Sigma$ yields another Smith map $u \circ F : \Sigma \to M$. See Proposition~\ref{prop:smith-u-conf-inv}.
\end{rmk}

Note that both of the above conditions are trivially satisfied at any critical point of $u$. The first condition expresses the fact that $(d u)_x : T_x \Sigma \to T_{u(x)} M$ is a conformal injection at any noncritical point.

Smith maps have several basic properties which are immediate from the results of $\S$\ref{sec:smith-calib}.

\begin{prop} \label{prop:smith-u-image}
Let $u : \Sigma \to M$ be a Smith map. Let $\Sigma^{\circ} = \Sigma \setminus \crit_u$ be the open set on which $u$ has no critical points. Then for any $x \in \Sigma^{\circ}$, the image $(d u)_x (T_x \Sigma)$ is an $n$-dimensional subspace of $T_{u(x)} M$ that is calibrated with respect to $\alpha_Q$. That is, $u(\Sigma^{\circ})$ is an immersed calibrated submanifold of $M$. Here the orientation of $u(\Sigma^{\circ})$ is the one naturally induced by the injection $(d u)_x$ for each $x \in \Sigma^{\circ}$.
\end{prop}
\begin{proof}
This is immediate from Corollary~\ref{cor:image-calib}.
\end{proof}

\begin{prop} \label{prop:smith-u-calib}
Let $Z$ be a smooth oriented $n$-manifold, and let $(M, h, Q, \alpha_Q)$ be as above. Let $u : Z \to M$ be a smooth immersion. Give the image $u(Z)$ the orientation induced by the injection $(d u)_z$ for each $z \in Z$. If $u(Z)$ is calibrated with respect to $\alpha_Q$, then let $g = u^* h$, which is a Riemannian metric on $Z$. With respect to $g$ and the orientation on $Z$, the map $u : Z \to M$ is a Smith map.
\end{prop}
\begin{proof}
This is immediate from Proposition~\ref{prop:smith-calibrated}.
\end{proof}

\begin{prop} \label{prop:smith-u-conf-inv}
Let $u : \Sigma \to M$ be a Smith map. Let $F : \Sigma \to \Sigma$ be an orientation preserving conformal diffeomorphism. Then the composition $u \circ F : \Sigma \to M$ is a Smith map.
\end{prop}
\begin{proof}
This is immediate from Corollary~\ref{cor:conformal-invariance}.
\end{proof}

Thus the image of any Smith map is a calibrated submanifold (away from the critical points), and any calibrated submanifold is the image of a Smith map, in many different ways, since the Smith equation~\eqref{eq:smith-u} is conformally invariant in the sense of Proposition~\ref{prop:smith-u-conf-inv}.

Therefore, for those calibrations that correspond to a vector cross product, calibrated submanifolds are in some sense equivalent to Smith maps. We observe that Gray maps also have these two properties with respect to calibrated submanifolds. However, the crucial difference is that the Smith equation is conformally invariant, whereas the Gray equation is not. Conformal invariance is a fundamental feature of those geometric partial differential equations that exhibit phenomena of removable singularities, compactness, and bubbling. This is why Smith maps are much more preferable than Gray maps.

We close this section with another demonstration of this conformal invariance, which is instructive. Let $u : \Sigma \to M$ be a smooth map. Define the \textit{Smith operator} $\smith$ to be the operator that takes $u$ to
\begin{equation*}
\smith u = Q (\Lambda^{n-1} d u) - \frac{1}{(\sqrt{n})^{n-2}} |d u|^{n-2} (d u) \ast.
\end{equation*}
(In~\cite{Sm} this operator is called the \textit{multi-Cauchy-Riemann operator}.) We observe that $\smith u$ is a section of $\Lambda^{n-1} T^* \Sigma \otimes u^* TM$ and that $\smith u = 0$ if and only if $u$ is a Smith map. Let $F : \Sigma \to \Sigma$ be a smooth map. This corresponds to the special case when $(M, h, Q, \alpha_Q) = (\Sigma, g, \ast, \vol)$. In this case the Smith operator applied to $F$ is
\begin{equation*}
\smith F = \ast (\Lambda^{n-1} d F) - \frac{1}{(\sqrt{n})^{n-2}} |d F|^{n-2} (d F) \ast.
\end{equation*}
Now consider the composition $u \circ F : \Sigma \to M$. Using $d (u \circ F) = (d u)(d F)$ and Corollary~\ref{cor:conf-isom-comp} (which holds trivially at critical points) we compute
\begin{align*}
\smith(u \circ F) & = Q \big( \Lambda^{n-1} \big( (d u)(d F) \big) \big) - \frac{1}{(\sqrt{n})^{n-2}} |(d u)(d F)|^{n-2} (d u) (d F) \ast \\
& = Q (\Lambda^{n-1} d u) (\Lambda^{n-1} d F) - \frac{1}{(\sqrt{n})^{n-2} (\sqrt{n})^{n-2}} |d u|^{n-2} |d F|^{n-2} (d u) (d F) \ast \\
& = Q (\Lambda^{n-1} d u) (\Lambda^{n-1} d F) - \frac{1}{(\sqrt{n})^{n-2}} |d u|^{n-2} (d u) \ast (\Lambda^{n-1} d F) \\
& \qquad {} + \frac{1}{(\sqrt{n})^{n-2}} |d u|^{n-2} (d u) \ast (\Lambda^{n-1} d F) - \frac{1}{(\sqrt{n})^{n-2} (\sqrt{n})^{n-2}} |d u|^{n-2} |d F|^{n-2} (d u) (d F) \ast \\
& = (\smith u) (\Lambda^{n-1} d F) + \frac{1}{(\sqrt{n})^{n-2}} |d u|^{n-2} (d u) (\smith F).
\end{align*}
Thus we see explicitly that if $u$ is a Smith map and $F$ is also a Smith map (which for $F$ means it is an orientation preserving conformal diffeomorphism) then $u \circ F$ is also a Smith map.

\begin{rmk} \label{rmk:conf-inv-again}
Here is yet another way to see the conformal invariance of the Smith equation~\eqref{eq:smith-u}, which shows explicitly the necessity of the particular power of $|du|$ on the right hand side. Let $g$ be a metric on $\Sigma$ and let $\widetilde{g} = f^2 g$ be another metric in the same conformal class, where $f$ is a positive function. Let $\mathcal V_g \in \Lambda^n (T\Sigma)$ denote the volume form \emph{on tangent vectors} with respect to $g$. Then $\mathcal V_{\widetilde{g}} = f^{-n} \mathcal V_g$, and consequently the Hodge star on $\Lambda^k (T\Sigma)$ with respect to $\widetilde{g}$ is $\ast_{\widetilde{g}} = f^{2k-n} \ast_{g}$. Moreover, we also have $|du|^2_{\widetilde{g}} = f^{-2} |du|^2_{g}$. Putting these all together with $k = n-1$, we deduce that $|du|^{n-2}_{\widetilde{g}} \ast_{\widetilde{g}} = |du|^{n-2}_g \ast_{g}$ and thus the right hand side of~\eqref{eq:smith-u} is independent of the conformal class of $g$.
\end{rmk}

\subsection{Four classes of Smith maps} \label{sec:smith-examples}

In this section we consider the four classes of Smith maps, corresponding to the four types of vector cross product $Q$ on $M$ from Table~\ref{table:VCP}. We also consider the Smith equation~\eqref{eq:smith-u} explicitly in local coordinates for the four classes. In all these cases we have a smooth Smith map $u : (\Sigma, g, \ast, \vol) \to (M, h, Q, \alpha)$. Define $\Sigma^0 = \Sigma \setminus \crit_u$, where $\crit_u$ is the critical locus of $u$.

\begin{ex} \label{ex:smith-u-I}
Suppose that $Q$ is of type I from Table~\ref{table:VCP}. Then $Q = \ast_M$ is the Hodge star operator corresponding to the metric $h$ and an orientation on $M$. Furthermore, $\alpha_Q = \vol_M$ is the associated volume form. Since $n-1=m-1$, we have $n=m$. The Smith equation~\eqref{eq:smith-u} in this case becomes
\begin{equation*}
\ast_M (\Lambda^{n-1} d u) = \frac{1}{(\sqrt{n})^{n-2}} |d u|^{n-2} (d u) \ast_\Sigma.
\end{equation*}
This equation is equivalent to the pair of equations $u^* h = \frac{1}{n} |d u|^2 g$ and $u^* \vol_M = \frac{1}{(\sqrt{n})^n} |d u|^n \vol_\Sigma$. Thus in this case a Smith map $u : \Sigma \to M$ is an orientation preserving weakly conformal diffeomorphism. The image $u(\Sigma^{\circ})$ is an open submanifold of $M$. 
\end{ex}

\begin{ex} \label{ex:smith-u-II}
Suppose that $Q$ is of type II from Table~\ref{table:VCP}. Then necessarily $m$ is even and $Q = J_M$ is an orthogonal almost complex structure on $(M, h)$, which need not be integrable. Furthermore, $\alpha_Q = \omega_M$ is the associated K\"ahler form, which need not be parallel nor even closed. Since $n-1=1$, we have $n=2$. Thus $\ast = J_\Sigma$ is simultaneously of both Type I and Type II, where $J_\Sigma$ is an orthogonal almost complex structure on $(\Sigma, g)$. In particular $(\Sigma, g, \ast, \vol)$ is a complex $1$-dimensional K\"ahler manifold, as a $U(1)$-structure is always torsion-free. The Smith equation~\eqref{eq:smith-u} in this case becomes
\begin{equation} \label{eq:smith-u-II}
J_M (d u) = (d u) J_\Sigma.
\end{equation}
Note that since the Smith equation depends only on the conformal class of the metric $g$ on $\Sigma$, in this case it really depends only on the \emph{underlying Riemann surface} $(\Sigma, J_\Sigma)$. In particular, in this case a Smith map $u : \Sigma \to M$ is a $J$-holomorphic map. The image $u(\Sigma^{\circ})$ is a $J$-holomorphic curve in $M$, also called a $1$-dimensional almost complex submanifold. This Smith equation is the classical Cauchy--Riemann equation.
\end{ex}

We remark that if $n=2$ then Example~\ref{ex:smith-u-I} is a special case of Example~\ref{ex:smith-u-II} corresponding to $m=2$.

For our purposes, the two most important Smith maps are the exceptional cases corresponding to the vector cross products of type III and IV from Table~\ref{table:VCP}.

\begin{ex} \label{ex:smith-u-III}
Suppose that $Q$ is of type III from Table~\ref{table:VCP}. Then necessarily $n=3$, $m=7$ and $Q$ is a $2$-fold vector cross product corresponding to a $\G$-structure $\varphi = \alpha_Q$, which is the associated calibration $3$-form. This $\G$-structure need not be torsion-free. The Smith equation~\eqref{eq:smith-u} in this case becomes
\begin{equation} \label{eq:smith-u-III}
Q (\Lambda^{2} d u) = \frac{1}{\sqrt{3}} |d u| (d u) \ast.
\end{equation}
The image $u(\Sigma^{\circ})$ is an \textit{associative submanifold} of $M$ and this Smith equation is called the \textit{associative Smith equation}.
\end{ex}

\begin{ex} \label{ex:smith-u-IV}
Suppose that $Q$ is of type IV from Table~\ref{table:VCP}. Then necessarily $n=4$, $m=8$ and $Q$ is a $3$-fold vector cross product corresponding to a $\Spin{7}$-structure $\Phi = \alpha_Q$, which is the associated calibration $4$-form. This $\Spin{7}$-structure need not be torsion-free. The Smith equation~\eqref{eq:smith-u} in this case becomes
\begin{equation*}
Q (\Lambda^{3} d u) = \frac{1}{2} |d u|^2 (d u) \ast.
\end{equation*}
The image $u(\Sigma^{\circ})$ is a \textit{Cayley submanifold} of $M$ and this Smith equation is called the \textit{Cayley Smith equation}.
\end{ex}

In Section~\ref{sec:smith-relations} we discuss relations between different classes of Smith maps, which is nontrivial.

We now express these equations in local coordinates. Let $(x^1, \ldots, x^n)$ be local coordinates for $\Sigma$ and let $(y^1, \ldots, y^m)$ be local coordinates for $M$. Let $u : \Sigma \to M$ be a smooth map. In terms of these coordinates we have $y^a = y^a (x^1, \ldots, x^n)$ for $1 \leq a \leq m$. The section $d u$ of $T^* \Sigma \otimes u^* TM$ becomes
\begin{equation*}
(d u)_x = \frac{\partial u^a}{\partial x^i}\Big|_x d x^i|_x \otimes \frac{\partial}{\partial y^a}\Big|_{u(x)}.
\end{equation*}
This means that $(d u)_x \frac{\partial}{\partial x^i}\big|_x = \frac{\partial u^a}{\partial x^i}\big|_x \frac{\partial}{\partial y^a}\big|_{u(x)}$. Write
\begin{equation*}
\vol_\Sigma = \frac{1}{n!} \mu_{i_1 \cdots i_n} d x^{i_1} \wedge \cdots \wedge d x^{i_n} \qquad \text{and} \qquad \alpha_Q = \frac{1}{n!} \alpha_{b_1 \cdots b_n} d y^{b_1} \wedge \cdots \wedge d y^{b_n}
\end{equation*}
where $1 \leq i_k \leq n$ and $1 \leq b_k \leq m$ for all $k$. Then by~\eqref{eq:alphaP} the Hodge star $\ast$ on $\Lambda^{n-1} T\Sigma$ and the vector cross product $Q$ on $\Lambda^{n-1} TM$ are given by
\begin{align*}
\ast \Big( \frac{\partial}{\partial x^{i_1}} \wedge \cdots \wedge \frac{\partial}{\partial x^{i_{n-1}}} \Big) & = \mu_{i_1 \cdots i_{n-1} j} g^{jl} \frac{\partial}{\partial x^l}, \\
Q \Big( \frac{\partial}{\partial y^{b_1}} \wedge \cdots \wedge \frac{\partial}{\partial y^{b_{n-1}}} \Big) & = \alpha_{b_1 \cdots b_{n-1} c} h^{ca} \frac{\partial}{\partial y^a}.
\end{align*}
Then the equation~\eqref{eq:smith-u} becomes
\begin{equation*}
\alpha_{b_1 \cdots b_{n-1} c} \frac{\partial u^{b_1}}{\partial x^{i_1}} \cdots \frac{\partial u^{b_{n-1}}}{\partial x^{i_{n-1}}} h^{c a} = \frac{1}{(\sqrt{n})^{n-2}} |d u|^{n-2} \frac{\partial u^{a}}{\partial x^l} \mu_{i_1 \cdots i_{n-1} j} g^{jl}.
\end{equation*}
which simplifies to
\begin{equation} \label{eq:smith-coordinates}
\alpha_{b_1 \cdots b_{n-1} c} \frac{\partial u^{b_1}}{\partial x^{i_1}} \cdots \frac{\partial u^{b_{n-1}}}{\partial x^{i_{n-1}}} = \frac{1}{(\sqrt{n})^{n-2}} |d u|^{n-2} \frac{\partial u^{b}}{\partial x^l} \mu_{i_1 \cdots i_{n-1} j} g^{jl} h_{bc},
\end{equation}
where $|d u|$ is given in the coordinates by
\begin{equation} \label{eq:du-coordinates}
|d u|^2 = \frac{\partial u^a}{\partial x^i} \frac{\partial u^b}{\partial x^j} h_{ab} g^{ij}. 
\end{equation}

For the cases of type II, III, IV in Examples~\ref{ex:smith-u-II},~\ref{ex:smith-u-III},~\ref{ex:smith-u-IV}, respectively, equation~\eqref{eq:smith-coordinates} becomes
\begin{align*}
\omega_{ab} \frac{\partial u^a}{\partial x^i} & = \frac{\partial u^a}{\partial x^l} \mu_{i j} g^{jl} h_{ab}, & \text{(classical Cauchy--Riemann equation)}, \\
\varphi_{abc} \frac{\partial u^a}{\partial x^i} \frac{\partial u^b}{\partial x^j} & = \frac{1}{\sqrt{3}} |d u| \frac{\partial u^a}{\partial x^l} \mu_{i j k} g^{kl} h_{ac} & \text{(associative Smith equation)}, \\
\Phi_{abcd} \frac{\partial u^a}{\partial x^i} \frac{\partial u^b}{\partial x^j} \frac{\partial u^c}{\partial x^k} & = \frac{1}{2} |d u|^2 \frac{\partial u^a}{\partial x^p} \mu_{i j k l} g^{lp} h_{ad} & \text{(Cayley Smith equation)},
\end{align*}
where $\mu$ is the Riemannian volume form on $(\Sigma, g)$ and $|d u|$ is given by~\eqref{eq:du-coordinates}.

\subsection{Relations between classes of Smith maps} \label{sec:smith-relations}

Let us recall the following two product constructions of associative and Cayley submanifolds.

\textbf{(A.)} Let $(Y^6, J, \omega, \Upsilon)$ be a Calabi-Yau $3$-fold with K\"ahler form $\omega$ and holomorphic volume form $\Upsilon$. Define $M^7 = Y^6 \times S^1$ and
\begin{equation} \label{eq:varphi}
\varphi = \real(\Upsilon) + d\theta \wedge \omega
\end{equation}
where $\theta$ is the standard periodic ``coordinate'' on $S^1$. Then $\varphi$ is a torsion-free $\G$-structure on $M^7$. Moreover, if $\Sigma^2$ is a complex submanifold of $Y^6$ then $\Sigma^2 \times S^1$ is an associative submanifold of $M^7$.

\textbf{(B.)} Let $(Y^7, \varphi)$ be a torsion-free $\G$ manifold with associative $3$-form $\varphi$ and Hodge dual coassociative $4$-form $\psi = \ast_{\varphi} \varphi$. Define $M^8 = Y^7 \times S^1$ and $\Phi = d\theta \wedge \varphi + \psi$ where $\theta$ is as before. Then $\Phi$ is a torsion-free $\Spin{7}$-structure on $M^8$. Moreover, if $\Sigma^3$ is an associative submanifold of $Y^7$, then $\Sigma^3 \times S^1$ is a Cayley submanifold of $M^8$.

In light of the relationships \textbf{(A.)} and \textbf{(B.)} between complex, associative, and Cayley submanifolds, it is natural to ask whether there exist analogous relationships between Smith maps of types II, III, and IV. In this section, we focus on the relationship between Smith maps of types II and III. We use this discussion in \S\ref{sec:uv-stuff} to clarify the (non-)relation between bubbling of $J$-holomorphic curves and bubbling of associative Smith maps.

To begin, let $v: \Sigma^2 \to Y^6$ be a Smith map of type II as in Example~\ref{ex:smith-u-II}, where $\Sigma^2$ is a Riemann surface equipped with a conformal class of metrics $[g_2]$ and volume form $\vol_2$, and where $Y^6$ is a Calabi-Yau $3$-fold with Riemannian metric $h_6$ and data $(J, \omega, \Upsilon)$ as above. By Proposition~\ref{prop:smith-conf-calib-maps}, we have
\begin{equation} \label{eq:v-mu}
v^* h_6 = \mu^2 g_2, \qquad v^* \omega = \mu^2 \vol_2.
\end{equation}
where we are writing
\begin{equation} \label{eq:v-mu-defn}
\mu^2 = \tfrac{1}{2} |dv|^2
\end{equation}
for ease of notation, and where $|dv|^2$ is computed with respect to $g_2$ on $\Sigma$ and $h_6$ on $Y$. We remark that by the first equation in~\eqref{eq:v-mu}, or by direct computation, the expression $|dv|^2 g_2$ is independent of the choice of representative metric in the conformal class $[g_2]$.

Next, equip $M^7 = Y^6 \times S^1$ with the $\G$-structure $\varphi = \real(\Upsilon) + d\theta \wedge \omega$ as above. The induced Riemannian metric $h_7$ on $M^7$ is simply the product metric
\begin{equation} \label{eq:h7}
h_7 = h_6 + (d\theta)^2
\end{equation}
where $\theta$ denotes the angle coordinate in $Y^6 \times S^1$. Since the image $v(\Sigma)$ is a complex curve in $Y$, it follows that $v(\Sigma) \times S^1$ is an associative submanifold of $M^7 = Y^6 \times S^1$.

Finally, we equip the $3$-manifold $\Sigma \times S^1$ with a Riemannian metric $g_3$, and let $\vol_3$ denote the corresponding volume form. We let $\phi$ denote the angle coordinate for the $S^1$ factor of $\Sigma \times S^1$.

We may now consider maps $u : \Sigma \times S^1 \to Y \times S^1$ of the form
\begin{equation} \label{eq:uvdefn}
u(x, \phi) = ( v(x), f(x, \phi) )
\end{equation}
where $f : \Sigma \times S^1 \to S^1$ is a smooth function $\theta = f(x, \phi)$. For simplicity, we suppose that
\begin{equation*}
f' := \frac{\partial f}{\partial \phi}
\end{equation*}
nonvanishing, so that $u$ as defined in~\eqref{eq:uvdefn} is an immersion whenever $v$ is an immersion. The particular form~\eqref{eq:uvdefn} of $u$ is motivated by the desire to have the image $u(\Sigma \times S^1)$ be an open subset of the associative submanifold $v(\Sigma) \times S^1$. Note that the choice $f(z,\phi) = \phi$ is permissible, which corresponds to $u = v \times \mathrm{Id}_{S^1}$, but for later use in \S\ref{sec:uv-stuff} we need to consider the general form in~\eqref{eq:uvdefn}. See Remark~\ref{rmk:fn-nontrivial}.

We aim to understand when such maps $u$ are associative Smith. To this end, note that using~\eqref{eq:h7}, followed by the first part of~\eqref{eq:v-mu}, we see that
\begin{equation} \label{eq:metric-condition}
u^* h_7 = u^* (h_6 + (d\theta)^2) = v^* h_6 + (u^* d\theta)^2 = \mu^2 g_2 + (d (u^* \theta))^2 = \mu^2 g_2 + (df)^2.
\end{equation}
Similarly, using~\eqref{eq:varphi}, followed by the second part of~\eqref{eq:v-mu}, we see that
\begin{align} \nonumber
u^* \varphi = u^* ( \real(\Upsilon) + d\theta \wedge \omega) & = v^* \real(\Upsilon) + (d (u^* \theta)) \wedge v^* \omega \\ \nonumber
& = 0 + df \wedge (\mu^2 \vol_2) \\ \label{eq:vol-condition}
& = \mu^2 f' d\phi \wedge \vol_2
\end{align}
using that $df = f'd\phi + d_{\Sigma} f$ and $(d_{\Sigma} f) \wedge \vol_2 = 0$. Finally, to streamline notation, we write
\begin{equation} \label{eq:u-lambda-defn}
\lambda^2 = \tfrac{1}{3} |du|^2
\end{equation}
where $|du|^2$ is computed with respect to $g_3$ on $\Sigma \times S^1$ and $h_7$ on $Y \times S^1$. 

\begin{prop} \label{prop:g3-vol3}
The map $u$ of~\eqref{eq:uvdefn} is an associative Smith map if and only if the metric $g_3$ and volume form $\vol_3$ on $\Sigma \times S^1$ satisfy
\begin{equation} \label{eq:g3}
\lambda^2 g_3 = \mu^2 g_2 + (df)^2
\end{equation}
and
\begin{equation} \label{eq:vol3}
\lambda^3 \vol_3 = \mu^2 f' d\phi \wedge \vol_2.
\end{equation}
\end{prop}

\begin{proof} By Proposition~\ref{prop:smith-conf-calib-maps}, the map $u$ is associative Smith if and only if it satisfies
\begin{equation} \label{eq:u-lambda}
u^* h_7 = \lambda^2 g_3, \qquad u^* \varphi = \lambda^3 \vol_3.
\end{equation}
The result now follows from \eqref{eq:metric-condition} and \eqref{eq:vol-condition}.
\end{proof}

\subsection{The fundamental energy identity for Smith maps} \label{sec:energy-identity}

In this section we establish the fundamental energy identity for Smith maps, when $\Sigma$ is \textit{compact} and the calibration form $\alpha = \alpha_Q$ on $M$ is \textit{closed}. We explain why $d \alpha = 0$ is necessary for the energy identity, which says that the $L^{n}$-energy of a Smith map is in some sense a topological invariant when $\partial \Sigma = \varnothing$.

The energy identity, which is our Theorem~\ref{thm:energy-identity} below, originally appeared in the unpublished preprint~\cite[Proposition 6.5]{Sm}. The way we present it here, the proof is almost immediate due to our parcelling out the preliminary linear algebraic results in $\S$\ref{sec:linear}.

Let $\Sigma^n$ and $M^m$ be as in $\S$\ref{sec:smith-maps-defn}, and let $u$ be any smooth map. Further, throughout this section we suppose that $\Sigma$ is \textit{compact}, so that we may integrate over $\Sigma$.

\begin{defn} \label{defn:n-energy}
The \textit{$n$-energy} of $u$, denoted $E_n (u)$, is defined to be
\begin{equation*}
E_n (u) = \frac{1}{(\sqrt{n})^n} \int_\Sigma |d u|^n \vol.
\end{equation*}
Thus, up to a factor which is chosen for later convenience, we observe that $E_n (u)$ is essentially $\| d u \|_n^n$, where $\| d u \|_n$ is the $L^n$-norm of $u$.
\end{defn}

\begin{thm}[The energy identity for Smith maps] \label{thm:energy-identity}
When the $n$-energy $E_n$ is written as
\begin{equation*}
E_n (u) = \int_\Sigma \left( \frac{1}{(\sqrt{n})^n} |d u|^n \vol - u^* \alpha \right) + \int_\Sigma u^* \alpha,
\end{equation*}
then the first term is \textit{always nonegative} and vanishes if and only if $u$ is a Smith map. 

Moreover, if $d \alpha = 0$ and $\partial \Sigma = \varnothing$, then the second term is a topological invariant, depending only on the cohomology class $[\alpha] \in H^n (M, \R)$ of $\alpha$ in $M$ and on the homology class $u_* [\Sigma] \in H_n(M)$, which is the image of the fundamental class $[\Sigma] \in H_n (\Sigma)$ by the pushforward homomorphism $u_* : H_n (\Sigma) \to H_n (M)$ induced by the map $u : \Sigma \to M$.

Consequently, when $d \alpha = 0$, a smooth map $u : \Sigma \to M$ is a Smith map if and only if its $n$-energy is given by
\begin{equation*}
E_n (u) = \int_\Sigma u^* \alpha = [\alpha] \cdot u_* [\Sigma].
\end{equation*}
\end{thm}
\begin{proof}
The generalized calibration inequality Theorem~\ref{thm:generalized-calib} applied to $A = d u$ says that
\begin{equation}\label{eq:generalized-calib-2}
\frac{1}{(\sqrt{n})^n} |d u|^{n} \vol - u^* \alpha \geq 0
\end{equation}
with equality if and only if $u : \Sigma \to M$ is a Smith map. If $d \alpha = 0$, then by Stokes's Theorem, $\int_\Sigma u^* \alpha = [\alpha] \cdot u_* [\Sigma]$ depends only on the cohomology class $[\alpha] \in H^n(M, \R)$.
\end{proof}

\begin{rmk} \label{rmk:energy-works-sobolev}
The first part of Theorem~\ref{thm:energy-identity} actually holds for maps $u \in W^{1, n}(\Sigma; M)$ since in this case the inequality~\eqref{eq:generalized-calib-2} holds almost everywhere on $\Sigma$ and vanishes almost everywhere if it integrates to zero. Recall that a map $u \in W^{1, n}(\Sigma; M)$ is said to be a Smith map if its weak derivative $du$ satisfies the Smith equation almost everywhere on $\Sigma$.
\end{rmk}

\begin{cor} \label{cor:energy-identity}
Let $d \alpha = 0$. If a Smith map $u : \Sigma \to M$ exists, then it is an absolute minimizer of the $n$-energy $E_n$ amongst all smooth maps $v : \Sigma \to M$ representing the same homology class in $H_n (M)$ as $u$.
\end{cor}
\begin{proof}
This is immediate from Theorem~\ref{thm:energy-identity}.
\end{proof}

Of course, when $n=2$, so that we are in the case of Type II from Table~\ref{table:VCP}, then this result is classical in the theory of $J$-holomorphic maps, and can be found, for example, in McDuff--Salamon~\cite{MS}.

\begin{rmk} \label{rmk:smith-energy}
A consequence of Theorem~\ref{thm:energy-identity} and Proposition~\ref{prop:smith-conf-calib} is that if $u : \Sigma \to M$ is a Smith map with no critical points, then the $n$-energy $E_n (u)$ equals the volume of $\Sigma$ with respect to the metric $u^* h = \frac{1}{n} | d u |^2 g$ on $\Sigma$. In particular, if $u$ is a Gray map with no critical points, then $E_n (u)$ equals the volume of $(\Sigma, g)$. (See also Proposition~\ref{conformalcalibrated} and Remark~\ref{conformalcalibrated-rmk} for the geometric measure theory analogue of this in the setting of currents.)
\end{rmk}

\subsection{Smith maps and $n$-harmonic maps} \label{sec:smith-n-harmonic}

The theory of $p$-harmonic maps has been extensively studied. Some possible references for this theory (certainly not exhaustive) include~\cite{DF, F, Ha-Li, MY, NVV, T}. Let $(\Sigma, g)$ and $(M, h)$ be Riemannian manifolds. A smooth map $u : \Sigma \to M$ is $p$-harmonic if is satisfies
\begin{equation} \label{eq:n-harmonic}
\Div ( |d u|^{p-2} d u ) = 0,
\end{equation}
where $\Div$ is the Riemannian divergence on $(M, g)$ taking a section of $T^*\Sigma \otimes u^* TM$ to a section of $u^* TM$. When $\Sigma$ is compact, the $p$-harmonic maps are the critical points of the $p$-energy $\int_\Sigma |d u|^p \vol$.

Suppose that $(\Sigma^n, g, \ast, \vol)$ and $(M^m, h, Q, \alpha_Q)$ are as in $\S$\ref{sec:smith-maps-defn} and suppose that $d \alpha = 0$, where we write $\alpha$ for $\alpha_Q$. Then by Corollary~\ref{cor:energy-identity}, since a Smith map $u : \Sigma \to M$ is an absolute minimizer of the $n$-energy in its homology class, it is certainly a critical point, and thus it is an $n$-harmonic map. (Here $p = n = \dim \Sigma$.)

In this section we show explicitly that Smith maps satisfy~\eqref{eq:n-harmonic} when $d \alpha = 0$. This argument is interesting because we cannot just differentiate the Smith equation to obtain the $n$-harmonic map equation. Rather, we also need to use the fact that a Smith map is conformally calibrating.

\begin{lemma} \label{lemma:smith-n-harmonic}
Let $Q$ be a $k$-fold vector cross product on $(M,h)$ with associated calibration $(k+1)$-form $\alpha$. Let $\nabla$ be the Levi-Civita connection of $h$. Then we have
\begin{align} \label{eq:lemma-VCP-Y-1}
(\nabla_V \alpha) (W_1, \ldots, W_{k+1}) & = h \big( (\nabla_V Q)(W_1, \ldots, W_k), W_{k+1} \big), \\ \label{eq:lemma-VCP-Y-2}
h \big( (\nabla_V Q)(W_1, \ldots, W_k), Q(W_1, \ldots, W_k) \big) & = 0,
\end{align}
for all smooth vector fields $V, W_1, \ldots, W_{k+1}$ on $M$.
\end{lemma}
\begin{proof}
The metric $h$ is parallel with respect to $\nabla$. Equation~\eqref{eq:lemma-VCP-Y-1} is a consequence of applying $\nabla_V$ to both sides of
\begin{equation*}
\alpha(W_1, \ldots, W_{k+1}) = h\big( Q(W_1, \ldots, W_k), W_{k+1} \big),
\end{equation*}
which is the definition of $\alpha$ from $Q$ as in~\eqref{eq:alphaP}. Similarly~\eqref{eq:lemma-VCP-Y-2} is obtained by applying $\nabla_V$ to both sides of
\begin{equation*}
h \big( Q(W_1, \ldots, W_k), Q(W_1, \ldots, W_k) \big) = h (W_1 \wedge \cdots \wedge W_k, W_1 \wedge \cdots \wedge W_k),
\end{equation*}
which is the fundamental property~\eqref{eq:Pmetric} of a vector cross product.
\end{proof}

\begin{cor} \label{cor:smith-n-harmonic}
Let $u : \Sigma \to M$ be a smooth Smith map as in $\S$\ref{sec:smith-maps-defn}. Then we have $u^* (\nabla_V \alpha) = 0$ for any smooth vector field $V$ on $M$.
\end{cor}
\begin{proof}
Fix $x \in \Sigma$ and let $e_1, \ldots, e_n$ be an oriented orthonormal frame for $T_x \Sigma$. Since $u$ is a Smith map, it follows from Proposition~\ref{prop:smith-conf-calib} that $u_* e_1, \ldots, u_* e_n$ are all orthogonal and of the same length $\lambda = \frac{1}{\sqrt{n}} |(d u)_x|$, and that they span a calibrated subspace $L = u_* (T_x \Sigma)$ of $T_{u(x)} M$. Then the vectors $f_i = \frac{1}{\lambda} u_* e_i$ for $i = 1, \ldots, n$ form an oriented orthonormal basis of $L$. By Proposition~\ref{prop:VCP-calibration} we deduce that
\begin{equation*}
f_n = Q(f_1, \ldots, f_{n-1}).
\end{equation*}
Using the above equation as well as~\eqref{eq:lemma-VCP-Y-1} and~\eqref{eq:lemma-VCP-Y-2}, we compute
\begin{align*}
\lambda^{-n} (u^* \nabla_V \alpha)(e_1, \ldots, e_n) & = \lambda^{-n} (\nabla_V \alpha)(u_* e_1, \ldots, u_* e_n) \\
& = (\nabla_V \alpha)(f_1, \ldots, f_{n-1}, f_n) \\
& = h \big( (\nabla_V Q)(f_1, \ldots, f_{n-1}), f_n \big) \\
& = h \big( (\nabla_V Q)(f_1, \ldots, f_{n-1}), Q(f_1, \ldots, f_{n-1}) \big) \\
& = 0,
\end{align*}
which is equivalent to $u^* (\nabla_V \alpha) = 0$.
\end{proof}

Now we show that a Smith map is $n$-harmonic if $d \alpha = 0$. Because the Hodge star $\ast$ on $\Lambda^{n-1} T \Sigma$ is invertible, with $\ast^{-1} = (-1)^{n-1} \ast$, the Smith equation~\eqref{eq:smith-u} is equivalent to
\begin{equation*}
Q (\Lambda^{n-1} d u) \, \ast = (-1)^{n-1} \frac{1}{(\sqrt{n})^{n-2}} |d u|^{n-2} (d u).
\end{equation*}
Here both sides of the above equation are sections of $T^* \Sigma \otimes u^* TM$. Thus, in order to show that a Smith map satisfies the $n$-harmonic map equation~\eqref{eq:n-harmonic}, we need to show that $\Div (Q (\Lambda^{n-1} d u) \, \ast) = 0$.

\begin{prop} \label{prop:smith-n-harmonic}
If $d \alpha = 0$, then $\Div (Q (\Lambda^{n-1} d u) \, \ast) = 0$, so any Smith map is an $n$-harmonic map.
\end{prop}
\begin{proof}
First consider an arbitrary section $B$ of $T^* \Sigma \otimes u^* TM$. Using local coordinates as in the end of $\S$\ref{sec:smith-examples} we can write $B = B^a_j d x^j \otimes \big( \frac{\partial}{\partial y^a} \circ u \big)$, which means that
\begin{equation*}
B_x = B^a_j(x) \, d x^j|_x \otimes \frac{\partial}{\partial y^a} \Big|_{u(x)}.
\end{equation*}
Now we consider the particular case $B = Q (\Lambda^{n-1} d u) \, \ast$, where $\ast : T \Sigma \to \Lambda^{n-1} T \Sigma$. Writing $\vol = \frac{1}{n!} \mu_{i_1 \cdots \mu_{i_n}} d x^{i_1} \wedge \cdots \wedge d x^{i_n}$ and $\alpha = \frac{1}{n!} \alpha_{b_1 \cdots b_n} d y^{b_1} \wedge \cdots \wedge d y^{b_n}$, we can compute that
\begin{equation} \label{eq:Bspecial}
B^a_j = \frac{1}{(n-1)!} \mu_{i_1 \cdots i_{n-1} j} g^{i_1 l_1} \cdots g^{i_{n-1} l_{n-1}} \frac{\partial u^{b_1}}{\partial x^{l_1}} \cdots \frac{\partial u^{b_{n-1}}}{\partial x^{l_{n-1}}} \alpha_{b_1 \cdots b_{n-1} c} h^{ca}.
\end{equation}
The connection $\nabla$ on $T^* \Sigma \otimes u^* TM$ is the tensor product of the Levi-Civita connection of $g$ on $T^* \Sigma$ and the pullback by $u$ of the Levi-Civita connection of $h$ on $TM$. Using the fact that $g$, $h$, and $\vol$ are all parallel, applying $g^{ij} \nabla_i$ to~\eqref{eq:Bspecial} gives
\begin{align*}
(\Div B)^a & = g^{ij} (\nabla_i B)^a_j \\
& = \frac{1}{(n-1)!} \mu_{i_1 \cdots i_{n-1} j} g^{ij} g^{i_1 l_1} \cdots g^{i_{n-1} l_{n-1}} h^{ca} \nabla_i \Big( \frac{\partial u^{b_1}}{\partial x^{l_1}} \cdots \frac{\partial u^{b_{n-1}}}{\partial x^{l_{n-1}}} \alpha_{b_1 \cdots b_{n-1} c} \Big).
\end{align*}
If we choose $(x^1, \ldots, x^n)$ to be Riemannian normal coordinates for $g$ centred at $x \in \Sigma$, and $(y^1, \ldots, y^m)$ to be Riemannian normal coordinates for $h$ centred at $u(x) \in M$, then the covariant derivative $\nabla_i$ evaluated \textit{at the point $x$} is the same as the partial derivative $\frac{\partial}{\partial x^i}$ at $x$, since the Christoffel symbols vanish. Thus, at $x$, we have
\begin{align*}
(\Div B)^a & = \frac{1}{(n-1)!} \mu_{i_1 \cdots i_{n-1} j} g^{ij} g^{i_1 l_1} \cdots g^{i_{n-1} l_{n-1}} h^{ca} \Big( \sum_{k=1}^{n-1} \frac{\partial^2 u^{b_k}}{\partial x^{l_k} \partial x^i} \frac{\partial u^b_{i_1}}{\partial x^{l_1}} \cdots \widehat{\frac{\partial u^b_{i_k}}{\partial x^{l_k}}} \cdots \frac{\partial u^{b_{n-1}}}{\partial x^{l_{n-1}}} \alpha_{b_1 \cdots b_{n-1} c} \Big) \\
& \qquad {} + \frac{1}{(n-1)!} \mu_{i_1 \cdots i_{n-1} j} g^{ij} g^{i_1 l_1} \cdots g^{i_{n-1} l_{n-1}} h^{ca} \frac{\partial u^{b_1}}{\partial x^{l_1}} \cdots \frac{\partial u^{b_{n-1}}}{\partial x^{l_{n-1}}} \Big( \nabla_{\! \! \frac{\partial}{\partial x^i}} \alpha_{b_1 \cdots b_{n-1} c}\Big).
\end{align*}
Since $\frac{\partial^2 u^{b_k}}{\partial x^{l_k} \partial x^i}$ is symmetric in $l_k, i$ and $\mu_{i_1 \cdots i_{n-1} j}$ is skew in $i_k, j$, all the terms in the first line above vanish. For the covariant derivative of $\alpha$, we use $\nabla_{\! \! \frac{\partial}{\partial x^i}} = \frac{\partial u^{b_n}}{\partial x^{i}} \nabla_{\! \! \frac{\partial \, \, \, \, \, }{\partial y^{b_n}}}$, which we write as $ \frac{\partial u^{b_n}}{\partial x^{i}} \nabla_{b_n}$ for simplicity. We thus have that at the point $x$, 
\begin{equation*}
(\Div B)^a = \frac{1}{(n-1)!} \mu_{i_1 \cdots i_{n-1} j} g^{ij} g^{i_1 l_1} \cdots g^{i_{n-1} l_{n-1}} h^{ca} \frac{\partial u^{b_1}}{\partial x^{l_1}} \cdots \frac{\partial u^{b_{n-1}}}{\partial x^{l_{n-1}}} \frac{\partial u^{b_n}}{\partial x^i} \Big( \nabla_{b_n} \alpha_{b_1 \cdots b_{n-1} c} \Big).
\end{equation*}
Relabel $j \to i_n$ and $i \to l_n$. We have
\begin{equation*}
(\Div B)^a = \frac{1}{(n-1)!} \mu_{i_1 \cdots i_n} g^{i_1 l_1} \cdots g^{i_n l_n} h^{ca} \frac{\partial u^{b_1}}{\partial x^{l_1}} \cdots \frac{\partial u^{b_n}}{\partial x^{l_n}} \Big( \nabla_{b_n} \alpha_{b_1 \cdots b_{n-1} c} \Big).
\end{equation*}
It follows by the skew-symmetry of $\mu_{i_1 \cdots i_n}$ that if we interchange any two of $b_1, \ldots, b_n$ in the last factor $ \nabla_{b_n} \alpha_{b_1 \cdots b_{n-1} c}$, then the right hand side above will change sign. Consequently, we can write
\begin{equation} \label{eq:smith-n-harmonic-2}
(\Div B)^a = \frac{1}{(n-1)!} \mu_{i_1 \cdots i_n} g^{i_1 l_1} \cdots g^{i_n l_n} h^{ca} \frac{\partial u^{b_1}}{\partial x^{l_1}} \cdots \frac{\partial u^{b_n}}{\partial x^{l_n}} \Big( \frac{1}{n} \sum_{k=1}^n \nabla_{b_k} \alpha_{b_1 \cdots b_{k-1} c b_{k+1} \cdots b_n} \Big).
\end{equation}
Now we use the hypothesis that $d \alpha = 0$. Since $d$ is the skew-symmetrization of $\nabla$, we have
\begin{equation*}
0 = (d \alpha)_{cb_1 \cdots b_n} = \nabla_c \alpha_{b_1 \ldots b_n} - \sum_{k=1}^n \nabla_{b_k} \alpha_{b_1 \cdots b_{k-1} c b_{k+1} \cdots b_n}.
\end{equation*}
Substituting the above into~\eqref{eq:smith-n-harmonic-2} gives
\begin{equation} \label{eq:smith-n-harmonic-3}
(\Div B)^a = \frac{1}{n!} \mu_{i_1 \cdots i_n} g^{i_1 l_1} \cdots g^{i_n l_n} h^{ca} \frac{\partial u^{b_1}}{\partial x^{l_1}} \cdots \frac{\partial u^{b_n}}{\partial x^{l_n}} \nabla_{c} \alpha_{b_1 \cdots b_n}.
\end{equation}
We observe that Corollary~\ref{cor:smith-n-harmonic} with $V = \frac{\partial}{\partial y^c}$ evaluated on $W_1, \ldots, W_n$ with $W_k = \frac{\partial \, \, \, }{\partial x^{l_k}}$ says that
\begin{equation*}
0 = (u^* \nabla_c \alpha)\big( \frac{\partial \, \, \, }{\partial x^{l_1}}, \ldots, \frac{\partial \, \, \, }{\partial x^{l_n}} \big) = \frac{\partial u^{b_1}}{\partial x^{l_1}} \cdots \frac{\partial u^{b_n}}{\partial x^{l_n}} \nabla_{c} \alpha_{b_1 \cdots b_n}.
\end{equation*}
Substituting the above into~\eqref{eq:smith-n-harmonic-3} gives $\Div B = 0$, as claimed.
\end{proof}

\subsection{Smith maps via currents}

Most of this section can be read independently from the rest of the paper, with the exception of Corollary~\ref{energyminimizing}, which is used later. We give two proofs of Corollary~\ref{energyminimizing}, one using the geometric measure theory framework of this section, and another using the energy identity of Theorem~\ref{thm:energy-identity}. The reader who is not interested in geometric measure theory can skip this entire section with the exception of Corollary~\ref{energyminimizing} and its second proof.

Here we continue using the notation of $\S$\ref{sec:smith-maps-defn}, and we assume that the $n$-form $\alpha$ associated with the vector cross product $Q$ on $M$ is closed. Also, we emphasize that we allow $\Sigma$ to have \emph{nonempty boundary}.

Recall that in $\S$\ref{sec:smith-maps-defn} and $\S$\ref{sec:energy-identity} the following result is effectively proved.
\begin{prop}\label{prop:smith-manifold-summary}
Suppose $\Sigma$ is compact, possibly with boundary, and let $u: \Sigma \to M$ be a Lipschitz map. Then 
\begin{equation}\label{eq:energy-identity2}
\int_{\Sigma}u^{\ast}\alpha \leq \frac{1}{(\sqrt{n})^n}\int_{\Sigma}|du|^{n} \vol.
\end{equation}
Moreover, we have the following two equivalent characterizations for when equality holds:
\begin{enumerate}
\item[(a)] $u$ is a Smith map.
\item[(b)] $u$ is conformally calibrating; that is, $u^{\ast}h = \frac{1}{n}|du|^{2}g$ and $u^{\ast}\alpha = \frac{1}{(\sqrt{n})^{n}}|du|^{n}\vol$.
\end{enumerate}
\end{prop}

The main purpose of this section is to complete the inequality~\eqref{eq:energy-identity2} by showing that it is part of a longer string of inequalities which involves the volume of the image of $\Sigma$ under $u$, which has to be understood as a \emph{current} when $\crit_u$ is nonempty. In addition, we show that, in the equality case, the image current is calibrated by $\alpha$. Below we very briefly recall some relevant definitions from the theory of currents. For a general introduction to the subject, see for instance~\cite[Chapter 6]{Si}.

Following standard practice in geometric measure theory, we assume that $\Sigma$ and $M$ are isometrically embedded into $\R^{l}$ and $\R^{d}$, respectively, in which case the $n$-dimensional Hausdorff measure $H^{n}$ restricts to the volume measure $\vol$ on $\Sigma$. Next, to $\Sigma$ we associate an integral current $\current{\Sigma}$ given by 
\begin{equation*}
\current{\Sigma}(\omega) := \int_{\Sigma} \langle \omega_x, \xi_x \rangle dH^{n}(x) \quad \text{ for all }\omega \in \mathcal{D}^{n}(\R^{l}),
\end{equation*}
where $\mathcal{D}^{n}(\R^{l})$ denotes the space of all smooth compactly supported $n$-forms on $\R^{l}$, and $\xi_x \in \Lambda^{n}\R^{l}$ is the unit simple $n$-vector giving the orientation on $T_{x}\Sigma$ for each $x \in \Sigma$. (Here $T_{x}\Sigma$ is considered a subspace of $\R^{l}$ via the embedding $\Sigma \to \R^{l}$.)

Given a Lipschitz map $u: \Sigma \to M$, we regard it as a map into $\R^{d}$ and define the pushforward $u_{\#}\current{\Sigma}$ of $\current{\Sigma}$ by $u$, which is an integral current supported in $M$, by
\begin{equation*}
u_{\#}\current{\Sigma}(\eta) := \int_{\Sigma} \langle \eta_{u(x)}, \Lambda^{n}(du)_x \xi_x \rangle dH^{n}(x) \quad \text{ for all }\eta \in \mathcal{D}^{n}(\R^{d}).
\end{equation*}
Note that the right hand side is just $\current{\Sigma}(u^{\ast}\eta)$ when $u$ is smooth. Also, strictly speaking, on the right we should write $\overline{u}$ instead $u$, where $\overline{u}: \R^{l} \to \R^{d}$ is any compactly supported extension of $u$. (The choice of extension does not affect the definition of $u_{\#}\current{\Sigma}$.) 

The mass of the current $u_{\#}\current{\Sigma}$ is, by definition,
\begin{equation*}
\mathbf{M}(u_{\#}\current{\Sigma}) = \sup\left\{ u_{\#}\current{\Sigma}(\eta)\ \big|\ \eta \in \mathcal{D}^{n}(\R^{d}), \|\eta_y \|\leq 1 \text{ for all }y \in \R^{d} \right\},
\end{equation*}
where $\| \cdot \|$ denotes the comass norm of a covector. 

To pair the current $u_{\#}\current{\Sigma}$ with the calibrating $n$-form $\alpha$, we need to extend the latter to a compactly supported form on $\R^{d}$. Specifically, take $\delta_{0}$ sufficiently small so that the neighborhood
\begin{equation*}
\cN_{2\delta_{0}}(M) := \{x \in \R^{d}\ \rvert\ \dist_{\R^{d}}(x, M) < 2\delta_{0}\}
\end{equation*}
is strictly contained in a tubular neighborhood of $M$ in $\R^{d}$, and let $\pi: \cN_{2\delta_{0}}(M) \to M$ be the nearest-point projection. In addition, fix a cutoff function $\zeta$ which is identically $1$ on $\cN_{\frac{3\delta_{0}}{2}}(M)$ and vanishes outside of $\cN_{2\delta_{0}}(M)$. 
\begin{lemma} \label{extendG2}
The $n$-form $\tilde{\alpha}$ on $\R^{d}$ defined by $\tilde{\alpha} = \zeta \pi^{\ast}\alpha$ has the following properties.
\begin{enumerate}[(a)]
\item $d\tilde{\alpha} = 0$ on $\cN_{\frac{3\delta_{0}}{2}}(M)$.
\item For $y \in \R^{d}$, the comass norm of $\tilde{\alpha}_y$ satisfies $\|\tilde{\alpha}_y\| \leq 1$.
\item For $y \in M$ and $\tau \in \Lambda^{n}T_{y}M$, we have $\langle \tilde{\alpha}_y, \tau \rangle = \langle \alpha_y, \tau \rangle$, where on the left hand side we view $\tau$ as an element of $\Lambda^{n}\R^{d}$ via the embedding $M \to \R^{d}$.
\end{enumerate}
\end{lemma}
\begin{proof}
Statements (a) and (c) can be verified by direct computation. For (b), take a unit simple $n$-vector $\tau \in \bigwedge\nolimits^{n}\R^{d}$. Then by the definition of $\tilde{\alpha}$ we have
\begin{align*}
\left| \langle \tilde{\alpha}_y, \tau \rangle \right| & = \zeta(y) \left| \langle \alpha_{\pi(y)}, \Lambda^{n} (d\pi)_y \tau \rangle \right|\\
& \leq \|\alpha_{\pi(y)}\| \, | \Lambda^{n} (d\pi)_y \tau | \leq | \Lambda^{n} (d\pi)_y \tau |,
\end{align*}
where we used that fact that $\|\alpha_{\pi(y)}\| \leq 1$ in the last inequality. Therefore, taking the \textit{supremum} over all unit simple $n$-vectors $\tau$, we obtain
\begin{equation*}
\| \tilde{\alpha}_y \| \leq | \Lambda^{n} (d\pi)_y | \leq 1,
\end{equation*}
which implies the desired estimate.
\end{proof}

\begin{lemma} \label{pmepointwise} 
Let $u : \Sigma \to M$ be a Lipschitz map. For any $x \in \Sigma$ such that $(du)_x$ exists, and any $\eta \in \mathcal{D}^{n}(\R^{d})$ with $\|\eta_y\| \leq 1$ for all $y \in M$, we have
\begin{equation*}
\langle \eta_{u(x)}, \Lambda^{n} (du)_x \xi_x \rangle \leq | \Lambda^{n} (du)_x \xi_x | \leq \frac{1}{(\sqrt{n})^{n}}|(du)_x|^{n}.
\end{equation*}
\end{lemma}
\begin{proof}
This is immediate from the proof of the generalized calibration inequality, Theorem~\ref{thm:generalized-calib}.
\end{proof}

We now give the precise version of the main result of this section along with two corollaries.
\begin{prop} \label{conformalcalibrated}
Let $u : \Sigma \to M$ be a Lipschitz map. 
\begin{enumerate}[(a)]
\item We have
\begin{equation} \label{pmeinequality}
u_{\#}\current{\Sigma}(\tilde{\alpha}) \leq \mathbf{M}(u_{\#}\current{\Sigma}) \leq \int_{\Sigma}| \Lambda^{n} (du)_x \xi_x | dH^{n}(x) \leq \frac{1}{(\sqrt{n})^{n}}\int_{\Sigma}|du|^{n}dH^{n}.
\end{equation}
\item If $u$ is a Smith map, then all the inequalities in~\eqref{pmeinequality} become equalities. In particular, the integral current $u_{\#}\current{\Sigma}$ is calibrated by $\tilde{\alpha}$. That is, it is a positive $\tilde{\alpha}$-current in the sense of~\cite[Section II.4]{HL}.
\end{enumerate}
\end{prop}

\begin{rmk} \label{conformalcalibrated-rmk}
The third term in~\eqref{pmeinequality} coincides with the mass of the pushforward of $\Sigma$ as a varifold via the map $u$.
\end{rmk}

\begin{proof}[Proof of Proposition~\ref{conformalcalibrated}]
We begin with part (a). Applying Lemma~\ref{pmepointwise} to any $\eta \in \mathcal{D}^{n}(\R^{d})$ with $\| \eta_y \| \leq 1$ everywhere, integrating over $\Sigma$, and then taking the supremum over all such $\eta$, we see that
\begin{equation*}
\mathbf{M}(u_{\#}\current{\Sigma}) \leq \int_{\Sigma}| \Lambda^{n} (du)_x \xi_x |dH^{n}(x) \leq \frac{1}{(\sqrt{n})^{n}}\int_{\Sigma}|du|^{n}dH^{n}.
\end{equation*}
To finish the proof of (a), we simply note that $u_{\#}\current{\Sigma}(\tilde{\alpha}) \leq \mathbf{M}(u_{\#}\current{\Sigma})$ by the definition of $\mathbf{M}(u_{\#}\current{\Sigma})$, because $\|\tilde\alpha_y\| \leq 1$ for all $y \in \R^{d}$.

The first claim of part (b) follows directly from Proposition~\ref{prop:smith-manifold-summary}. The second part holds because any integral current $T$ supported in $M$ and satisfying $T(\tilde{\alpha}) = \mathbf{M}(T)$ must be a positive $\tilde{\alpha}$-current.
\end{proof}

\begin{cor} \label{minimizingproperty}
Let $A, B$ be compact subsets of $\Sigma$ with smooth boundary and let $u: A \to M$ and $v : B \to M$ be Lipschitz maps, with $u$ being a Smith map. Suppose further that $u_{\#}\current{A} + v_{\#}\current{B} = \partial S$ for some $(n + 1)$-current supported in $M$. Then
\begin{equation*}
\int_{A}|du|^{n}dH^{n} \leq \int_{B}|dv|^{n}dH^{n}.
\end{equation*}
\end{cor}
\begin{proof}[Proof of Corollary~\ref{minimizingproperty}]
By Proposition~\ref{conformalcalibrated}(b), the current $u_{\#}\current{A}$ is calibrated by $\tilde{\alpha}$ and hence is mass-minimizing among all currents homologous to it. Since $-v_{\#}\current{B}$ is homologous to $u_{\#}\current{A}$ by assumption, we get 
\begin{equation*}
\mathbf{M}(u_{\#}\current{A}) \leq \mathbf{M}(-v_{\#}\current{B}) = \mathbf{M}(v_{\#}\current{B}) \leq \frac{1}{(\sqrt{n})^{n}}\int_{B}|dv|^{n}dH^{n},
\end{equation*}
where the last inequality follows from~\eqref{pmeinequality}. We complete the proof by noting that 
\begin{equation*}
\mathbf{M}(u_{\#}\current{A}) = \frac{1}{(\sqrt{n})^{n}}\int_{A}|du|^{n}dH^{n},
\end{equation*}
because of Proposition~\ref{conformalcalibrated}(b).
\end{proof}

The following result is used in $\S$\ref{sec:energy-annuli} to derive an estimate crucial to the proofs of Theorems~\ref{thm:no-energy-loss} and~\ref{thm:zero-neck-length}. As mentioned at the beginning of this section, we give two proofs, the second of which depends only on Theorem~\ref{thm:energy-identity} and does not use the geometric measure theory framework above.

\begin{cor} \label{energyminimizing}
Suppose $\Sigma = \partial W$ where $W$ is a compact oriented Riemannian $(n + 1)$-manifold isometrically embedded in $\R^{l}$, and let $u: A \to M$ be a Lipschitz Smith map on a compact set $A \subseteq \Sigma$ with smooth boundary. Suppose moreover that $u$ has a Lipschitz extension $f: \Sigma \to M$, which extends further to a Lipschitz map $F : W \to M$. Then we have 
\begin{equation*}
\frac{1}{(\sqrt{n})^{n}}\int_{A}|du|^{n}dH^{n} \leq \frac{1}{(\sqrt{n})^{n}}\int_{\Sigma\setminus A} |df|^{n}dH^{n}.
\end{equation*}
\end{cor}
\begin{proof}[Proof of Corollary~\ref{energyminimizing}]
This follows at once from Corollary~\ref{minimizingproperty} with $A$ being the same, $B = \overline{\Sigma\setminus A}$, and $v = f\rvert_{B}$, and $S = F_{\#}\current{W}$. (Note that $\partial S = f_{\#}\current{\Sigma} = u_{\#}\current{A} + v_{\#}\current{B}$.)

We also give a second proof, using the energy identity of Theorem~\ref{thm:energy-identity}. By the energy identity applied to the calibration $-\alpha$ and the map $f$ with domain $\Sigma \setminus A$, we have 
\begin{equation*}
\frac{1}{(\sqrt{n})^{n}}\int_{\Sigma\setminus A}|df|^{n}dH^{n} \geq -\int_{\Sigma \setminus A}f^{\ast}\alpha.
\end{equation*}
On the other hand, recalling that $f\rvert_{A} = u$ is a Smith map and using the equality case of the energy identity (with the calibration $\alpha$), we get
\begin{equation*}
\frac{1}{(\sqrt{n})^{n}}\int_{A}|df|^{n}dH^{n} = \int_{A}f^{\ast}\alpha = \int_{\Sigma}f^{\ast}\alpha - \int_{\Sigma \setminus A}f^{\ast}\alpha.
\end{equation*}
Combining the two relations above, we arrive at
\begin{equation*}
\frac{1}{(\sqrt{n})^{n}}\int_{A}|du|^{n}dH^{n} \leq \int_{\Sigma}f^{\ast}\alpha + \frac{1}{(\sqrt{n})^{n}}\int_{\Sigma \setminus A}|df|^{n}dH^{n}.
\end{equation*}
To finish, it suffices to prove that $\int_{\Sigma}f^{\ast}\alpha = 0$. For that, note that by Stokes's theorem we have
\begin{equation*}
\int_{\Sigma}f^{\ast}\alpha = \int_{\Sigma}F^{\ast}\alpha = \int_{W}F^{\ast}(d\alpha) = 0,
\end{equation*}
as claimed.
\end{proof}

\section{Analytic aspects of Smith maps} \label{sec:int-reg-cpt}

\textbf{Note.} Beginning in this section and for the remainder of the paper, for concreteness we restrict to the case of associative Smith maps $u \colon (\Sigma^3, g, \ast, \mu) \to (M^7, h, J, \varphi)$, although all of our results also apply to Cayley Smith maps after making the obvious modifications. 

Thus, $(\Sigma, g)$ is an oriented Riemannian $3$-manifold with volume form $\mu$ and Hodge star operator $\ast$, and $(M, \varphi)$ is a $7$-manifold with a $\G$-structure $\varphi$ inducing a Riemannian metric $h$ and an associated $2$-fold vector cross product which we denote by $J$ in analogy with almost complex structures. In addition, \emph{recall that $M$ is compact without boundary and that it is isometrically embedded into $(\R^{d}, g_{\euc})$.}

The present section is divided into five parts. In $\S$\ref{sec:smith-eq} we derive an explicit useful form of the Smith equation in local coordinates. In $\S$\ref{sec:int-reg-sub} we prove the $\ep$-regularity theorem, which is Theorem~\ref{thm:ep-reg}. In $\S$\ref{sec:mean-value-r-s} we establish a mean value inequality, interior regularity, and a removable singularity theorem. $\S$\ref{sec:basic-convergence} is devoted to the basic convergence result for a sequence of Smith maps with uniformly bounded $3$-energy, which gives $C^{1}$-subsequential convergence away from a set of isolated points and is important for the bubble tree construction in $\S$\ref{sec:bubble-tree}. Finally in $\S$\ref{sec:energy-lower-bound} we mention two well-known results giving positive lower bounds for the $3$-energy of maps from $S^{3}$ which are ``nontrivial'' in some sense.

\subsection{The Smith equation} \label{sec:smith-eq}

In this section and the next, Greek indices run from $1$ to $3$ and Latin indices run from $1$ to $d$, where $M$ is isometrically embedded into $(\R^d, g_{\euc})$. In particular we have global coordinates $y^1, \ldots, y^d$ on $\R^d$ and the vector fields $\paop{y^{1}}, \ldots, \paop{y^{d}}$ are orthonormal. We extend the vector cross product $J$ on $M$ to $\R^{d}$ as follows. Let $\delta_{0}$ be chosen such that
\begin{equation*}
\cN_{2\delta_{0}}(M) : = \{x \in \R^{d} \ \rvert\ \dist_{\R^{d}}(x, M) < 2\delta_{0}\}
\end{equation*}
is strictly contained in a tubular neighborhood of $M$ in $\R^{d}$, and let $\pi$ denote the nearest-point projection onto $M$. Moreover, let $\zeta$ be a cutoff function which is identically equal to $1$ on $\cN_{\frac{3\delta_{0}}{2}}(M)$ and vanishes outside of $\cN_{2\delta_{0}}(M)$. Then we define
\begin{equation*}
\tilde{J}_y (X, Y) = \begin{cases}
\zeta(y) J_{\pi(y)} \big( (d\pi)_{y} X, (d\pi)_{y} Y \big) & \text{if } y \in \cN_{2\delta_{0}}(M), \\
0 & \text{otherwise}.
\end{cases}
\end{equation*}
Note that $\tilde{J}_y$ is still bilinear and skew-symmetric for all $y \in \R^{d}$, and moreover 
\begin{equation*}
\tilde{J}_y (X, Y) = J_y (X, Y) \text{ whenever }y \in M\ \text{and }X, Y \in T_{y}M.
\end{equation*}
Note that $\tilde{J}$ is \emph{not} a vector cross product on $\R^d$. However, this does not affect any of the results where $\tilde{J}$ is used, as only the bilinearity and skew-symmetry are required whenever we use $\tilde{J}$.

We write $\tilde{J}$ in coordinates as
\begin{equation*}
J^{i}_{jk}(y) \paop{y^{i}} = \tilde{J}_y \Big( \paop{y^{j}}, \paop{y^{k}} \Big).
\end{equation*}
Note that we have dropped the tilde in the notation for the components.

In addition, because all of the results in $\S$\ref{sec:int-reg-sub} and $\S$\ref{sec:mean-value-r-s} are local, in these two sections we assume the domain is the ball $B(2) \subseteq \R^{3}$ equipped with a Riemannian metric $g = (g_{\alpha\beta})$. Moreover we assume $B(2)$ is given the standard orientation, so that $\partial_{1}, \partial_{2}, \partial_{3}$ is a positive basis everywhere. For a map $u: B(2) \to \R^{d}$ we write the components of its differential as
\begin{equation*}
u_{\alpha}^{i} \paop{y^{i}} = \pa{u}{x^{\alpha}} = : u_{\alpha}.
\end{equation*}

We need to rewrite the Smith equation~\eqref{eq:smith-u} in terms of the components of $\tilde{J}$ and $\mu$. It is more convenient to precompose both sides of~\eqref{eq:smith-u} with $\ast$. Thus we have (recall $n=3$ from now on) that
\begin{equation} \label{eq:smith-assoc}
\frac{1}{\sqrt{3}} |du| du = J (\Lambda^2 du) \ast.
\end{equation}
If we write $\partial_{\lambda}$ for $\paop{x^{\lambda}}$, then $\ast \partial_{\lambda} = \frac{1}{2}C_{\lambda}^{\alpha \beta} \partial_{\alpha} \wedge \partial_{\beta}$ for some functions $C_{\lambda}^{\alpha \beta} = - C_{\lambda}^{\beta \alpha}$. Using that $\ast$ is an isometry and $\ast^2 = 1$ in three dimensions, we find that
\begin{align*}
\mu_{\gamma \delta \lambda} & = g( \ast (\partial_{\gamma} \wedge \partial_{\delta}), \partial_{\lambda}) = g( \partial_{\gamma} \wedge \partial_{\delta}, \ast \partial_{\lambda}) \\
& = \tfrac{1}{2} C_{\lambda}^{\alpha \beta} g( \partial_{\gamma} \wedge \partial_{\delta}, \partial_{\alpha} \wedge \partial_{\beta} ) = \tfrac{1}{2} C_{\lambda}^{\alpha \beta} ( g_{\gamma \alpha} g_{\delta \beta} - g_{\gamma \beta} g_{\delta \alpha}) \\
& = C_{\lambda}^{\alpha \beta} g_{\gamma \alpha} g_{\delta \beta}.
\end{align*}
Thus we conclude that $C_{\lambda}^{\alpha \beta} = \mu_{\lambda \gamma \delta} g^{\gamma \alpha} g^{\delta \beta} = \mu^{\alpha \beta \gamma} g_{\alpha \lambda}$. The second expression is preferable, since $\mu^{\alpha \beta \gamma}$ are the components of the $3$-vector that is metric dual to the Riemannian volume form $\mu$. Thus
\begin{equation*}
\mu^{\alpha \beta \gamma} = \frac{1}{\sqrt{g}} \perm^{\alpha \beta \gamma},
\end{equation*}
where $\sqrt{g}$ denotes $\sqrt{\det ( g_{\alpha \beta} )}$ and $\perm^{\alpha \beta \gamma}$ is the permutation symbol on three letters. That is, $\perm^{\sigma(1) \sigma(2) \sigma(3)} = \sgn(\sigma)$ for $\sigma \in S_3$. Thus
\begin{equation} \label{eq:ast-coords}
\ast \partial_{\lambda} = \frac{1}{2 \sqrt{g}} \perm^{\alpha \beta \gamma} g_{\lambda \gamma} \partial_{\alpha} \wedge \partial_{\beta}.
\end{equation}

Using~\eqref{eq:ast-coords} and the notation defined above, equation~\eqref{eq:smith-assoc} becomes
\begin{align*}
\frac{1}{\sqrt{3}}|du| u^{i}_{\lambda} \paop{y^{i}} & = \left[J\circ (du \wedge du)\circ \ast \right](\partial_{\lambda}) \\
& = \frac{1}{2 \sqrt{g} } \left[ J \circ (du \wedge du) \right] (\perm^{\alpha \beta \gamma} g_{\lambda \gamma} \partial_{\alpha} \wedge \partial_{\beta}) \\
& = \frac{1}{2 \sqrt{g}} \perm^{\alpha \beta \gamma} g_{\lambda \gamma} J(u_{\alpha}, u_{\beta}) \\
& = \frac{1}{2 \sqrt{g} } \perm^{\alpha \beta \gamma} g_{\lambda \gamma} (J^{i}_{jk}\circ u) u_{\alpha}^{j} u_{\beta}^{k} \paop{y^{i}}.
\end{align*}
Finally, equating the $\paop{y^{i}}$ components of each side and skew-symmetrizing in $\alpha, \beta$, we conclude that
\begin{equation} \label{eq:smith-local-2}
\frac{1}{\sqrt{3}}|du| u^{i}_{\lambda} = \frac{1}{4 \sqrt{g} } \perm^{\alpha \beta \gamma} g_{\lambda \gamma} (J^{i}_{jk}\circ u) ( u_{\alpha}^{j} u_{\beta}^{k} - u_{\beta}^{j} u_{\alpha}^{k} ),
\end{equation}
which is the form of the Smith equation that we require in the next section.

\subsection{The $\ep$-regularity theorem} \label{sec:int-reg-sub}

The main result of this section, namely Theorem~\ref{thm:ep-reg}, is an $\ep$-regularity theorem in the spirit of Sacks--Uhlenbeck~\cite{SU}. The proof is accomplished in a number of stages. First, in Proposition~\ref{2ndorderSmith} we derive from the Smith equation~\eqref{eq:smith-local-2} a system of second order elliptic equations that resembles the $n$-harmonic map system. Then, combining the work of Uhlenbeck~\cite{U} and some results from harmonic analysis, which are collected in Appendix~\ref{sec:harmonic-analysis-appendix} (see also Remarks~\ref{hardyspacermk} and~\ref{rmk:compensation}), we prove a decay estimate on the $3$-energy under a smallness assumption in Proposition~\ref{energydecay}. From there and using standard theory, we derive $C^{0, \alpha}$-regularity for any $\alpha \in (0, 1)$ in Proposition~\ref{holderregularity}. Finally, choosing $\alpha$ sufficiently close to $1$ and adapting an argument from~\cite{DM}, which is itself based on a refinement~\cite{GiMa} of Uhlenbeck's work, allows us to prove Theorem~\ref{thm:ep-reg}.

We point out that the system~\eqref{2ndorderSmithequation2} we derive in Proposition~\ref{2ndorderSmith} belongs to the class of $n$-harmonic type systems studied by Mou--Wang in~\cite{MW}. Using the same harmonic analysis tools that we use below, but rather differently, they obtained a H\"older continuity result similar to Corollary~\ref{holderregularity} for some $\alpha \in (0, 1)$. They also obtained analogues of Theorem~\ref{thm:int-reg} (the first conclusion) and Proposition~\ref{prop:energy-gap} using arguments different from ours.

\begin{prop} \label{2ndorderSmith}
Let $u \in W^{1, 3}(B(2); M)$ be an associative Smith map. That is, $u$ satisfies the Smith equation~\eqref{eq:smith-local-2} almost everywhere on $B(2)$. Then the following hold:
\begin{enumerate}[(a)]
\item We have
\begin{equation} \label{2ndorderSmithequation}
\int_{B(2)}\langle |du|du, d\eta \rangle_{g} d\mu_{g} = -\frac{\sqrt{3}}{4}\int_{B(2)} \perm^{\alpha \beta \gamma} (u^{j} - \xi^{j}) \big( (J^{i}_{jk}\circ u)_{\alpha} u^{k}_{\beta} - (J^{i}_{jk}\circ u)_{\beta} u^{k}_{\alpha} \big) \eta_{\gamma}^{i} dx,
\end{equation}
for any constant vector $\xi \in \R^{d}$ and any $\eta \in W^{1, 3}_{0} \cap L^{\infty} (B(2); \R^{d})$.
\item We have
\begin{equation} \label{2ndorderSmithequation2}
\int_{B(2)}\langle |du|du, d\eta \rangle_{g} d\mu_{g} = -\frac{\sqrt{3}}{2}\int_{B(2)} \perm^{\alpha \beta \gamma} (J^{i}_{jk}\circ u)_{\gamma} u^{j}_{\alpha}u^{k}_{\beta} \eta^i dx.
\end{equation}
for any $\eta \in W^{1, 3}_{0} \cap L^{\infty} (B(2); \R^{d})$.
\end{enumerate}
\end{prop}
\begin{proof}
The idea is to compute $\int_{B(2)}\langle |du|du, d\eta \rangle_{g} d\mu_{g}$ using the Smith equation~\eqref{eq:smith-local-2} and then integrate by parts with the help of smooth approximations. We only give the details for part (a), since part (b) is similar. 

For any $\eta$ as in the statement of the theorem, using~\eqref{eq:smith-local-2} and that $\xi$ is a constant vector, we have
\begin{align*}
\int_{B(2)}\langle |du|du, d\eta \rangle_{g} d\mu_{g} & = \int_{B(2)} |du|u_{\lambda}^{i}\eta^{i}_{\delta}g^{\lambda\delta} \sqrt{g} dx \\
& = \frac{\sqrt{3}}{4}\int_{B(2)} \perm^{\alpha \beta \gamma} (J^{i}_{jk}\circ u) ( u_{\alpha}^{j}u_{\beta}^{k} - u_{\beta}^{j}u_{\alpha}^{k} ) \eta_{\gamma}^{i} dx \\
& = \frac{\sqrt{3}}{4}\int_{B(2)} \perm^{\alpha \beta \gamma} (J^{i}_{jk}\circ u) \left( (u^{j} - \xi^{j})_{\alpha}u_{\beta}^{k} - (u^{j} - \xi^{j})_{\beta} u_{\alpha}^{k} \right) \eta_{\gamma}^{i} dx.
\end{align*}
We want to integrate by parts, but $u$ does not necessarily possess weak second derivatives. Thus we take a sequence $(v_{n})$ in $C^{\infty}(\overline{B(2)}; \R^{d})$ such that $\|v_{n} - u\|_{1, 3; B(2)} \to 0$ as $n \to \infty$, and consider
\begin{align*}
& \int_{B(2)} \perm^{\alpha \beta \gamma} (J^{i}_{jk}\circ u) \left( (u^{j} - \xi^{j})_{\alpha}v_{n, \beta}^{k} - (u^{j} - \xi^{j})_{\beta}v_{n, \alpha}^{k} \right) \eta_{\gamma}^{i} dx \\
& \qquad = \perm^{\alpha \beta \gamma} \int_{B(2)} (J^{i}_{jk}\circ u) (u^{j} - \xi^{j})_{\alpha}v_{n, \beta}^{k} \eta_{\gamma}^{i} dx - \perm^{\alpha \beta \gamma} \int_{B(2)}(J^{i}_{jk}\circ u) (u^{j} - \xi^{j})_{\beta}v_{n, \alpha}^{k}\eta_{\gamma}^{i} dx.
\end{align*}
The first term can be treated as follows.
\begin{align*}
& \int_{B(2)} (J^{i}_{jk}\circ u) (u^{j} - \xi^{j})_{\alpha}v_{n, \beta}^{k} \eta_{\gamma}^{i} dx\\
& \qquad = -\int_{B(2)} (u^{j} - \xi^{j}) \left( v^{k}_{n, \beta\alpha}\eta_{\gamma}^{i}(J^{i}_{jk}\circ u) + v^{k}_{n, \beta}\eta^{i}_{\gamma\alpha}(J^{i}_{jk}\circ u) + v^{k}_{n, \beta} \eta_{\gamma}^{i}(J^{i}_{jk}\circ u)_{\alpha} \right).
\end{align*}
Similarly, 
\begin{align*}
& \int_{B(2)} (J^{i}_{jk}\circ u) (u^{j} - \xi^{j})_{\beta}v_{n, \alpha}^{k} \eta_{\gamma}^{i} dx\\
& \qquad = -\int_{B(2)} (u^{j} - \xi^{j}) \left( v^{k}_{n, \beta\alpha}\eta_{\gamma}^{i}(J^{i}_{jk}\circ u) + v^{k}_{n, \alpha}\eta^{i}_{\beta\gamma}(J^{i}_{jk}\circ u) + v^{k}_{n, \alpha} \eta_{\gamma}^{i}(J^{i}_{jk}\circ u)_{\beta} \right).
\end{align*}
Subtracting this from the previous identity, multiplying by $\perm^{\alpha \beta \gamma}$, and summing over $\alpha, \beta, \gamma$, we arrive at 
\begin{align*}
& \int_{B(2)} \perm^{\alpha \beta \gamma} (J^{i}_{jk}\circ u) \left( (u^{j} - \xi^{j})_{\alpha}v_{n, \beta}^{k} - (u^{j} - \xi^{j})_{\beta}v_{n, \alpha}^{k} \right) \eta_{\gamma}^{i} dx\\
& \qquad = - \int_{B(2)} \perm^{\alpha \beta \gamma} (u^{j} - \xi^{j})\left( (J^{i}_{jk}\circ u)_{\alpha} v^{k}_{n, \beta} - (J^{i}_{jk}\circ u)_{\beta} v^{k}_{n, \alpha} \right) \eta_{\gamma}^{i} dx.
\end{align*}
Letting $n \to \infty$ completes the proof.
\end{proof}

\begin{rmk} \label{hardyspacermk}
Equation~\eqref{2ndorderSmithequation} is useful because, fixing any $i, j$, the vector field $X$ on $B(2)$ defined by
\begin{equation*}
X^{\gamma} = \perm^{\alpha \beta \gamma} \big( (J^{i}_{jk}\circ u)_{\alpha} u^{k}_{\beta} - (J^{i}_{jk}\circ u)_{\beta} u^{k}_{\alpha} \big) 
\end{equation*}
lies in $L^{\frac{3}{2}}(B(2))$ and furthermore has divergence zero in $B(2)$ in the sense of distributions. Hence, by Proposition~\ref{divcurlhardy}, we know that for any $\eta \in W^{1, 3}_{0}(B(1); \R^{d})$ and fixed $i,j$, the function
\begin{equation*}
\perm^{\alpha \beta \gamma} \big( (J^{i}_{jk}\circ u)_{\alpha} u^{k}_{\beta} - (J^{i}_{jk}\circ u)_{\beta} u^{k}_{\alpha} \big) \eta_{\gamma}^{i}
\end{equation*}
lies in the Hardy space $\cH^{1}(\R^{3})$ and has $\cH^{1}$-norm bounded by
\begin{equation*}
C\|X\|_{\frac{3}{2}; B(2)} \|D\eta\|_{3; B(1)} \leq C\|Du\|^{2}_{3; B(2)} \|D\eta\|_{3; B(1)}.
\end{equation*} 
\end{rmk}

Next, we establish the regularity of $W^{1, 3}$-Smith maps on $B(2)$. 
\begin{prop} \label{energydecay}
For all $\alpha \in (0, 1)$, there exist constants $\ep_{0}, \theta \in (0, \frac{1}{9})$, depending only on $\alpha, M, J$ and on the embedding $M \to \R^{d}$, such that if the metric $g$ on $B(2)$ satisfies
\begin{equation} \label{metricnearflat}
|g - g_{\euc}|_{0; B(2)} + |Dg|_{0; B(2)} \leq \ep_{0},
\end{equation}
and if $u: B(2) \to M$ is a $W^{1, 3}$ associative Smith map with
\begin{equation} \label{smallenergy}
\int_{B(2)} |Du|^{3}dx < \ep_{0},
\end{equation}
then we have
\begin{equation} \label{decayinequality}
\int_{B(\theta)}|Du|^{3}dx \leq \theta^{3\alpha} \int_{B(2)}|Du|^{3}dx.
\end{equation}
\end{prop}

\begin{rmk}
Note that if $\ep_{0} < \frac{1}{9}$ then, because the domain is 3-dimensional, condition~\eqref{metricnearflat} implies that
\begin{equation} \label{metriccomparable}
\frac{1}{2}|\xi|^{2} \leq g_{\alpha\beta}(x)\xi^{\alpha}\xi^{\beta} \leq 2|\xi|^{2}, \text{ for all }x \in B(2), \xi \in \R^{3}.
\end{equation}
Here the $|\cdot|$ denotes length measured with respect to the Euclidean metric $g_{\euc}$. We often use relation~\eqref{metriccomparable} to go back and forth between $g$ and $g_{\euc}$.
\end{rmk}

\begin{proof}[Proof of Proposition~\ref{energydecay}]
First note that the inequalities~\eqref{metriccomparable} mean that for any map $w \in W^{1, 3}(B(1); \R^{d})$ we have
\begin{equation} \label{energycomparable}
\frac{1}{8} \int_{A} |Dw|^{3}dx \leq \int_{A} |d w|_{g}^{3}d\mu_{g} \leq 8\int_{A}|D w|^{3}dx, \quad \text{ for all }A \subseteq B(1).
\end{equation}
Next we let $v$ be the unique minimizer of $\int_{B(1)}|Dw|^{3}dx$ amongst all maps $w \in W^{1, 3}(B(1); \R^{d})$ satisfying $w - u \in W^{1, 3}_{0}(B(1); \R^{d})$. Then $v$ satisfies
\begin{equation} \label{homogeneous3harmonic}
\int_{B(1)} (|Dv|Dv \cdot D\eta) dx = 0 \quad \text{ for all }\eta \in W^{1, 3}_{0}(B(1); \R^{d}).
\end{equation}
Fix $i \in \{ 1, \ldots, d \}$ and choose $\eta = (v^{i} - \esssup_{B(1)}u^{i} )_{+} \paop{y^{i}}$, where $(\cdot)_{+}$ denotes the positive part. This $\eta$ belongs to $W^{1, 3}_{0}(B(1); \R^{d})$ because $v - u $ does. Then~\eqref{homogeneous3harmonic} implies
\begin{equation} \label{esssupbound}
v^{i} \leq \esssup_{B(1)}u^{i} \text{ on }B(1), \quad \text{ for all }i = 1, \ldots, d.
\end{equation}
Similarly, we can prove that 
\begin{equation} \label{essinfbound}
v^{i} \geq \essinf_{B(1)}u^{i} \text{ on }B(1), \quad \text{ for all }i = 1, \ldots, d.
\end{equation}
Since $M$ is compact, the two sets of inequalities above imply that $v$ lies in $L^{\infty}(B(1); \R^{d})$. Moreover, by the $3$-energy minimizing property of $v$ we have
\begin{equation} \label{vminimizing}
\int_{B(1)} |Dv|^{3}dx \leq \int_{B(1)}|Du|^{3}dx.
\end{equation}
The following result due to Uhlenbeck~\cite{U} is the reason we introduced the map $v$. Specifically, her work tells us that $v \in C^{1, \gamma}_{\loc}(B(1); \R^{d})$, and satisfies 
\begin{equation} \label{UhlenbeckC1}
\sup_{B(\frac{1}{2})} |Dv|^{3} \leq C\int_{B(1)}|Dv|^{3}dx \quad \Big( \leq C\int_{B(1)}|Du|^{3}dx \Big),
\end{equation}
where $\gamma, C$ are universal constants. 

To establish the asserted energy decay, we compare the $3$-energy of $u$ with that of $v$ and use~\eqref{2ndorderSmithequation} to estimate the difference. To do that, we need a system of equations satisfied by $v$ which resembles~\eqref{2ndorderSmithequation}. Therefore we use~\eqref{homogeneous3harmonic} to deduce, for all $\eta \in W^{1, 3}_{0}(B(1); \R^{d})$, that
\begin{align} \nonumber
\int_{B(1)} \langle |dv|_{g}dv, d\eta \rangle \sqrt{g}dx & = \int_{B(1)} \big( \langle |dv|_{g}dv, d\eta \rangle\sqrt{g} - |Dv|Dv\cdot D\eta \big) dx \\
\nonumber & = \int_{B(1)} \big( \sqrt{g}|dv|_{g}\langle dv, d\eta \rangle - |Dv|Dv\cdot D\eta \big) dx \\
\nonumber & = \int_{B(1)} \big( (\sqrt{g}|dv|_{g} - |Dv|)g^{\alpha\beta}v_{\alpha}\cdot \eta_{\beta}+ |Dv|\left( g^{\alpha\beta} - g_{\euc}^{\alpha\beta}\right)v_{\alpha}\cdot \eta_{\beta} \big) dx \\
 \label{modified3harmonic} & = \int_{B(1)} (F^{\beta}\cdot\eta_{\beta}) dx,
\end{align}
where we defined $F^{\beta}= (\sqrt{g}|dv|_{g} - |Dv|)g^{\alpha\beta}v_{\alpha}+ |Dv|\left( g^{\alpha\beta} - g_{\euc}^{\alpha\beta}\right)v_{\alpha}$. Note that by the definition of $F$ and~\eqref{vminimizing}, for some universal constant $C$ we have
\begin{equation} \label{errortermestimate}
\|F\|_{\frac{3}{2}; B(1)} \leq C|g - g_{\euc}|_{0; B(1)} \|Dv\|^{2}_{3; B(1)}\leq C|g - g_{\euc}|_{0; B(1)} \|Du\|^{2}_{3; B(1)}.
\end{equation}
Now we test both~\eqref{2ndorderSmithequation} and~\eqref{modified3harmonic} against $w := u - v$, extended to be zero outside of $B(1)$, and consider the differences of the resulting identities to get
\begin{align}
\nonumber& \int_{B(1)} \langle |du|_{g}du - |dv|_{g}dv, dw \rangle d\mu_{g} \\
\nonumber& \qquad = - \frac{\sqrt{3}}{4}\int_{B(1)} \perm^{\alpha \beta \gamma} (u^{j} - \xi^{j})\left( (J^{i}_{jk}\circ u)_{\alpha} u^{k}_{\beta} - (J^{i}_{jk}\circ u)_{\beta} u^{k}_{\alpha} \right) w_{\gamma}^{i} dx - \int_{B(1)} ( F^{\beta}\cdot w_{\beta} ) dx \\
\label{eq:3-harmonic-diff}& \qquad = - \frac{\sqrt{3}}{4}\int_{B(1)} \perm^{\alpha \beta \gamma} (u^{j} - \xi^{j}) \zeta \left( (J^{i}_{jk}\circ u)_{\alpha} u^{k}_{\beta} - (J^{i}_{jk}\circ u)_{\beta} u^{k}_{\alpha} \right) w_{\gamma}^{i} dx - \int_{B(1)} ( F^{\beta}\cdot w_{\beta} ) dx,
\end{align}
where $\zeta$ is a cutoff function which is identically $1$ on $B(1)$ and vanishes outside of $B(\frac{3}{2})$. 

We next show how the equation~\eqref{eq:3-harmonic-diff} can be used to estimate $\|Dw\|_{3; B(1)}$. Note that the second term on the last line of~\eqref{eq:3-harmonic-diff} can be estimated by H\"older's inequality and~\eqref{errortermestimate}, yielding
\begin{equation}\label{eq:differenceterm-2}
 \int_{B(1)} ( F^{\beta}\cdot w_{\beta} ) dx \leq C\|F\|_{\frac{3}{2}; B(1)}\|Dw\|_{3; B(1)} \leq C|g - g_{\euc}|_{0; B(1)} \|Du\|^{2}_{3; B(1)}\|Dw\|_{3; B(1)}.
\end{equation}
To estimate the first term, we choose $\xi = \fint_{B(2)} udx$ and use Proposition~\ref{cutoffBMO} to deduce that 
\begin{equation}\label{eq:u-xi-BMO}
[(u - \xi) \zeta]_{BMO} \leq C\|Du\|_{3; B(2)}. 
\end{equation}
On the other hand, denoting by $X_{j}$ the ($\R^{3\times d}$-valued) function defined by
\[
(X_{j})_{i}^{\gamma}= \perm^{\alpha \beta \gamma} \big( (J^{i}_{jk}\circ u)_{\alpha} u^{k}_{\beta} - (J^{i}_{jk}\circ u)_{\beta} u^{k}_{\alpha} \big),
\]
we see by Remark~\ref{hardyspacermk} and Proposition~\ref{divcurlhardy} that
\begin{equation*}
\| X_{j} \cdot D w \|_{\cH^{1}} \leq C\|Du\|_{3; B(2)}^{2}\|Dw\|_{3; B(1)}.
\end{equation*}
Using this together with~\eqref{eq:u-xi-BMO} and Theorem~\ref{BMOdual}, we obtain
\begin{equation}\label{eq:differenceterm-1}
\int_{B(1)} \perm^{\alpha \beta \gamma} (u^{j} - \xi^{j}) \zeta \left( (J^{i}_{jk}\circ u)_{\alpha} u^{k}_{\beta} - (J^{i}_{jk}\circ u)_{\beta} u^{k}_{\alpha} \right) w_{\gamma}^{i} dx \leq C\|Du\|_{3; B(2)}^{3}\|Dw\|_{3; B(1)}.
\end{equation}
Putting~\eqref{eq:differenceterm-1} and~\eqref{eq:differenceterm-2} back into the equation~\eqref{eq:3-harmonic-diff} and rearranging, we obtain
\begin{equation*}
\int_{B(1)} \langle |du|_{g}du - |dv|_{g}dv, dw \rangle d\mu_{g} \leq C\big( \|Du\|_{3; B(2)} + |g - g_{\euc}|_{0; B(1)} \big)\|Du\|_{3; B(2)}^{2}\|Dw\|_{3; B(1)}.
\end{equation*}

Next, we show how the left-hand side controls $\|Dw\|_{3; B(1)}$. Specifically, we have
\begin{align*}
\int_{B(1)}|Dw|^{3}dx &\leq C\int_{B(1)}|dw|_{g}^{3}d\mu_{g}\\
&\leq C\int_{B(1)} \big(|du|_{g} + |dv|_{g}\big)|dw|_{g}^{2}d\mu_{g}\\
&\leq C\int_{B(1)} \langle |du|_{g}du - |dv|_{g}dv, dw \rangle d\mu_{g},
\end{align*}
where the first inequality above follows from~\eqref{metriccomparable}, and in the last inequality we used~\cite[equation (10) on page 240]{DM}. 
To sum up, we arrive at 
\begin{equation}\label{eq:diff-w13-estimate}
\int_{B(1)}|Dw|^{3}dx \leq C\big( \|Du\|_{3; B(2)} + |g - g_{\euc}|_{0; B(1)} \big)\|Du\|_{3; B(2)}^{2}\|Dw\|_{3; B(1)}.
\end{equation}
Cancelling a factor of $\|Dw\|_{3; B(1)}$ from both sides and recalling the assumptions~\eqref{metricnearflat} and~\eqref{smallenergy}, we obtain
\begin{equation*}
\int_{B(1)}|Dw|^{3}dx \leq C\ep_{0}\|Du\|_{3; B(2)}^{3}.
\end{equation*}
On the other hand, by~\eqref{UhlenbeckC1} and~\eqref{vminimizing}, for each $\theta \in (0, \frac{1}{9})$ we have
\begin{equation*}
\int_{B(\theta)}|Dv|^{3}dx \leq C\theta^{3}\int_{B(2)} |Du|^{3}dx.
\end{equation*}
Combining the two estimates above gives
\begin{equation} \label{almostdecay}
\int_{B(\theta)} |Du|^{3}dx \leq C\int_{B(\theta)} |Dv|^{3} + |Dw|^{3}dx \leq C(\theta^{3} + \ep_{0})\int_{B(2)}|Du|^{3}dx.
\end{equation}
Now fix $\theta \in (0, \frac{1}{9})$ such that 
\begin{equation*}
C\theta^{3} < \frac{1}{2}\theta^{3\alpha},
\end{equation*}
which is possible because $\alpha \in (0, 1)$. Then choose $\ep_{0} = \theta^{3}$. We have
\begin{equation*}
C(\theta^{3} + \ep_{0}) < 2C\theta^{3} < \theta^{3\alpha}.
\end{equation*}
By the above and~\eqref{almostdecay}, we get~\eqref{decayinequality} and the proof is complete.
\end{proof}

\begin{rmk}\label{rmk:compensation}
The Hardy space $\cH^{1}$ and the Fefferman--Stein identification of its dual space with $BMO$ (see~\cite{FS} and also our Theorem~\ref{BMOdual}) underlie a ``compensation phenomenon'' that shows up in many geometric PDEs to give solutions better regularity than afforded by standard theory. Examples include the $H$-system for CMC surfaces (see for example~\cite{BC, We}), the harmonic map system~\cite{Be, E, He}, $p$-harmonic maps into special targets~\cite{T, TW}, and the Cauchy--Riemann equation for $J$-holomorphic maps~\cite{Wa}. Proposition~\ref{energydecay} is another manifestation of this compensation phenomenon. 
\end{rmk}

\begin{cor} \label{holderregularity}
For all $\alpha \in (0, 1)$, let $\ep_{0}$ be as in Proposition~\ref{energydecay}. If $u \in W^{1, 3}(B(2); M)$ is an associative Smith map with respect to $g$ and if both~\eqref{metricnearflat} and~\eqref{smallenergy} hold, then for all $x_{0} \in B(\frac{3}{2})$ and $r \in (0, \frac{1}{4})$ we have
\begin{equation} \label{morreyestimate}
\int_{B(x_0; r)} |Du|^{3}dx \leq Kr^{3\alpha}\int_{B(2)}|Du|^{3}dx \text{ for all }x_{0} \in B(\tfrac{3}{2}), r \in (0, \tfrac{1}{4}).
\end{equation}
In particular, the map $u$ lies in $C^{0, \alpha}(B(\tfrac{3}{2}); \R^{d})$, with 
\begin{equation} \label{holderestimate}
[u]_{\alpha; B(\frac{3}{2})}^{3} \leq C\int_{B(2)}|Du|^{3}dx.
\end{equation}
The constants $K$ and $C$ above have the same dependence as $\ep_{0}$.
\end{cor}
\begin{proof}
Choose any $x_{0} \in B(\frac{3}{2})$ and $r < \frac{1}{4}$ and consider the rescalings 
\begin{equation*}
\tilde{u}(x) = u(x_{0} + rx),\qquad \tilde{g}(x) = g(x_{0} + rx).
\end{equation*}
Then $(B(2), \tilde{g})$ and $(B(x_0; 2r), g)$ are conformal via $x \mapsto x_{0} + rx$, and by the conformal invariance of the Smith equation, we see that $\tilde{u}$ is a Smith map with respect to $\tilde{g}$. Furthermore, $\tilde{g}$ again satisfies~\eqref{metricnearflat} on $B(2)$, and 
\begin{equation*}
\int_{B(2)}|D\tilde{u}|^{3}dx = \int_{B(x_0; 2r)}|Du|^{3}dx \leq \ep_{0}.
\end{equation*}
Hence we may apply Proposition~\ref{energydecay} on $B(2)$ to $\tilde{u}$ and undo the rescaling to get
\begin{equation} \label{energydecayscaled}
\int_{B(x_0, \theta r )}|Du|^{3}dx \leq \theta^{3\alpha}\int_{B(x_0; 2r)} |Du|^{3}dx \quad \text{ for all }x_0 \in B(\tfrac{3}{2}), r \in (0, \tfrac{1}{4}).
\end{equation}
It is now fairly standard (compare with~\cite[Lemma 8.23]{GT}) to iterate~\eqref{energydecayscaled} to obtain
\begin{equation*}
\int_{B(x_0; r)} |Du|^{3}dx \leq Kr^{3\alpha}\int_{B(x_0; \frac{1}{2})}|Du|^{3}dx \quad \text{ for all }x_{0} \in B(\tfrac{3}{2}), r \in (0, \tfrac{1}{4}),
\end{equation*}
for some $K$ depending on $\theta$ and $\alpha$. This immediately gives~\eqref{morreyestimate}, and the second conclusion follows by Morrey's embedding (compare with~\cite[Chapter 7]{GT}).
\end{proof}

\begin{thm}[$\ep$-regularity] \label{thm:ep-reg}
Fix $\alpha \in (\frac{3}{4}, 1)$ and let $\ep_{0}$ be given by Proposition~\ref{energydecay} with this choice of $\alpha$. There exists a constant $\beta \in (0, 1)$, depending only on $\alpha, M, J$ and the embedding $M \to \R^{d}$, such that if $u \in W^{1, 3}(B(2); M)$ is an associative Smith map with respect to $g$ and if both~\eqref{metricnearflat} and~\eqref{smallenergy} hold, then the following hold.
\begin{enumerate}[(a)]
\item The map $u$ belongs to $C^{1, \beta}(B(1); \R^{d})$. Moreover $|u|_{1, \beta; B(1)}$ can be estimated in terms of $M, J$, the embedding $M \to \R^{d}$ and $\|Du\|_{3; B(2)}$.
\item In addition, $u$ is smooth on the (open) set 
\begin{equation*}
\{x \in B(1)\ |\ Du(x) \neq 0\}.
\end{equation*}
\end{enumerate}
\end{thm}
\begin{proof}
Part (a) can be proven as in~\cite[Lemma 6]{DM}, with only minor modifications, while part (b) is standard. For the sake of completeness we included proofs of both parts in Appendix~\ref{sec:regularity-appendix}.
\end{proof}

\begin{rmk} \label{rmk:threshold}
The constant $\ep_0$ of Theorem~\ref{thm:ep-reg} is called the \emph{threshold energy}.
\end{rmk}

\subsection{Mean value inequality, interior regularity and removable singularities} \label{sec:mean-value-r-s}

We now discuss three additional important properties of Smith maps, two of which are immediate consequences of the regularity results in $\S$\ref{sec:int-reg-sub}. These are Theorem~\ref{thm:MVI}, which is a mean value inequality for the gradient of an associative Smith map, Theorem~\ref{thm:int-reg}, which gives everywhere interior regularity of $W^{1, 3}$-Smith maps, and Theorem~\ref{thm:rem-sing}, which is a removable singularity result.

\begin{thm}[Mean value inequality] \label{thm:MVI}
Suppose that $\Ric_{g}$ is bounded by $K$, in the sense that 
\begin{equation*}
\left|\Ric_{g(x)}(v, v)\right| \leq K |v|_{g(x)}^{2} \quad \text{ for all }x \in B(2), v \in \R^{3}.
\end{equation*}
There exists $\ep_{1} > 0$ such that if~\eqref{metricnearflat} and~\eqref{smallenergy} hold with $\ep_{1}$ in place of $\ep_{0}$ and if $u \in W^{1, 3}(B(2); M)$ is an associative Smith map with respect to $g$, then we have
\begin{equation*}
\sup_{x \in B(\frac{1}{2})}|Du(x)| \leq C \Big(\int_{B(2)}|Du|^{3}dx \Big)^{\frac{1}{3}},
\end{equation*}
where both $\ep_{1}$ and $C$ depend only on $M, J$, the embedding $M \to \R^{d}$, and $K$.
\end{thm}
\begin{proof} This essentially follows from the arguments in~\cite[Theorem 2.1]{DF}, which concerns weakly $p$-harmonic maps with respect to the Euclidean metric. Indications of the main steps of the proof along with necessary modifications can be found Appendix~\ref{sec:regularity-appendix}.
\end{proof}

\begin{thm}[Interior regularity] \label{thm:int-reg}
Suppose that $g$ is a smooth Riemannain metric on $B(2)$, and that $u$ is a $W^{1, 3}$-Smith map on $B(2)$ with respect to $g$. Then $u$ has H\"older continuous first derivatives on $B(1)$. Moreover, $u$ is $C^\infty$ on the open set $\{x \in B(1) \colon du(x) \neq 0\}$.
\end{thm}
\begin{proof}
Obviously it suffices to prove that for all $x_{0} \in B(1)$ there exists $r \in (0, \frac{1}{4})$ such that both conclusions hold with $B(x_{0}; r)$ in place of $B(1)$. To that end, take any $x_{0} \in B(1)$ and define the rescalings $\tilde{u}$ and $\tilde{g}$ as in the proof of Corollary~\ref{holderregularity}, with $r < \frac{1}{4}$ to be determined. Then by the conformal invariance of the Smith equation, $\tilde{u}$ is a $W^{1, 3}$-Smith map with respect to $\tilde{g}$ on $B(2)$. Moreover, we have 
\begin{equation*}
\int_{B(2)} |D\tilde{u}|^{3}dx = \int_{B(x_0; 2r)} |Du|^{3}dx,
\end{equation*}
and
\begin{equation*}
|\tilde{g} - g_{\euc}|_{0; B(2)} + |D\tilde{g}|_{0; B(2)} = |g - g_{\euc}|_{0; B(x_0; 2r)} + r |Dg|_{0; B(x_0; 2r)}\leq 2r |Dg|_{0; B(x_0; 2r)}.
\end{equation*}
Since $u \in W^{1, 3}(B(2); M)$ and $g$ is smooth on $B(2)$, the two relations above imply that we may choose $r$ sufficiently small so that ~\eqref{metricnearflat} and ~\eqref{smallenergy} hold with $\ep_{0}$ given by Theorem~\ref{thm:ep-reg}. Consequently $\tilde{u} \in C^{1, \beta}(B(1); M)$ and is smooth on the open set $\{x \in B(1)\ |\ D\tilde{u}(x) \neq 0\}$. Since $x_{0} \in B(1)$ is arbitrary, we are done upon scaling back to $u$ and recalling the observation at the beginning of the proof.
\end{proof}

\begin{thm}[Removable singuarity] \label{thm:rem-sing}
Suppose that $g$ is a smooth Riemannian metric on $B(2)$, and that $u \in C^{1}_{\loc}(B(2)\setminus\{0\}; M)$ is a Smith map with respect to $g$, satisfying
\begin{equation*}
\int_{B(2)}|Du|^{3}dx < \infty.
\end{equation*}
Then in fact $u$ extends to a $C^{1}$-Smith map on all of $B(2)$.
\end{thm}
\begin{proof}

We first note that the assumptions imply that $u$ belongs to $W^{1, 3}(B(2); M)$ and that its weak derivative, which coincides with its classical derivative away from the origin, satisfies the Smith equation (with respect to $g$) almost everywhere. Next, we again consider the rescalings $\tilde{u}$ and $\tilde{g}$ introduced in the proof of Corollary~\ref{holderregularity}, this time choosing $x_{0} = 0$. Then, as in the proof of Theorem~\ref{thm:int-reg}, there exists a small enough $r$ such that Theorem~\ref{thm:ep-reg} is applicable to $\tilde{u}$ on $B(2)$. Hence $\tilde{u}$ lies in $C^{1, \beta}(B(1); M)$. In other words, $u \in C^{1, \beta}(B(r); M)$, which immediately gives the desired conclusion.
\end{proof}

\subsection{Convergence modulo bubbling} \label{sec:basic-convergence}

In this section we study sequences of Smith maps with uniformly bounded $3$-energy, which we may assume to be $C^{1}$ thanks to Theorem~\ref{thm:int-reg}. We show that the estimates in $\S$\ref{sec:mean-value-r-s} give $C^{1}$-(subsequential) convergence locally away from a finite set of points to a $C^{1}$ associative Smith map. As mentioned in \S\ref{sec:results-methods}, the main result of the present section, Proposition~\ref{prop:conv-mod-bubbling}, figures prominently in the construction of the bubble tree limit. In particular, we show in $\S$\ref{sec:bubble-tree} that bubbling phenomena occur precisely at the points where $C^{1}$-convergence fails. 

We now state the main result of this section.

\begin{prop}[Convergence modulo bubbling] \label{prop:conv-mod-bubbling}
Let $(\Sigma, g)$ be a closed Riemannian $3$-manifold and let $\Omega$ be an open subset of $\Sigma$ with $\{\Omega_{n}\}$ an increasing sequence of open subsets exhausting $\Omega$. Moreover, suppose that for each $n$ we have a Riemannian metric $g_{n}$ on $\Omega_{n}$ and an associative Smith map $u_{n} \in C^{1}(\Omega_{n}, M)$ with respect to $g_{n}$, such that 
\begin{equation*}
g_{n} \text{ converges smoothly to }g \text{ on compact subsets of }\Omega,
\end{equation*}
and
\begin{equation} \label{uniformenergybound}
\int_{\Omega_{n}}|du_{n}|_{g_{n}}^{3}d\mu_{g_{n}} \leq E_{0} \text{ for all }n.
\end{equation}
Then, there exists a finite set of points $\cS\subseteq \Sigma$ such that, up to taking a subsequence, the following hold.
\begin{enumerate}[(a)]
\item $(u_{n})$ converges in $C^{1}_{\loc}(\Omega\setminus \cS)$ to $u \in C^{1}_{\loc}(\Omega; M)$ which is an associative Smith map with respect to $g$, satisfying
\begin{equation*}
\int_{\Omega}|du|_{g}^{3}d\mu_{g} \leq E_{0}.
\end{equation*}
\item As Radon measures on $\Omega$, we have
\begin{equation*}
|du_{n}|_{g_{n}}^{3}d\mu_{g_{n}} \to |du|_{g}^{3}d\mu_{g} + \sum_{x \in \cS} m_{x}\delta_{x},
\end{equation*}
with each $m_{x} \geq \frac{\ep_{0}}{2}$, where the constant $\ep_{0}$ is the threshold energy from Remark~\ref{rmk:threshold}.
\item If $\| du_n \|_{p; \Omega_n} \leq C$ for some $p \in (3, \infty]$, then $\mathcal{S} = \varnothing$.
\end{enumerate}
\end{prop}
\begin{proof} 
We first identify the set $\cS$. Let $\nu_{n}$ denote the Radon measure $|du_{n}|_{g_{n}}^{3}d\mu_{g_{n}}$. Then~\eqref{uniformenergybound} implies that, for all $m \in \NN$, the sequence $(\nu_{n})_{n > m}$ has uniformly bounded total mass on $\Omega_{m}$. Hence, by standard functional analysis and a diagonal process, we obtain a subsequence, which we do not relabel, that converges in the weak-$\ast$ sense to a Radon measure $\nu$ on $\Omega$ which satisfies
\begin{equation*}
\nu(\Omega) \leq E_{0}.
\end{equation*}
Consequently, if we define
\begin{equation*}
\cS = \{x \in \Omega\ |\ \nu(\{x\})\geq \tfrac{1}{2}\ep_{0} \},
\end{equation*}
where $\ep_{0}$ is the constant from Theorem~\ref{thm:ep-reg}, then $\cS$ is finite. 

To prove part (a), take any $x \notin \cS$. Then by the definition of $\nu$ and $\cS$, there exists a radius $r > 0$ such that for all large enough $n$ we have
\begin{equation*}
\int_{B(x; r)} |d u_{n}|_{g_{n}}^{3}d\mu_{g_{n}} < \tfrac{1}{2}\ep_{0},
\end{equation*}
where the geodesic ball $B(x; r)$ is taken with respect to $g$. By the smooth convergence of $g_{n}$ to $g$ on $\overline{B(x; r)}$ and the above inequality, and shrinking $r$ if necessary, we may pull $u_{n}$ and $g_{n}$ back to $B(r)$ via the exponential map $\exp_{x}$ with respect to $g$, and then dilate to $B(2)$ to conclude that the hypotheses of Theorem~\ref{thm:ep-reg} are all verified. Thus we obtain uniform $C^{1, \beta}$-estimates on $B(x; \frac{r}{2})$ for the sequence $(u_{n})$. Therefore, passing successively to subsequences (without relabelling), we see that $(u_{n})$ converges in $C^{1}$ on every compact subset of $\Omega\setminus \cS$ to some map $u \in C^{1, \beta}_{\loc}(\Omega\setminus \cS; M)$ which satisfies the Smith equation pointwise on $\Omega\setminus \cS$. To see that $u$ is in fact $C^{1}$ on all of $\Sigma$, note that for each $\Omega' \subset\subset \Omega \setminus \cS$ we have
\begin{equation*}
\int_{\Omega'}|du|^{3}_{g}d\mu_{g} =\lim_{n \to \infty}\int_{\Omega'}|du_{n}|^{3}_{g_{n}}d\mu_{g_{n}} \leq E_{0},
\end{equation*}
and hence $\int_{\Omega}|du|_{g}^{3}d\mu_{g} \leq E_{0}$, since $\cS$ is a finite set. But then Theorem~\ref{thm:rem-sing} applies to give $C^{1}$-regularity of $u$ on all of $\Sigma$, and we have completed the proof of (a).

For (b), we consider the so-called \textit{defect measure}, defined by
\begin{equation*}
\theta := \nu - |d u|^{3}_{g}d\mu_{g},
\end{equation*}
and observe that $\supp(\theta) \subset \cS$ because $u_{n}$ converges to $u$ in $C^{1}_{\loc}(\Omega\setminus \cS)$. Therefore $\theta$ must be of the form
\begin{equation*}
\theta = \sum_{x \in \cS}m_{x}\delta_{x},
\end{equation*}
and it remains to show that $m_{x} \geq \frac{1}{2} \ep_{0}$ for all $x \in \cS$. Indeed, for any $x \in \cS$ and $r > 0$ such that $B(x; r) \cap \cS = \{x\}$, by the definition of $\theta$ we have
\begin{equation*}
m_{x} = \theta(B(x; r)) = \nu(B(x; r)) - \int_{B(x; r)} |du|^{3}_{g}d\mu_{g}.
\end{equation*}
Since $u \in C^{1}(\Sigma; M)$, letting $r$ tend to zero yields
\begin{equation} \label{mx}
m_{x} = \lim_{r \to 0}\nu(B(x; r)) = \nu(\{x\}) \geq \tfrac{1}{2} \ep_{0},
\end{equation}
and we are done with part (b).

Finally we prove part (c). Since $p \in (3, \infty]$, we see by the assumption and H\"older's inequality that for all $B(x; r) \subset \Omega_n$ there holds
\begin{align*}
\int_{B(x; r)} |d u_{n}|_{g_{n}}^{3}d\mu_{g_{n}} &\leq \Big( \int_{B(x; r)} |d u_{n}|_{g_{n}}^{p} d\mu_{g_{n}} \Big)^{\frac{3}{p}}\Big( \mu_{g_{n}}(B(x; r)) \Big)^{\frac{p-3}{p}} \\
&\leq C^{3}\Big( \mu_{g_{n}}(B(x; r)) \Big)^{\frac{p-3}{p}}.
\end{align*}
This implies, by the smooth local convergence of $g_n$ to $g$, that for all $x \in \Omega$ there exists a constant $C_{1}$ independent of $r$ and $n$ such that 
\begin{equation*}
\int_{B(x; r)} |d u_{n}|_{g_{n}}^{3}d\mu_{g_{n}} \leq C_{1}r^{\frac{3(p-3)}{p}},
\end{equation*}
for $r$ small enough and $n$ large enough. Consequently for all $x \in \Omega$ and small enough $r$ we have 
\[
\nu(\{x\}) \leq \nu(B(x; r)) \leq C_{1}r^{\frac{3(p-3)}{p}},
\]
which implies that $\cS = \varnothing$ by its definition. 
\end{proof}

\begin{rmk} We make three remarks about Proposition~\ref{prop:conv-mod-bubbling}.
\begin{enumerate}[(i)]
\item In practice we only apply Proposition~\ref{prop:conv-mod-bubbling} to the case $\Omega = \Sigma$ or $\Omega = S^{3}\setminus \{p^{-}\}$, where $p^{-}$ is the south pole. Note that in the latter case, the finite-energy property of $u$ and Theorem~\ref{thm:rem-sing} imply that $u$ is actually $C^{1}$ on all of $S^{3}$.
\item From~\eqref{mx} we deduce that the numbers $\{m_{x}\}_{x \in \cS}$ can be characterized by
\begin{equation*}
m_{x} = \lim_{r \to 0}\lim_{n \to \infty}\nu_{n}(B(x; r))
\end{equation*}
for all $x \in \cS$. Thus one can view $m_{x}$ as the amount of energy concentrating at the point $x$. For this reason, the set $\cS$ is often referred to as the \textit{energy concentration set}. Note that $\cS$ can equivalently be defined as
\begin{equation*}
\cS = \big\{ x \in \Sigma\ \big|\ \liminf_{n \to \infty}\nu_{n}(B(x; r)) \geq \tfrac{1}{2} \ep_{0} \text{ for all } r > 0 \big\}.
\end{equation*}
\item The defect measure $\theta$ can be thought to capture the energy that escapes the limit map $u$. One of the main goals of the present paper is to study what happens to this ``escaped energy''. We do this in $\S$\ref{sec:bubble-tree}. The defect measure and the energy concentration set have long been used in many other contexts, such as the study of harmonic maps (for example, see~\cite{Li, Sch}) and Yang--Mills connections (for example, see~\cite{Ri, Ti}).
\end{enumerate}
\end{rmk}

\subsection{Non-relation between $J$-holomorphic and associative bubbling} \label{sec:uv-stuff}

In Proposition~\ref{prop:conv-mod-bubbling} we established that a sequence $u_n$ of associative Smith maps with uniformly bounded $3$-energy will (up to passing to a subsequence) converge away from a finite set of points, where bubbling occurs. These are points where \emph{the $3$-energy of $u_n$ concentrates}.

As stated in item (ix) of \S\ref{sec:motivation}, an entirely analogous result holds in the classical situation of holomorphic curves. That is, if $v_n$ is a sequence of holomorphic curves with uniformly bounded $2$-energy, then up to passing to a subsequence, $v_n$ will converge away from a finite set of points where bubbling occurs. In this case, these are points where \emph{the $2$-energy of $v_n$ concentrates}.

Recall from \S\ref{sec:smith-relations} that a complex curve $\Sigma^2$ in a Calabi-Yau $3$-fold $Y^6$ gives rise to an associative submanifold $\Sigma^2 \times S^1$ in the $\G$ manifold $Y^6 \times S^1$. Therefore it is natural to wonder, if $v_n : (\Sigma^2, [g_{2, n}]) \to (Y^6, h_6)$ is a sequence of holomorphic curves bubbling at points $\{ z_1, \ldots, z_N \}$ in $\Sigma$, whether it should be the case that $u_n = v_n \times \mathrm{Id}_{S^1} : \Sigma^2 \times S^1 \to Y^6 \times S^1$ is a sequence of associative Smith maps bubbling along a finite set of circles $\{ z_k \} \times S^1$, for $k=1, \ldots, N$. This would appear to contradict Proposition~\ref{prop:conv-mod-bubbling} which says that bubbling of associative Smith maps always occurs in codimension $3$, not codimension $2$.

The following proposition clarifies in what precise sense the above reasoning is faulty. We use the notation of \S\ref{sec:smith-relations}. For simplicity we assume that the representatives $g_{2, n}$ in the conformal classes $[g_{2, n}]$ can be chosen to converge smoothly to a limit metric $g_{2}$, since $2$-energy concentration for holomorphic curves can already occur in this setting.

\begin{prop} \label{prop:no-contradiction}
Let $v_{n}: \Sigma \to (Y^6, h_6)$ be a sequence of immersions, where each $v_{n}$ is holomorphic with respect to the conformal class $[g_{2, n}]$ of metrics on $\Sigma$. Assume that the representatives $g_{2, n}$ converge smoothly to a limit metric $g_2$. Define $u_{n}: \Sigma \times S^{1} \to (Y \times S^{1}, h_{7} = h_{6} + \big( d\theta^{2}) \big)$ by
$$u_{n}(x, \phi) = (v_n(x), f_n(x, \phi))$$
and suppose that the $f_n' := \frac{\partial f_n}{\partial \phi}$ are nonvanishing, so that the $u_n$ are also immersions. Further suppose that the following three conditions all hold:
\begin{enumerate}[(a)]
\item Each $u_n$ is an associative Smith map with respect to the conformal class $[g_{3, n}]$ of metrics on $\Sigma \times S^1$.
\item The representatives $g_{3, n}$ can be chosen to converge smoothly to some limit metric $g_{3}$.
\item The $3$-energies
$$\int_{\Sigma \times S^1} |du_n|^{3}_{g_{3, n}} \vol_{g_{3, n}} \quad \text{are uniformly bounded}.$$
\end{enumerate}
Then the $3$-energy of $v_n$ with respect to $g_{2,n}$ is uniformly bounded. Hence, by item (ix)(c) in \S\ref{sec:motivation}, the sequence $\{v_n\}$ does not bubble.
\end{prop}
\begin{proof} Suppose that (a), (b), and (c) all hold. Using the notation of~\eqref{eq:v-mu-defn} and~\eqref{eq:u-lambda-defn} we write
\begin{equation} \label{eq:mu-lambda}
\mu_n^2 = \tfrac{1}{2} |dv_n|^2_{g_{2,n}}, \qquad \lambda_n^2 = \tfrac{1}{3} |du_n|^2_{g_{3,n}}.
\end{equation}
Since the $v_n$ are holomorphic with respect to $[g_{2, n}]$, and using assumption (a), Proposition~\ref{prop:g3-vol3} applies. From equation~\eqref{eq:g3} in that proposition, we have
\begin{equation} \label{eq:g3n}
g_{3, n} = \lambda_{n}^{-2} \big( \mu_{n}^{2}g_{2, n} + (df_{n})^{2} \big).
\end{equation}
From assumption (b), in local coordinates the components of $g_{3, n}$ converge smoothly to the components of the limit metric $g_3$. In local coordinates $(x^{1}, x^{2})$ on $\Sigma$, write $g_{2,n} = (g_{2,n})_{ij} dx^i dx^j$. Then in terms of $(x^1, x^2, \phi)$ we can write~\eqref{eq:g3n} as
\begin{align*}
g_{3, n} & = \Big( \frac{\mu_n^{2}}{\lambda_{n}^{2}}(g_{2,n})_{ij} + \lambda_{n}^{-2} \pa{f_n}{x^{i}}\pa{f_n}{x^{j}} \Big) dx^i dx^j + 2 \Big(\lambda_{n}^{-2}\pa{f_n}{x^{i}} \pa{f_n}{\phi}\Big) dx^i d\phi + \lambda_n^{-2}\Big( \pa{f_n}{\phi} \Big)^{2} (d\phi)^{2}.
\end{align*}
Therefore we must have
\begin{equation} \label{eq:k-limit}
\lambda_n^{-1}\pa{f_n}{\phi} \to k \quad \text{ smoothly as $n \to \infty$},
\end{equation}
where $k > 0$ on $\Sigma \times S^1$. Using this in the cross term $dx^i d\phi$ above tells us that for $i = 1, 2$ the functions
$$
\lambda_n^{-1}\pa{f_n}{x^i} = \frac{\lambda_{n}^{-2}\pa{f_n}{x^{i}} \pa{f_n}{\phi}}{\lambda_n^{-1}\pa{f_n}{\phi}}
$$
converge smoothly to some limit function. Consequently, looking at the $dx^i dx^j$ terms above, we conclude that 
\begin{equation} \label{eq:h-limit}
\lambda_n^{-1} \mu_n \to h \quad \text{ smoothly as $n \to \infty$}
\end{equation}
where $h \geq 0$ on $\Sigma \times S^1$ because $\lambda_n, \mu_n$ are both positive. We may not have $h> 0$ everywhere, but this does not matter. Writing $f_n' := \pa{f_n}{\phi}$, by combining~\eqref{eq:k-limit} and~\eqref{eq:h-limit} we find that
\begin{equation} \label{eq:muf-limit}
\frac{\mu_n}{f_n'} \to \frac{h}{k} \quad \text{smoothly as $n \to \infty$}.
\end{equation}
Thus, for $n$ sufficiently large, we have
$$
0 < \frac{\mu_n}{f_n'} < \Big |\frac{h}{k} \Big|_{0; \Sigma \times S^{1}} + 1 =: L
$$
and hence
\begin{equation} \label{eq:Linv}
\frac{f_n'}{\mu_n} > L^{-1} \quad \text{ everywhere on $\Sigma \times S^1$ for sufficiently large $n$}.
\end{equation}
Note that $L$ is just a positive \emph{constant}.

Since the $v_n$ are holomorphic with respect to $[g_{2, n}]$, and using assumption (a), Proposition~\ref{prop:g3-vol3} applies. Equation~\eqref{eq:vol3} in that proposition gives
\begin{equation} \label{eq:vol3n}
\lambda_n^3 \vol_{g_{3,n}} = \mu_n^2 f_n' d\phi \wedge \vol_{g_{2,n}}.
\end{equation}
From~\eqref{eq:vol3n} and~\eqref{eq:mu-lambda}, we have that for $n$ sufficiently large,
\begin{align*}
\frac{1}{3 \sqrt{3}} \int_{\Sigma \times S^{1}}|du_{n}|^{3}_{g_{3, n}} \vol_{g_{3, n}} & = \int_{\Sigma \times S^{1}} \lambda_{n}^{3}\vol_{g_{3, n}} \\
& = \int_{\Sigma \times S^{1}} f_n' \mu_n^{2} \vol_{g_{2,n}} d\phi \\
& = \int_{\Sigma \times S^{1}} \frac{f_n'}{\mu_n} \mu_n^{3}\vol_{g_{2,n}} d\phi \\
& > L^{-1}\int_{\Sigma \times S^{1}} \mu_{n}^{3}\vol_{g_{2,n}} d\phi \\
& = \frac{2\pi L^{-1}}{2^{\frac{3}{2}}} \int_{\Sigma}|dv_n|^{3}_{g_{2,n}} \vol_{g_{2,n}}.
\end{align*}
Hence, assumption (c) implies that the \emph{$3$-energy} of $v_n$ with respect to $g_{2,n}$ is uniformly bounded, which by item (ix)(c) in \S\ref{sec:motivation} prevents $\sup_{x \in \Sigma} |dv_n(x)|_{g_{2}}$ from going to infinity.
\end{proof}

\begin{rmk} \label{rmk:S1-bubbling} The above argument suggests that in order to construct a sequence of associative Smith maps $u_n$ from a sequence of holomorphic curves $v_n$ that ``bubbles'' along circles, one would have to violate assumptions (b) or (c) above. That is, one would have to allow the conformal classes $[g_{3,n}]$ to degenerate, or else allow for the $3$-energies of $u_n$ to be unbounded.
\end{rmk}

\begin{rmk} \label{rmk:fn-nontrivial}
Suppose we chose $f_n(x, \phi) = \phi$, so that $u_n = v_n \times \mathrm{Id}_{S^1}$. With this choice for $f_n$, if $(v_n)_*[\Sigma] \in H^2(Y, \R)$ is constant, then $(u_n)_*[\Sigma \times S^1] \in H^3(Y \times S^1, \R)$ is constant. But then $f_n' = 1$ for all $n$, and equation~\eqref{eq:Linv} shows that the sequence $\{\sup_{x \in \Sigma} \mu_n(x)\}$ is bounded. Thus for the choice $f_n(x, \phi) = \phi$, if (a), (b), (c) all hold, then the sequence $v_n$ does not bubble by item (ix)(c) of \S\ref{sec:motivation}.\end{rmk}

\subsection{Energy lower bound for maps from $S^{3}$} \label{sec:energy-lower-bound}

In this section we record for later use two results, both of which give lower bounds for the $3$-energy of maps from $S^{3}$ which are ``nontrivial'' in some sense. The first result concerns Smith maps on the standard sphere $S^{3}$ and is important for proving that the bubble tree terminates after finitely many steps.

\begin{prop}[Energy Gap. Compare with {\cite[Proposition 4.1.4]{MS}}]\label{prop:energy-gap}
Let $\ep_{0}$ be as in Theorem~\ref{thm:ep-reg}. If $u : S^{3} \to M$ is a $C^{1}$-associative Smith map with respect to the round metric on $S^{3}$, and if
\begin{equation*}
\int_{S^{3}}|du|^{3}d\mu < \ep_{0},
\end{equation*}
then $u$ is constant.
\end{prop}
\begin{proof}
The argument below is inspired by the proof of~\cite[Proposition 4.1.4]{MS} and relies on the fact that the round $S^{3}$ is conformally flat. Let $\sigma : \R^{3} \to S^{3}\setminus \{p\}$ denote the stereographic projection. By the conformal invariance of the Smith equation and the $3$-energy, the map $\bar{u} = \sigma^{\ast}u$ is a Smith map on $\R^{3}$ with respect to the flat metric and satisfies
\begin{equation*}
\int_{\R^{3}}|D\bar{u}|^{3}dx < \ep_{0}.
\end{equation*}
In particular, Theorem~\ref{thm:MVI} is applicable to the map $\bar{u}_{R}(x) := \bar{u}(Rx)$ on $B(2)$ for any $R > 0$, giving us
\begin{equation*}
|D\bar{u}|_{0; B(\frac{R}{2})} \leq CR^{-3}\int_{B(2R)}|D\bar{u}|^{3}dx \leq CR^{-3}\ep_{0} \quad \text{ for all }R > 0.
\end{equation*}
From this we easily see that $\bar{u}$ must be constant, and hence $u$ is constant as well.
\end{proof}

The next result is a special case of a theorem due to White~\cite{Wh} and concerns a general Lipschitz map on $S^{3}$ with small energy and is a crucial ingredient in proving that there is no energy loss through the necks.

\begin{prop}[{\cite[Theorem 2]{Wh}}] \label{Brian}
There exists a constant $\gamma_{1} > 0$ depending only on $M$ and the embedding $M \to \R^{d}$ such that if $u : S^{3} \to M$ is a Lipschitz map with 
\begin{equation} \label{Briangap}
\int_{S^{3}}|du|^{3}d\mu < \gamma_{1},
\end{equation}
then $u$ is homotopic to a constant. Here $|du|$ and $d\mu$ are both with respect to the round metric on $S^3$.
\end{prop}
\begin{proof}
For the convenience of the reader we briefly describe what the proof in~\cite{Wh} becomes in our case. In short, the homotopy is constructed out of the averages of $u$ on geodesic balls, and the assumption~\eqref{Briangap} enters the argument through the Poincar\'e inequality. To set things up, for a geodesic ball $B(x; r)$, we define
\begin{equation*}
(u)_{x, r} = \fint_{B(x; r)} u d\mu.
\end{equation*}
Then there exist universal constants $C, r_{0} > 0$ such that for all $x \in S^{3}$ and $r \leq r_{0}$ the following two conditions hold:
\begin{align*}
\frac{1}{2} r^{3} \leq \mu(B(x; r)) & \leq 2r^{3}, \\
r^{-3} \int_{B(x; r)}|u - (u)_{x, r}|^{3}d\mu & \leq C\int_{B(x; r)}|du|^{3} d\mu.
\end{align*}
Furthermore, there exists $\delta_{0}$ depending on $M$ such that 
\begin{equation*}
\cN_{\delta_{0}}(M) := \{x \in \R^{d}\ |\ \dist_{\R^{d}}(x, M) < \delta_{0}\}
\end{equation*}
is strictly contained in a tubular neighborhood of $M$ in $\R^{d}$. Let $\pi: \cN_{\delta_{0}}(M) \to M$ denote the nearest-point projection.

Now since $u$ maps into $M$, the inequalities above and~\eqref{Briangap} imply that if $\gamma_{1}$ is small enough, then
\begin{equation*}
\dist_{\R^{d}}((u)_{x, r}, M)^{3} < 2C\gamma_{1} < \delta_{0}^{3} \quad \text{ for all }x \in S^{3}, r\leq r_{0},
\end{equation*}
and thus $\pi((u)_{x, r})$ is well-defined, and we have 
\begin{equation} \label{distancebound}
|(u)_{x, r} - \pi((u)_{x, r})| \leq \left( 2C\gamma_{1} \right)^{\frac{1}{3}} \quad \text{ for all }x \in S^{3}, r \leq r_{0}.
\end{equation}
Consider the function $H: S^{3} \times [0, r_{0}] \to M$ defined by 
\begin{equation*}
H(x, r) = \begin{cases}
\pi((u)_{x, r}) & \text{if }r \in (0, r_{0}], \\
u(x) & \text{if } r = 0.
\end{cases}
\end{equation*}
Since $u$ is Lipschitz by assumption, the function $H$ is continuous and gives a homotopy from $u$ to $H(\cdot, r_{0})$. It remains to show that the latter is null-homotopic. To that end we recall the following Poincar\'e inequality on the whole $S^{3}$. Namely,
\begin{equation*}
\int_{S^{3}}|u - (u)_{S^{3}}|^{3}d\mu \leq C\int_{S^{3}} |du|^{3}d\mu.
\end{equation*}
Then for all $x \in S^{3}$ we compute
\begin{align*}
\left| (u)_{x, r_{0}} - (u)_{S^{3}} \right| & = \Big( \fint_{B(x; r_0)}|(u)_{x_{0}, r} - (u)_{S^{3}}|^{3} d\mu \Big)^{\frac{1}{3}} \\
& \leq \Big( \fint_{B(x; r_0)}|(u)_{x_{0}, r} - u |^{3}d\mu \Big)^{\frac{1}{3}} + \Big( \fint_{B(x; r_0)}|u - (u)_{S^{3}}|^{3} d\mu \Big)^{\frac{1}{3}} \\
& \leq C \Big( \frac{r_{0}^{3}}{\mu(B(x; r_0))}\int_{B(x; r_0)}|du|^{3}d\mu \Big)^{\frac{1}{3}} + C \left[\mu(B(x; r_0))\right]^{-\frac{1}{3}} \Big( \int_{S^{3}} |du|^{3} d\mu \Big)^{\frac{1}{3}} \\
& \leq C(1 + r_{0}^{-1})\gamma_{1}^{\frac{1}{3}}.
\end{align*}
Thus for all $x_{1}, x_{2} \in S^{3}$ we obtain
\begin{equation*}
|(u)_{x_{1}, r_{0}} - (u)_{x_{2}, r_{0}}| \leq 2C(1 + r^{-1}_{0})\gamma_{1}^{\frac{1}{3}}.
\end{equation*}
Combining this with~\eqref{distancebound}, we arrive at
\begin{equation*}
|H(x_{1}, r_{0}) - H(x_{2}, r_{0})| \leq C'(1 + r_{0}^{-1}) \gamma_{1}^{\frac{1}{3}} \quad \text{ for all }x_{1}, x_{2} \in S^{3}.
\end{equation*}
Thus, making a smaller choice of $\gamma_{1}$ if necessary, we deduce that $H(\cdot, r_{0})$ has image contained in a geodesic ball in $M$, and hence must be null-homotopic. The proof is complete.
\end{proof}

\section{The bubble tree} \label{sec:bubble-tree}

As stated in $\S$\ref{sec:int-reg-cpt}, for the remainder of this paper, we are exclusively concerned with associative Smith maps $u \colon (\Sigma^3, g) \to (M^7, h)$, although our results also apply to Cayley Smith maps after making the obvious modifications.

In this section we employ the following notation:
\begin{itemize} \setlength\itemsep{-1mm}
\item For a $W^{1,3}$ map $u \colon (\Sigma, g) \to (M,h)$ and a measurable subset $A \subset \Sigma$, we write
$$E_g(u; A) = E(u;A) = \int_A |du|^3\,d\mu_g,$$
often suppressing explicit reference to the domain metric $g$ when it is clear from context. (From now on, for convenience, we drop the $\frac{1}{3 \sqrt{3}}$ factor in the $3$-energy.) For brevity, we often write $E(u) := E(u; \Sigma)$.
\item We use $I = (i_1, \ldots, i_k)$ to denote a multi-index of length $|I| = k \geq 1$.
\item The symbol $\ep_0$ always denotes the threshold energy constant of Theorem~\ref{thm:ep-reg} and Remark~\ref{rmk:threshold}.
\end{itemize}

\subsection{Overview} \label{sec:bubble-overview}

Let $u_n \colon (\Sigma^3, g) \to (M^7, h)$ be a sequence of associative Smith maps with bounded $3$-energy. That is,
$$E(u_n) \leq E_0 \quad \text{ for all } n \in \NN.$$
By Proposition~\ref{prop:conv-mod-bubbling}, after passing to a subsequence, there exists an associative Smith map $u_\infty \in C^1_{\loc}(\Sigma; M)$ called the \textbf{base map} for which
\begin{equation}
\label{eq:Setup}
u_n \to u_\infty \text{ in } C^1_{\loc}(\Sigma \setminus \mathcal{S})
\end{equation}
where $\mathcal{S} = \{x_1, \ldots, x_q\}$ is a finite set (possibly empty) whose elements are called \textbf{(zeroth level) bubble points}, and also
\begin{equation}
|du_n|^3\,d\mu_g \to |du_\infty|^3\,d\mu_g + \sum_{i=1}^q m_i \delta(x_i) \quad \text{ as Radon measures,} \label{eq:Radon-Conv}
\end{equation}
where each $m_i \geq \frac{1}{2}\ep_0$. The purpose of the whole of $\S$\ref{sec:bubble-tree} is to understand the behavior of the sequence $\{u_n\}$ at the bubble points $x_i$.

Suppose that $\mathcal{S} \neq \varnothing$. At each $x_i \in \mathcal{S}$, we choose local coordinates centered at $x_i$ and construct a sequence of conformal maps
$$R_n \colon \Omega_n \subset S^3_{x_i} \setminus \{p^-\} \to \R^3,$$
where $S^3_{x_i}$ is the round $3$-sphere (the subscript $x_i$ is simply for bookkeeping), $p^- \in S^3_{x_i}$ is the south pole, and $\Omega_n \subset S^3_{x_i}$ is a particular increasing sequence of open sets that exhaust $S^3_{x_i} \setminus \{p^-\}$. The rescaled maps
\begin{equation} \label{eq:rescaled-maps}
\widetilde{u}_{n,i} := u_n \circ R_n \colon \Omega_n \to M
\end{equation}
are a sequence of associative Smith maps (with respect to appropriate metrics $h_n$ on $\Omega_n$) that has bounded $3$-energy.

Hence, again by Proposition~\ref{prop:conv-mod-bubbling}, there exists an associative Smith map $\widetilde{u}_{\infty, i} \colon (S^3_{x_i}, g_{\round}) \to (M, h)$ of class $C^1_{\loc}$, called a \textbf{first level bubble map}, and a finite set (possibly empty)
\begin{equation*}
\mathcal{S}_i = \{x_{i1}, \ldots, x_{iq_i}\} \subset S^3_{x_i} \setminus \{p^-\},
\end{equation*}
called the \textbf{first level bubble points}, for which (after passing to a subsequence) we have
\begin{align*}
\widetilde{u}_{n,i} \to \widetilde{u}_{\infty, i} \quad \text{ in } C^1_{\loc} \big( S^3_{x_i} \setminus (\{p^-\} \cup \mathcal{S}_i) \big)
\end{align*}
and
\begin{align*}
|d\widetilde{u}_{n,i}|^3\,d\mu_{h_n} \to |d\widetilde{u}_{\infty,i}|^3\,d\mu_{\round} + \sum_{j=1}^{q_i} m_{ij}\delta(x_{ij}) \quad \text{ as Radon measures on }S^3_{x_i} \setminus \{p^-\}
\end{align*}
where each $m_{ij} \geq \frac{1}{2}\ep_0$.

If $\mathcal{S}_i \neq \varnothing$ for some $i = 1,\ldots, q$, then this process may be repeated, yielding associative Smith maps $\widetilde{u}_{\infty, ij} \colon (S^3_{x_{ij}}, g_{\round}) \to M$ of class $C^1_{\loc}$, called \textbf{second level bubble maps}, and a finite set $\mathcal{S}_{ij} \subset S^3_{x_{ij}} \setminus \{p^-\}$ (possibly empty) of \textbf{second level bubble points}. Evidently, we may iterate this procedure as long as the sets of bubble points remain nonempty. In $\S$\ref{sec:bubble-construction}, we show that, in fact, this process must eventually terminate.

The result of this iteration is a \textbf{bubble tree}. That is, one obtains a \emph{tree} $(\{T_0, T_I\}, \{E_I\})$, meaning a connected graph without cycles, in the following way:
\begin{itemize} \setlength\itemsep{-1mm}
\item The vertex $T_0$ corresponds to the base map $u_\infty$.
\item Each vertex $T_i$ corresponds to the first-level bubble map $\widetilde{u}_{\infty, i}$. Each vertex $T_i$ (for $i = 1,\ldots, q$) is joined to $T_0$ by an edge $E_i$, which corresponds to the bubble point $x_i$.
\item Each vertex $T_{ij}$ corresponds to the second-level bubble map $\widetilde{u}_{\infty, ij}$. Each vertex $T_{ij}$ is joined to $T_i$ by an edge $E_{ij}$, which corresponds to the bubble point $x_{ij}$.
\item And so on.
\end{itemize}
As we show in $\S$\ref{sec:bubble-construction}, the construction of the bubble tree, which primarily amounts to a careful choice of the rescaling maps $R_n$ and open sets $\Omega_n$, is essentially a formal process. Indeed, in that discussion we do not need the full strength of the associative Smith condition, but only the following properties of such maps:
\begin{itemize} \setlength\itemsep{-1mm}
\item The conformal invariance of the $3$-energy functional $E$.
\item The conformal invariance of the Smith equation from Proposition~\ref{prop:smith-u-conf-inv}.
\item Removal of singularities from Theorem~\ref{thm:rem-sing}.
\item The convergence result from Proposition~\ref{prop:conv-mod-bubbling}.
\item The energy gap from Proposition~\ref{prop:energy-gap}.
\end{itemize}

Analogous properties hold in several other conformally invariant settings, including harmonic maps~\cite{PW}, Yang--Mills connections~\cite{DK, Fe}, and holomorphic curves~\cite{MS, P, Ye}, leading to bubble trees in such settings.

\textbf{Zero Energy Loss and Zero Neck Length.} By virtue of~\eqref{eq:Radon-Conv}, we see that energy appears to be lost in the limit. That is,
$$\lim_{n \to \infty} E(u_n) \geq E(u_\infty).$$
Ideally, we would like to say that the discrepancy $\sum m_i$ is completely accounted for by the energies of the bubble maps. More precisely, we would like to say that each zeroth-level energy concentration $m_i$ is equal to the energy $E(\widetilde{u}_{\infty, i})$ of the first-level bubble map plus all of the the first-level energy concentrations $m_{ij}$. That is,
\begin{equation} \label{eq:ZeroEnergyLoss}
m_i = E(\widetilde{u}_{\infty, i}) + \sum_{j=1}^{q_i} m_{ij}.
\end{equation}
In turn, the first-level energy concentrations $m_{ij}$ ought to equal the sum of the energy of a second-level bubble map plus all of the second-level energy concentrations $m_{ijk}$, and so on. If this were the case, then we would indeed have
\begin{equation}
\lim_{n \to \infty} E(u_n) = E(u_\infty) + \sum_{I} E(\widetilde{u}_{\infty, I}) \label{eq:EnergyCons}
\end{equation}
essentially asserting that energy is preserved in the ``bubble tree limit''.

A priori, it is not obvious that~\eqref{eq:ZeroEnergyLoss} holds true, and we call the discrepancy $\tau_i$ the \textbf{energy loss}. That is, we define
\begin{align*}
\tau_i & = m_i - \Big( E(\widetilde{u}_{\infty, i}) + \sum_{j=1}^{q_i} m_{ij} \Big),
\end{align*}
and in $\S$\ref{sec:no-energy-loss} we prove:

\begin{thm}[No energy loss] \label{thm:no-energy-loss}
We have $\tau_i = 0$. Consequently,~\eqref{eq:ZeroEnergyLoss} and~\eqref{eq:EnergyCons} both hold.
\end{thm}

Now, although we have described the bubble tree in abstract terms, as a connected graph that records the various bubble maps and bubble points, it may also be regarded as a concrete geometric object. That is, we may regard the bubble tree as the subset of $M$ given by the union of the images of the base and bubble maps:
\begin{equation}
u_\infty(\Sigma) \cup \, \bigcup_I \widetilde{u}_{\infty, I}(S^3_{x_I}). \label{eq:bubbletree}
\end{equation}
We prove that this set is connected. In fact, in $\S$\ref{sec:zero-neck-length} we prove the stronger statement that:

\begin{thm}[Zero neck length] \label{thm:zero-neck-length}
We have
\begin{equation}
u_\infty(x_i) = \widetilde{u}_{\infty, i}(p^-). \label{eq:ZeroNeckLength}
\end{equation}
By extension, we have $\widetilde{u}_{\infty, I}(x_{Ij}) = \widetilde{u}_{\infty, Ij}(p^-)$.
\end{thm}

In fact, our proof of~\eqref{eq:ZeroNeckLength} establishes more: we prove the stronger result given in~\eqref{eq:ZeroNeckLength2} below. Moreover, combining~\eqref{eq:ZeroNeckLength2} and a straightforward adaptation of the proof of~\cite[Theorem 5.2.2(ii)]{MS}, we see that for large enough $k$, the map $u_{k}$ is homotopic to the connected sum $u_{\infty} \# \big( \#_{I} \widetilde{u}_{\infty, I} \big)$. Therefore homotopy is preserved in the bubble tree limit.

The proofs of Theorems~\ref{thm:no-energy-loss} and~\ref{thm:zero-neck-length} are significantly less trivial than the construction of the bubble tree. Both results rely on a deeper understanding of the energies of associative Smith maps on annuli
$$A = B^{\mathrm{outer}} \setminus B^{\mathrm{inner}},$$
where here $B^{\mathrm{inner}} \subset B^{\mathrm{outer}}$ are concentric $3$-balls. We devote $\S$\ref{sec:energy-annuli} to this study. Note that the results in $\S$\ref{sec:energy-annuli} require geometric properties of associative Smith maps beyond those used in the construction of the bubble tree, namely the energy gap of Proposition~\ref{Brian} and the energy identity of Corollary~\ref{energyminimizing}.

In Lemma~\ref{extension}, we explain how to cap off maps from the boundary $2$-spheres $\partial B^{\mathrm{outer}}$ and $\partial B^{\mathrm{inner}}$ to obtain maps from $3$-balls with controlled energy. We then patch these maps together to obtain a single map $S^3 \to M$ with small energy, which by the energy gap of Proposition~\ref{Brian} must be null-homotopic. As a result, we obtain a crucial bound (Proposition~\ref{annularenergy}) on the energy of associative Smith maps on annuli $A$ in terms of the energies on the boundary $2$-spheres $\partial B^{\mathrm{outer}}$ and $\partial B^{\mathrm{inner}}$.

In $\S$\ref{sec:no-energy-loss}, we prove Theorem~\ref{thm:no-energy-loss}. Intuitively, the idea is to ``trap'' $\tau_i$ in a sequence of carefully chosen annuli
$$A_k = B^{\mathrm{outer}}_k \setminus B^{\mathrm{inner}}_k.$$
That is, we reinterpret $\tau_i$ as a subsequential limit of energies on the annuli $A_k$ (in Lemma~\ref{innerboundary}(b)(i)), and observe that energies on the boundary $2$-spheres $\partial B^{\mathrm{outer}}_k$, $\partial B^{\mathrm{inner}}_k$ approach zero (in Lemma~\ref{lemma:rescaling}(c) and Lemma~\ref{innerboundary}(b)). Therefore, the bound (Proposition~\ref{annularenergy}) described in the previous paragraph implies that the energies on $A_k$ approach zero, whence $\tau_i = 0$.

Now, geometrically we think of the images of the outer $2$-spheres $u_k(\partial B^{\mathrm{outer}}_k)$ as being close to $u_\infty(x_i)$, and think of the images of the inner $2$-spheres $u_k(\partial B^{\mathrm{inner}}_k)$ as being close to $\widetilde{u}_{\infty, i}(p^-)$. In this way, we regard the images of the annuli $u_k(A_k) \subset M$ as ``necks'' of the bubble tree~\eqref{eq:bubbletree}.

This intuition leads to the beginning of the proof of Theorem~\ref{thm:zero-neck-length}. Indeed, choose points $y_k \in \partial B^{\mathrm{outer}}_k$ and $z_k \in \partial B^{\mathrm{inner}}_k$ and let $\widetilde{z}_k = R_k^{-1}(z_k) \subset S^3_{x_i}$. From~\eqref{eq:rescaled-maps} we get $u_k(z_k) = \widetilde{u}_{k,i}(\widetilde{z}_k)$. By the triangle inequality, we have
\begin{align*}
|u_\infty(x_i) - \widetilde{u}_{\infty,i}(p^-)| & \leq |u_\infty(x_i) - u_k(y_k)| + |u_k(y_k) - u_k(z_k)| + |\widetilde{u}_{k,i}(\widetilde{z}_k) - \widetilde{u}_{\infty, i}(p^-)| \\
& \leq |u_\infty(x_i) - u_k(y_k)| + \diam(u_k(A_k)) + |\widetilde{u}_{k,i}(\widetilde{z}_k) - \widetilde{u}_{\infty, i}(p^-)|.
\end{align*}
By our choice of $A_k$, it is relatively straightforward to see, using Lemma~\ref{lemma:rescaling}(c) and Lemma~\ref{innerboundary}(b), that the first and third terms approach zero as $k \to \infty$. The significantly more nontrivial matter, however, is establishing that the second term also approaches zero, that is:
\begin{equation}
\lim_{k \to \infty} \diam(u_k(A_k)) = 0. \label{eq:ZeroNeckLength2}
\end{equation}
In other words, the bubble tree~\eqref{eq:bubbletree} has no necks. In $\S$\ref{sec:zero-neck-length}, we prove~\eqref{eq:ZeroNeckLength2} by making use of both the mean value inequality of Theorem~\ref{thm:MVI} as well as the bound from Proposition~\ref{annularenergy}.

\subsection{Construction of the bubble tree} \label{sec:bubble-construction}

As in $\S$\ref{sec:bubble-overview}, we consider a sequence $u_n \colon (\Sigma^3, g) \to (M^7, h)$ of associative Smith maps with bounded $3$-energy:
$$E(u_n) \leq E_0 \quad \text{ for all $n \in \NN$}.$$
Let the base map $u_\infty \in C^1_{\loc}(\Sigma; M)$, the set of bubble points $\mathcal{S}$, and the energy concentrations $m_i \geq \frac{1}{2}\ep_0$ be as in~\eqref{eq:Setup} and~\eqref{eq:Radon-Conv}.

Suppose that $\mathcal{S} \neq \varnothing$ and fix a bubble point $x_i \in \mathcal{S}$. Choose a geodesic ball centered at $x_i$ that contains no other points in $\cS$, and identify it with $B(2)\subset \R^{3}$ via $\exp_{x_i}$ and a dilation. In this setting, the conclusion of Proposition~\ref{prop:conv-mod-bubbling} states that
$$u_n \to u_\infty \text{ in } C^1_{\loc}(B(2) \setminus \{0\}; M)$$
although not in $C^1_{\loc}(B(2); M)$, and that
$$|du_n|^3\,d\mu_g \to |du_\infty|^3\,d\mu_g + m_i\delta_0$$
as Radon measures on $B(1)$.

Given this setup, we now describe the construction of the bubble tree as a three step process. Steps one and two describe the base case, while step three indicates the induction step and contains a proof that the procedure eventually terminates.

\textbf{Step One: Choice of center points and dilation factors for rescaling.} We begin by focusing our attention on $x_i$. To that end, we let $\ep_k \in (0, \frac{1}{2})$ be a sequence of radii with $\ep_k \to 0$ and
\begin{equation} \label{eq:eps-choice}
E(u_\infty; B(2\ep_k)) = O(\tfrac{1}{k^2}).
\end{equation}
Having made this choice, we define a nested sequence of open balls $D_1(k) \subset D_2(k) \subset D_3(k) \subset D_4(k)$ centered at $0$ by:
\begin{align*}
D_1(k) & = B(\tfrac{1}{2 k^2} \ep_k), & D_2(k) & = B(\tfrac{1}{k^2} \ep_k), & D_3(k) & = B(\ep_k), & D_4(k) & = B(2\ep_k).
\end{align*}
We also fix, once and for all, a positive constant $\eta_0 > 0$ for which
$$\eta_0 < \tfrac{1}{16} \min\!\big(\tfrac{1}{3}\ep_0, \gamma_1\big),$$
where $\gamma_1$ is the energy gap constant of Proposition~\ref{Brian}.

Next, we choose center points and dilation factors by which to rescale. Our choice is given by the following lemma.

\begin{lemma} \label{lemma:rescaling}
The following results hold:
\begin{enumerate}[(a)]
\item There exists a subsequence $u_k := u_{n_k}$ such that:
\begin{align}
E(u_k; D_1(k)) & = m_i + O(\tfrac{1}{k^2}) \label{eq:rescaling-a1}, \\
E(u_k; D_4(k) \setminus D_1(k)) & = O(\tfrac{1}{k^2}) \label{eq:rescaling-a2}, \\
\left| u_k - u_\infty \right|_{1;D_4(k) \setminus D_1(k)} & = O(\tfrac{1}{k^2}). \label{eq:rescaling-a3}
\end{align}
\item There exist a sequence of centers $c_k \in D_2(k)$ and a sequence of radii $\lambda_k \in (0, \frac{\ep_k}{2 k^2}]$ such that
\begin{equation}
\eta_0 = E(u_k; D_4(k) \setminus B(c_k; \lambda_k)) \leq E(u_k; D_4(k) \setminus B(x; r)) \label{eq:eta-purpose}
\end{equation}
for all balls $B(x;r)$ with centers $x \in \overline{D_3(k)}$ and radii $r \leq \lambda_k$.
\item We have the following estimates:
\begin{align}
E(u_k; B(c_k; \ep_k)) & = m_i + O(\tfrac{1}{k^2}), \label{eq:rescaling-c1} \\
\ep_k \sup_{\partial B(c_k; \ep_k)} |du_k| & = o(1) \quad \text{ as } k \to \infty, \label{eq:rescaling-c2} \\
\lim_{k\to \infty}|u_{k} - u_{\infty}(0)|_{0; B(c_{k}; \ep_{k})\setminus B(c_{k}; \frac{1}{R} \ep_{k})} &= 0, \text{ for all } R > 1. \label{eq:rescaling-c3}
\end{align}
\end{enumerate}
\end{lemma}
\begin{proof}
(a) Since $u_n \to u_\infty$ in $C^1_{\loc}(B(2) \setminus \{0\}; M)$, and since we have the convergence of Radon measures $|du_n|^3\,d\mu_g \to |du_\infty|^3\,d\mu_g + m_i\delta_0$ on $B(1)$, for each fixed $k$, the following hold:
\begin{align*}
\lim_{n \to \infty} E(u_n; D_1(k)) & = m_i + E(u_\infty; D_1(k)), \\
\lim_{n \to \infty} E(u_n; D_4(k) \setminus D_1(k)) & = E(u_\infty; D_4(k) \setminus D_1(k)), \\
\lim_{n \to \infty} \left| u_n - u_\infty\right|_{1; D_4(k) \setminus D_1(k)} & = 0.
\end{align*}
Hence, because we chose $\ep_k$ to satisfy~\eqref{eq:eps-choice}, for each $k$, the relations~\eqref{eq:rescaling-a1},~\eqref{eq:rescaling-a2}, and~\eqref{eq:rescaling-a3} all hold for large enough choices of $n_k$. In this way, we obtain a subsequence $(u_k)$ satisfying~\eqref{eq:rescaling-a1},~\eqref{eq:rescaling-a2},~\eqref{eq:rescaling-a3}.

(b) Consider the quantity $E(u_k; D_4(k) \setminus B(x;r)) = \int_{B(2\ep_k) \setminus B(x;r)} |du_k|^3\,d\mu_g$. Note that it is continuous in $(r,x)$ on $[0, \ep_k] \times \overline{D_3(k)}$. Hence, for each $k$, the function
$$F_k(r) := \inf_{x \in \overline{D_3(k)}} E(u_k; D_4(k) \setminus B(x;r))$$
is continuous in $r$. In fact, $F_k \colon [0, \ep_k] \to \R$ is a decreasing function that satisfies
\begin{align*}
F_k(0) & = E(u_k; D_4(k)) = m_i + O(\tfrac{1}{k^2}), \\
F_k(\tfrac{1}{2 k^2} \ep_k) & \leq E(u_k; D_4(k) \setminus D_1(k)) = O(\tfrac{1}{k^2}).
\end{align*}
Recalling our choice of $\eta_{0}$ and that $m_{i} \geq \frac{\ep_{0}}{2}$, for each large enough $k$ we have
\[
F_k(\tfrac{1}{2 k^2} \ep_k) < \eta_{0} < F_{k}(0).
\]
Thus, for each sufficiently large $k$, there exists a smallest radius $\lambda_k \in [0, \frac{1}{2k^{2}}\ep_k]$ for which
$$F_k(\lambda_k) = \eta_0.$$

We now choose $c_k \in \overline{D_3(k)}$ to be a point that attains the infimum defining $F_k(\lambda_k)$, which exists because $\overline{D_3(k)}$ is compact. That is,
$$F_k(\lambda_k) = E( u_k; D_4(k) \setminus B(c_k; \lambda_k)).$$

In fact, we claim that eventually $|c_k| < \tfrac{1}{k^2} \ep_k$ for large enough $k$, meaning that $c_k \in D_2(k)$. Indeed, if we instead had $|c_k| \geq \tfrac{1}{k^2} \ep_k$, then $\lambda_k \leq \tfrac{1}{2 k^2} \ep_k$ implies that $D_1(k) \subset D_4(k) \setminus B(c_k; \lambda_k)$, whence
$$m_i + O(\tfrac{1}{k^2}) = E(u_k; D_1(k)) \leq E(u_k; D_4(k) \setminus B(c_k; \lambda_k)) = F_k(\lambda_k) = \eta_0.$$
But this contradicts $\eta_0 < \frac{1}{2}\ep_0 \leq m_i$.

(c) For~\eqref{eq:rescaling-c1}, since $|c_k| < \tfrac{1}{k^2} \ep_k$, we have $D_1(k) \subset B(c_k; \ep_k)$. Hence,
\begin{align*}
E(u_k; B(c_k; \ep_k)) & = E( u_k; D_1(k) ) + E(u_k; B(c_k; \ep_k) \setminus D_1(k)) \\
& = m_i + O(\tfrac{1}{k^2})
\end{align*}
where the last equality follows from~\eqref{eq:rescaling-a1} and~\eqref{eq:rescaling-a2} and the fact that $B(c_k; \ep_k) \subset D_4(k)$.

To prove~\eqref{eq:rescaling-c2}, note that since $\partial B(c_k; \ep_k) \subset D_4(k) \setminus D_1(k)$ for $k$ sufficiently large, by~\eqref{eq:rescaling-a3}, we get
$$\sup_{\partial B(c_k; \ep_k)} |du_k| \leq \sup_{\partial B(c_k; \ep_k)} |du_\infty| + O(\tfrac{1}{k^2}) = O(1),$$
where the boundedness of $|du_\infty|$ uses the fact that $u_\infty$ extends to a $C^1$ map on all of $B(1)$ by Theorem~\ref{thm:rem-sing}. It follows that
$$\ep_k \sup_{\partial B(c_k; \ep_k)} |du_k| = O(\ep_k) = o(1) \quad \text{ as } k \to \infty.$$

Finally, to see~\eqref{eq:rescaling-c3}, we observe that by the triangle inequality, for each $x \in B(c_{k}; \ep_{k})\setminus B(c_{k}; \frac{1}{R} \ep_{k})$, we have:
\begin{align*}
|u_k(x) - u_\infty(0)| & \leq |u_k(x) - u_\infty(x)| + |u_\infty(x) - u_\infty(0)| \\
& \leq \left| u_k - u_\infty \right|_{0;\, B(c_{k}; \ep_{k})\setminus B(c_{k}; \frac{1}{R} \ep_{k})} + C\ep_k \left| du_\infty \right|_{0; \,B(c_k; \ep_k)}.
\end{align*}
Since $c_{k} \in D_{2}(k)$, we see by the triangle inequality that, when $k$ is so large that $k^{2} > 2R$, there holds
\[
B(c_{k}; \ep_{k})\setminus B(c_{k};\tfrac{1}{R} \ep_{k}) \subset D_{4}(k)\setminus D_{1}(k).
\]
Thus we may use the estimate~\eqref{eq:rescaling-a3} and the fact that $u_{\infty}$ is $C^{1}$ on all of $B(1)$ to bound the last two terms in the above string of inequalities and get
\begin{align*}
|u_k(x) - u_\infty(0)| & \leq O(\tfrac{1}{k^2}) + O(\ep_k) \\
& = o(1) \quad \text{ as } k \to \infty.
\end{align*}
This proves the lemma.
\end{proof}

\textbf{Step Two: The rescaled maps and first level bubble points.} With the choices of center points $c_k$ and scale factors $\lambda_k$ of Lemma~\ref{lemma:rescaling} in place, we may now define the desired rescalings $R_k$.

We introduce some notation. Let $\sigma \colon S^3 \setminus \{p^-\} \to \R^3$ denote stereographic projection from the south pole $p^- \in S^3$. Also let $S^3_+$ and $S^3_-$ denote the upper and lower hemispheres, respectively. In particular,
\begin{equation*}
\sigma(p^+) = 0, \qquad \sigma(S^3_+) = B(1).
\end{equation*}
We also let $a_k \colon \R^3 \to \R^3$ denote the \emph{affine function} $a_k(x) = \lambda_k x + c_k$, so that
\begin{equation*}
a_k(B(x;r)) = B(\lambda_kx + c_k; \lambda_kr).
\end{equation*}
In particular,
\begin{equation*}
a_k(0) = c_k, \qquad a_k(B(1)) = B(c_k; \lambda_k).
\end{equation*}
We denote their composition by
$$R_k = a_k \circ \sigma \colon S^3 \setminus \{p^-\} \xrightarrow{\ \sigma \ } \R^3 \xrightarrow{\ a_k \ } \R^3$$
so that
\begin{equation*}
R_k(p^+) = c_k, \qquad R_k(S^3_+) = B(c_k; \lambda_k) \subset D_4(k).
\end{equation*}
Finally, we let $\Omega_k \subset S^3$ be the open sets for which
$$R_k(\Omega_k) = D_4(k).$$
Since $R_k$ preserves inclusions, we see that $S^3_+ \subset \Omega_k$. Note that the $\Omega_k$ are increasing and exhaust $S^3 \setminus \{p^-\}$.

\begin{lemma} \label{lemma:rescaled-maps}
The rescaled maps
\begin{align*}
& \widetilde{u}_{k,i} \colon \Omega_k \to M \\
& \widetilde{u}_{k,i} := u_k \circ R_k = u_k \circ a_k \circ \sigma
\end{align*}
are associative Smith maps with respect to a sequence of metrics $h_k$ on $\Omega_k \subset S^3 \setminus \{p^-\}$ that converge in $C^\infty_{\loc}$ to the round metric. Moreover, $\widetilde{u}_{k,i}$ has uniformly bounded $3$-energy.
\end{lemma}
\begin{proof}
By the conformal invariance of the Smith equation, the maps $u_k \circ a_k$ are associative Smith maps with respect to the metrics $\lambda_k^{-2} a_k^*g$ on $a_k^{-1}(D_4(k))$. Note that $\frac{1}{\lambda_k^2} a_k^*g \to g_{\euc}$ in $C^{\infty}_{\loc}$. Again by the conformal invariance of the Smith equation, the maps $\widetilde{u}_{k,i} = u_k \circ a_k \circ \sigma$ are associative Smith with respect to the metrics $\frac{1}{\lambda_k^2} R_k^*g = \frac{1}{\lambda_k^2} \sigma^*a_k^* g$ on $\Omega_k \subset S^3$, and $\frac{1}{\lambda_k^2} R_k^*g \to \sigma^*g_{\text{euc}}$ in $C^\infty_{\text{loc}}$. Finally, since $\sigma^*g_{\text{euc}}$ is conformal to the round metric, say $\sigma^*g_{\text{euc}} = b g_{\text{round}}$ for an appropriate function $b$, we see that $h_k := \frac{1}{b}\frac{1}{\lambda_k^2}R_k^*g \to g_{\text{round}}$ in $C^\infty_{\text{loc}}$.
\end{proof}

Therefore, by Lemma~\ref{lemma:rescaled-maps} and Proposition~\ref{prop:conv-mod-bubbling} applied to $\Omega = S^3_{x_i} \setminus \{p^-\}$ and $\Omega_k = R_k^{-1}(D_4(k))$, there exists a (possibly empty) finite set of points $\mathcal{S}_i = \{x_{i1}, \ldots, x_{iq_i}\} \subset S^3_{x_i} \setminus \{p^-\}$, called \textbf{first level bubble points}, and a $C^1$ associative Smith map $\widetilde{u}_{\infty, i} \colon (S^3_{x_i}, g_{\round}) \to M$, called a \textbf{first level bubble map}, such that after passing to a subsequence
\begin{equation*}
\widetilde{u}_{k,i} \to \widetilde{u}_{\infty,i} \quad \text{ in } C^1_{\loc}\big( S^3_{x_i} \setminus (\{p^-\} \cup \mathcal{S}_i) \big)
\end{equation*}
and such that, as Radon measures on $S^3_{x_i} \setminus \{p^-\}$, we have
\begin{equation}
|d\widetilde{u}_{k,i}|^3\,d\mu_{h_k} \to |d\widetilde{u}_{\infty,i}|^3\,d\mu_{\round} + \sum_{j=1}^{q_i} m_{ij}\delta(x_{ij}) =: \kappa_i \label{eq:FirstLevRadonConv}
\end{equation}
with each $m_{ij} \geq \frac{1}{2}\ep_0$. In fact, we claim that all of the first level bubble points lie in the closure of the upper hemisphere. This is part (d) of the following result.

\begin{lemma} \label{lemma:kappa-i}
The following results hold:
\begin{enumerate}[(a)]
\item $\kappa_i(S^3_{x_i} \setminus \{p^-\}) \leq m_i$,
\item $\kappa_i(\overline{S^3_+}) \geq m_i - \eta_0$,
\item $\kappa_i(S^3_- \setminus \{p^-\}) \leq \eta_0$,
\item $\mathcal{S}_i \subset \overline{S^3_+}$.
\end{enumerate}
\end{lemma}
\begin{proof}
(a) By conformal invariance of the $3$-energy and~\eqref{eq:rescaling-a1}--\eqref{eq:rescaling-a2}, we have
\begin{equation}
E_{h_k}( \widetilde{u}_{k,i}; \Omega_k) = E_g(u_k; D_4(k)) = m_i + O(\tfrac{1}{k^2}). \label{eq:Omega-MI}
\end{equation}
Hence, using~\eqref{eq:FirstLevRadonConv} and that $\Omega_\ell$ is an open set, for each $\ell$ we have:
$$\kappa_i(\Omega_\ell) \leq \liminf_{k \to \infty} E_{h_k}(\widetilde{u}_{k,i}; \Omega_\ell) \leq \liminf_{k \to \infty} E_{h_k}(\widetilde{u}_{k,i}; \Omega_k) = m_i.$$
Hence, $\kappa_i(S^3_{x_i} \setminus \{p^-\}) \leq m_i$.

(b) By conformal invariance of the $3$-energy and~\eqref{eq:eta-purpose}, we have
\begin{equation}
E_{h_k}(\widetilde{u}_{k,i}; \Omega_k \setminus \overline{S^3_+}) = E_g(u_k; D_4(k) \setminus B(c_k; \lambda_k)) = \eta_0. \label{eq:OmegaComp-Eta0}
\end{equation}
Thus, using~\eqref{eq:FirstLevRadonConv} and the fact that $\overline{S^3_+}$ is compact, followed by~\eqref{eq:Omega-MI} and~\eqref{eq:OmegaComp-Eta0}, we have
\begin{align*}
\kappa_i(\overline{S^3_+}) & \geq \limsup_{k \to \infty} E_{h_k}(\widetilde{u}_{k,i}; \overline{S^3_+}) \\
& = \limsup_{k \to \infty} \left[ E_{h_k}(\widetilde{u}_{k,i}; \Omega_k) - E_{h_k}(\widetilde{u}_{k,i}; \Omega_k \setminus \overline{S^3_+}) \right] \\
& = m_i - \eta_0.
\end{align*}
(c) This follows immediately from parts (a) and (b):
$$\kappa_i(S^3_- \setminus \{p^-\}) = \kappa_i(S^3_{x_i} \setminus \{p^-\}) - \kappa_i(\overline{S^3_+}) \leq \eta_0.$$
(d) On the one hand,
\begin{align}
\kappa_i(S^3_- \setminus \{p^-\}) & = E(\widetilde{u}_{\infty, i}; S^3_- \setminus \{p^-\}) + \sum_{j \,:\, x_{ij} \in S^3_-} m_{ij} \geq \sum_{j \,:\, x_{ij} \in S^3_-} m_{ij}. \label{eq:43c-1}
\end{align}
On the other hand, by part (c), we have
\begin{equation}
\kappa_i(S^3_- \setminus \{p^- \}) \leq \eta_0 < \tfrac{1}{2} \ep_0 \leq m_{ij} \label{eq:43c-2}
\end{equation}
for all $1 \leq j \leq q_i$. Comparing~\eqref{eq:43c-1} with~\eqref{eq:43c-2}, we conclude that none of the points in $\mathcal{S}_i$ can be in $S^3_-$. That is, $\mathcal{S}_i \subset \overline{S^3_+}$.
\end{proof}

\textbf{Step Three: Iteration and finite termination.} If any of the sets $\mathcal{S}_i$ is nonempty, then Steps One and Two above may be repeated, resulting in second-level bubble maps $\widetilde{u}_{\infty, ij} \colon (S^3_{x_{ij}}, g_{\round}) \setminus \{p^-\} \to M$ and a (possibly empty) finite set $\mathcal{S}_{ij} \subset S^3_{x_{ij}} \setminus \{p^-\}$ of second level bubble points. If any of the $\mathcal{S}_{ij}$ is nonempty, then we may repeat the process again, and so on.

As explained above in $\S$\ref{sec:bubble-overview}, iterating this procedure results in a tree $(\{T_0, T_I\}, \{E_I\})$ whose base vertex $T_0$ corresponds to the base map $u_\infty$, whose higher level vertices $T_I$ correspond to bubble maps $\widetilde{u}_{\infty, I}$, and whose edges $E_I$ correspond to bubble points $x_I$. In the tree, the edge $E_{i_1 \cdots i_k i_{k+1}}$ joins a $k$th level vertex $T_{i_1 \cdots i_k}$ to a $(k+1)$st level vertex $T_{i_1 \cdots i_k i_{k+1}}$.

Thus, to conclude the construction, it remains only to show that the bubble tree is finite, that is that the sets $\mathcal{S}_{i_1 \cdots i_k}$ of $k$th level bubble points are all empty for $k$ sufficiently large. To do this, we demonstrate that all of the $(k+1)$st level energy concentrations $m_{i_1 \cdots i_{k+1}i_{k+2}}$ are a fixed constant less than the $k$th level energy concentrations $m_{i_1 \cdots i_{k+1}}$. More precisely we have the following result.

\begin{lemma} \label{lemma:mij}
We have $m_{ij} \leq m_i - \frac{1}{2}\eta_0$ for each $1 \leq j \leq q_i$.
\end{lemma}
\begin{proof}
By definition of $\kappa_i$ followed by Lemma~\ref{lemma:kappa-i}(a), we have
\begin{equation}
E(\widetilde{u}_{\infty, i}; S^3_{x_i} \setminus \{p^-\}) + \sum_{j=1}^{q_i} m_{ij} = \kappa_i(S^3_{x_i} \setminus \{p^-\}) \leq m_i. \label{eq:441}
\end{equation}
Suppose for the sake of contradiction that the first-level bubble point $x_{i1} \in \mathcal{S}_i = \{x_{i1}, \ldots, x_{iq_i}\}$ has energy concentration $m_{i1} > m_i - \frac{1}{2}\eta_0$. Then~\eqref{eq:441} implies
\begin{align*}
E(\widetilde{u}_{\infty, i}; S^3_{x_i} \setminus \{p^-\}) + \sum_{j=2}^{q_i} m_{ij} \leq m_i - m_{i1} < \tfrac{1}{2}\eta_0 < \tfrac{1}{6}\ep_0.
\end{align*}
By Proposition~\ref{prop:energy-gap}, we see that $\widetilde{u}_{\infty, i}$ is a constant map. Moreover, since each $m_{ij} \geq \frac{1}{2}\ep_0$, we must have $\mathcal{S}_i = \{x_{i1}\}$. Thus, the Radon measure $\kappa_i = |d\widetilde{u}_{\infty,i}|^3 d\mu_{\round} + \sum m_{ij}\delta(x_{ij})$ is simply
$$\kappa_i = m_{i1}\delta(x_{i1}).$$
Let $B = B^{S^3}(x_{i1}; \rho)$ denote a small geodesic ball in $S^3$ centered at $x_{i1}$. Using the conformal invariance of $E$ first, followed by~\eqref{eq:FirstLevRadonConv} and $\kappa_i = m_{i1}\delta(x_{i1})$ second, and our hypothesis third, we obtain
$$\lim_{k \to \infty} E_g(u_k; R_k(B)) = \lim_{k \to \infty} E_{h_k}(\widetilde{u}_{k,i}; B) = m_{i1} > m_i - \tfrac{1}{2}\eta_0.$$
On the other hand,~\eqref{eq:rescaling-a1}--\eqref{eq:rescaling-a2} give
$$E_g(u_k; D_4(k)) = m_i + O(\tfrac{1}{k^2}).$$
We deduce that for $k$ sufficiently large,
\begin{equation}
E_g(u_k; D_4(k) \setminus R_k(B)) < \tfrac{1}{2}\eta_0 + O(\tfrac{1}{k^2}). \label{eq:442}
\end{equation}
Now, by Lemma~\ref{lemma:kappa-i}(d), we have $x_{i1} \in \overline{S^3_+}$, and hence $R_k(x_{i1}) \in \overline{B(c_k;\lambda_k)}.$ By choosing the radius $\rho > 0$ smaller if necessary, we may suppose that $R_k(B)$ is contained in a ball $B(z; r)$ with center $z \in \overline{B(c_k; \lambda_k)} \subset \overline{D_3(k)}$ and radius $r < \lambda_k$. But then~\eqref{eq:442} implies that
$$E_g(u_k; D_4(k) \setminus B(z;r)) \leq E_g(u_k; D_4(k) \setminus R_k(B)) < \eta_0,$$
which contradicts~\eqref{eq:eta-purpose}.
\end{proof}

Iterating Lemma~\ref{lemma:mij} shows that the bubble tree has finitely many levels. Indeed, at a $k$th level bubble point $x_I = x_{i_1 \cdots i_{k+1}}$, the energy concentration $m_I = m_{i_1 \cdots i_{k+1}}$ at $x_I$ satisfies
$$m_I \leq m_{i_1} - k \big (\tfrac{1}{2}\eta_0 \big).$$
On the other hand, the energy concentration $m_I$ must satisfy
$$m_I \geq \tfrac{1}{2}\ep_0.$$
Together, these two inequalities yield $k(\frac{1}{2}\eta_0) \leq m_{i_1} - \frac{1}{2}\ep_0$, an upper bound on $k$. That is, for $k$ is sufficiently large, the set $\bigcup_{|I| = k} \mathcal{S}_I$ of $k$th level bubble points is empty, and hence the bubble tree terminates.

\subsection{Energy of Smith maps on annuli} \label{sec:energy-annuli}

In this section we show (Proposition~\ref{annularenergy}) that the $3$-energies of associative Smith maps on annular regions are controlled by their $3$-energies on the boundary spheres. One could view this as an isoperimetric-type estimate. As mentioned in $\S$\ref{sec:bubble-overview}, such an estimate is crucial to the proofs of Theorems~\ref{thm:no-energy-loss} and~\ref{thm:zero-neck-length}. 

We begin by recalling the following well-known extension result (compare with~\cite[Remark 4.4.2]{MS}). For completeness we include a proof.
\begin{lemma} \label{extension}
Suppose $u: S^{2} \to M$ is a $C^{1}$-map and that
\begin{equation*}
\diam_{M}\left( u(S^{2}) \right) < \inj_{M}.
\end{equation*}
Then there exists a Lipschitz map $v: B(1) \to M$ such that $v\rvert_{S^{2}} = u$ and 
\begin{equation*}
\int_{B(1)}|Dv|^{3}dx \leq K\int_{S^{2}}|du|^{3},
\end{equation*}
where $K$ depends only on the geometry of $M$.
\end{lemma}
\begin{proof}
Fix an arbitrary $\xi_{0} \in S^{2}$ and let $p = u(\xi_{0})$. Then, by assumption, it makes sense to talk about the map $\tilde{u}: S^{2} \to \R^{7}$ given by
\begin{equation*}
\tilde{u}(\xi) = \exp_{p}^{-1}\left( u(\xi) \right).
\end{equation*}
Note that $|\tilde{u}(\xi)| = \dist_{M}(u(\xi), p)$. Next we define $v: B(1) \to M$ to be the homogeneous degree-1 extension of $u$. Specifically, let
\begin{equation*}
v(r\xi) = \exp_{p}\left( r\tilde{u}(\xi) \right) \quad \text{ for }r \in [0, 1], \xi \in S^{2}.
\end{equation*}
Then a straightforward calculation shows that
\begin{align*}
|Dv(r\xi)|^{3} & = \left( |\partial_{r}v|^{2} + r^{-2}|\partial_{\xi}v|^{2}\right)^{\frac{3}{2}} \\
& \leq C\left( |\tilde{u}(\xi)|^{3} + |du|^{3} \right).
\end{align*}
Therefore we have
\begin{align*}
\int_{B(1)} |Dv|^{3} dx & = \int_{0}^{1}r^{2}\int_{S^{2}}|Dv(r\xi)|^{3}d\xi dr \\
& \leq C\int_{S^{2}} ( |\widetilde{u}|^{3} + |du|^{3} ) d\xi \\
& \leq C\int_{S^{2}} ( |d\widetilde{u}|^{3} + |du|^{3} ) d\xi \\
& \leq C\int_{S^{2}}|du(\xi)|^{3}d\xi,
\end{align*}
where we used the Poincar\'e inequality for $W^{1, 3}$-functions on $S^{2}$ in the second-to-last inequality.
\end{proof}

\begin{prop} \label{annularenergy}
Let $u: B(2) \to M$ be a $C^{1}$-Smith map with respect to a metric which satisfies
\begin{equation} \label{comparablemetric2}
\tfrac{1}{2}|\xi|^{2} \leq g_{\alpha\beta}(x)\xi^{\alpha}\xi^{\beta} \leq 2|\xi|^{2}, \quad \text{ for all }x \in B(2), \xi \in \R^{3}.
\end{equation}
Suppose further that for some $0 < r < 1$ we have
\begin{equation} \label{smallboundarylength}
\diam_{M}\left( u(\partial B(r) ) \right) + \diam_{M}\left( u(\partial B(1)) \right) < \inj_{M},
\end{equation}
and that
\begin{equation} \label{smallannularenergy}
\int_{B(1)\setminus B(r)}|du|^{3}d\mu_{g} + Kr\int_{\partial B(r)}|du|^{3}dS_{g} + K\int_{\partial B(1)}|du|^{3}dS_{g} < \tfrac{1}{8} \gamma_1,
\end{equation}
with $\gamma_{1}$ coming from Proposition~\ref{Brian} and $K$ from Lemma~\ref{extension}. Then there holds
\begin{equation*}
\int_{B(1)\setminus B(r)} |du|^{3}d\mu_{g} \leq C \Big( r\int_{\partial B(r)}|du|^{3}dS_{g} + \int_{\partial B(1)}|du|^{3}dS_{g} \Big),
\end{equation*}
where $C$ depends only on the geometry of $M$. In particular $C$ is independent of $r$.
\end{prop}
\begin{proof}
By assumption~\eqref{smallboundarylength}, we may apply Lemma~\ref{extension} to $u\rvert_{\partial B(1)}$ and (a suitable rescaling of) $u\rvert_{\partial B(r)}$ to obtain Lipschitz maps $v: B(1) \to M$ and $w: B(r) \to M$ such that 
\begin{equation*}
v\rvert_{\partial B(1)} = u\rvert_{\partial B(1)}, \qquad w\rvert_{\partial B(r)} = u\rvert_{\partial B(r)},
\end{equation*}
\begin{equation*}
\int_{B(1)}|Dv|^{3}dx \leq K\int_{\partial B(1)}|du|^{3},
\end{equation*}
\begin{equation*}
\int_{B(r)} |Dw|^{3}dx \leq Kr\int_{\partial B(r)} |du|^{3}.
\end{equation*}
Note that $u|_{B(1)\setminus B(r)}$ and $w$ together form a Lipschitz map on $B(1)$ into $M$, which can be pulled back to a Lipschitz map from $S^{3}_{+}$ into $M$ via the stereographic projection. Similarly, $v$ gives rise to a Lipschitz map from $S^{3}_{-}$ into $M$. The two maps agree on the equator, and we thus obtain a Lipschitz map $f: S^{3} \to M$. Moreover, by the conformal invariance of the $3$-energy, we see that 
\begin{align*}
\int_{S^{3}}|Df|^{3}dx & = \int_{B(1)\setminus B(r)} |Du|^{3} dx + \int_{B(r)}|Dw|^{3}dx + \int_{B(1)}|Dv|^{3}dx\\
& \leq \int_{B(1)\setminus B(r)} |Du|^{3} dx + Kr\int_{\partial B(r)} |du|^{3} + K\int_{\partial B(1)}|du|^{3} \\
& \leq 8 \Big (\int_{B(1)\setminus B(r)}|du|^{3}d\mu_{g} + Kr\int_{\partial B(r)}|du|^{3}dS_{g} + K\int_{\partial B(1)}|du|^{3}dS_{g} \Big) \\
& < \gamma_{1}.
\end{align*}
where we used~\eqref{comparablemetric2} in the second to last inequality, and~\eqref{smallannularenergy} in the last one. But then by Proposition~\ref{Brian}, the map $f$ is null-homotopic and hence extends to a Lipschitz map $F:B^{4}(1) \to M$. We may then invoke Corollary~\ref{energyminimizing} with $\Sigma = S^{3}$ and $W = B^{4}(1)$ to obtain
\begin{align*}
\int_{B(1)\setminus B(r)}|du|^{3}d\mu_{g} & \leq \int_{B(r)}|dw|^{3}d\mu_{g} + \int_{B(1)}|dv|^{3}d\mu_{g} \\
& \leq C\Big( r\int_{\partial B(r)}|du|^{3}dS_{g} + \int_{\partial B(1)}|du|^{3}dS_{g} \Big),
\end{align*}
and the proof is complete.
\end{proof}

\subsection{No energy loss} \label{sec:no-energy-loss}

This section is devoted to the proof of Theorem~\ref{thm:no-energy-loss} (No energy loss). Recall that, by the definition of $\tau_i$ and the measure $\kappa_i$, we have $\tau_i = m - \kappa_i(S^3 \setminus \{p^-\})$. To streamline notation, in this section we omit the subscript $i$, writing $m := m_{i}$ and $\tau := \tau_{i}$ and $\kappa := \kappa_{i}$. We also write $v_{k} := \widetilde{u}_{k, i}$ and $v := \widetilde{u}_{\infty, i}$. Finally we write $\cS_{1} = \cS_{i}$.

As a preliminary step, we note the following consequence of Lemma~\ref{lemma:kappa-i}, which says that the energy loss $\tau$ is nonnegative, and lies below the threshold $\eta_{0}$ we chose in Lemma~\ref{lemma:rescaling}. 

\begin{lemma} \label{energylossthreshold}
We have $0 \leq \tau \leq \eta_{0}$.
\end{lemma}
\begin{proof}
The lower bound $\tau \geq 0$ is exactly the statement of Lemma~\ref{lemma:kappa-i}(a). For the upper bound, observe that Lemma~\ref{lemma:kappa-i}(b) gives
\begin{align*}
\tau & = m - \kappa (S^3 \setminus \{p^-\}) \leq m - \kappa (\overline{S^3_+}) \leq \eta_0
\end{align*}
as desired.
\end{proof}
Next, as explained in \S\ref{sec:bubble-overview}, we want to identify (after passing to a further subsequence if necessary) suitable annuli $A_{l}$ for which 
\begin{equation*}
\tau= \lim_{l \to \infty}E(u_{l}; A_{l}).
\end{equation*}
The outer boundaries of $A_{l}$ are a subsequence $\partial B(c_{k_{l}}; \ep_{k_{l}})$ of the spheres $\partial B(c_k; \ep_k)$ constructed in \S\ref{sec:bubble-construction}. Choosing this subsequence and finding the inner boundaries is the content of Lemma~\ref{innerboundary} below. The proof of Theorem~\ref{thm:no-energy-loss} is given after that of Lemma~\ref{innerboundary}.

We introduce the following notation: For $r > 0$ we define the open sets 
\begin{equation*}
G(r) = \sigma^{-1}\left( B(r) \right) \subset S^{3},
\end{equation*}
so that 
\begin{equation*}
R_{k}\left( G(r) \right) = B(c_{k}; r\lambda_{k}).
\end{equation*}
Then, since the sets $G( \sqrt{k} )$ exhaust $S^{3} \setminus \{p^{-}\}$, by the definition of $\tau$ we know that $\kappa\big( G(\sqrt{l}) \big)$ increases to $\kappa (S^{3} \setminus \{p^{-}\}) = m - \tau$ as $l \to \infty$. In other words, 
\begin{equation} \label{bubbleenergy}
\kappa\big( G(\sqrt{l}) \big) = m - \tau + o(1) \quad \text{ as }l \to \infty.
\end{equation}
Moreover, since $v$ is a $C^{1}$-map on $S^{3}$, we have 
\begin{equation} \label{mapenergy}
E(v; S^{3}\setminus G(\sqrt{l}))= o(1) \quad \text{ as }l \to \infty.
\end{equation}

\begin{lemma} \label{innerboundary}
We have:
\begin{enumerate}[(a)]
\item There exists a subsequence $(v_{k_{l}})$ of $(v_{k})$ such that the following hold:
\begin{equation} \label{innerboundary1}
E_{h_{k_{l}}}(v_{k_{l}}; G(\sqrt{l})) = m - \tau + o(1),
\end{equation}
\begin{equation} \label{innerboundary2}
E_{h_{k_{l}}}(v_{k_{l}}; G(l^{2})\setminus G(\sqrt{l})) = o(1),
\end{equation}
\begin{equation} \label{innerboundary3}
|v_{k_{l}} - v|_{1; G(l^{2}) \setminus G(\sqrt{l})} = o(1).
\end{equation}
\item Scaling back to $(u_{k_{l}})$, writing $u_{l}: = u_{k_{l}}$ and $c_{l}:= c_{k_{l}}$, and setting
\begin{equation*}
\rho_{l} = \ep_{k_{l}},\ \sigma_{l} = l\lambda_{k_{l}},
\end{equation*}
we have
\begin{align} \label{eq:innerboundary-b1}
E( u_{l}; B(c_{l}; \rho_l) \setminus B(c_{l}; \sigma_l) ) & = \tau + o(1), \\ \label{eq:innerboundary-b2}
l \lambda_{k_{l}} \sup_{\partial B(c_{k_{l}}; l\lambda_{k_{l}})} |du_{l}| & = o(1),
\end{align}
and
\begin{equation} \label{eq:innerboundary-b3}
\begin{aligned}
\lim\limits_{l \to \infty} \left| u_{l} - v(p^{-}) \right|_{0; B(c_{k_{l}}; Rl\lambda_{k_{l}}) \setminus B(c_{k_{l}}; l\lambda_{k_{l}})} & = 0 \quad \text{ for all $R > 1$}, \\
\text{or equivalently, } \quad \lim\limits_{l \to \infty} \left| v_{k_{l}} - v(p^{-}) \right|_{0; G(Rl) \setminus G(l)} & = 0 \quad \text{ for all $R > 1$}.
\end{aligned}
\end{equation}
\end{enumerate}
\end{lemma}
\begin{proof}
(a) Recall that, as Radon measures on $S^{3} \setminus \{p^{-}\}$ we have
\begin{equation*}
|dv_{k}|^{3}d\mu_{h_{k}} \to |dv|^{3}d\mu + \sum_{x_{ij} \in \cS_{1}} m_{ij}\delta(x_{ij} ) = \kappa, 
\end{equation*}
where $\cS_{1}\subseteq \overline{S^{3}_{+}}$. This has the following two consequences. First, for all $l > 1$,we have $\kappa (\partial G(\sqrt{l})) = 0$, which implies that 
\begin{equation*}
\kappa(G(\sqrt{l})) = \lim_{k \to \infty}E_{h_{k}}(v_{k}; G(\sqrt{l})).
\end{equation*}
Secondly, since $v_{k}$ converges in $C^{1}$ to $v$ on compact subsets of $S^{3}\setminus \left( \cS_{1} \cup \{p^{-}\} \right)$, we infer that, for all $l > 1$,
\begin{equation*}
\lim_{k \to \infty} |v_{k} - v|_{1; G(l^{2})\setminus G(\sqrt{l})} = 0.
\end{equation*}
From these facts and~\eqref{bubbleenergy},~\eqref{mapenergy}, we easily see that there is a subsequence $(v_{k_{l}})$ such that~\eqref{innerboundary1},~\eqref{innerboundary2} and~\eqref{innerboundary3} all hold.

(b) To prove~\eqref{eq:innerboundary-b1}, recall from Lemma~\ref{lemma:rescaling}(c) that 
\begin{equation*}
E(u_{l}; B(c_{l}; \rho_{l})) = m + o(1).
\end{equation*}
On the other hand,~\eqref{innerboundary1} and~\eqref{innerboundary2} imply that 
\begin{equation*}
E(u_{l}; B(c_{l}; \sigma_{l})) = m - \tau + o(1).
\end{equation*}
Subtracting this from the previous equality gives~\eqref{eq:innerboundary-b1}. 

For~\eqref{eq:innerboundary-b2}, we note that if $x \in \partial B(c_{k_{l}}; l\lambda_{k_{l}})$, then eventually $B(x; \frac{1}{2} l\lambda_{k_{l}}) \subseteq B(c_{k_{l}}; l^{2}\lambda_{k_{l}}) \setminus B(c_{k_{l}}; \sqrt{l}\lambda_{k_{l}})$. Hence, by~\eqref{innerboundary2}, 
\begin{equation*}
E\big( u_{l}; B(x; \tfrac{1}{2} l\lambda_{k_{l}}) \big) \leq E(u_{l}; B(c_{k_{l}}; l^{2}\lambda_{k_{l}}) \setminus B(c_{k_{l}}; \sqrt{l}\lambda_{k_{l}})) = o(1),
\end{equation*}
which means for large enough $l$ we may apply Theorem~\ref{thm:MVI} to get
\begin{equation*}
\left(l\lambda_{k_{l}}\right)^{3} |du_{l}(x)|^{3} \leq E(u_{l}; B(c_{k_{l}}; l^{2}\lambda_{k_{l}}) \setminus B(c_{k_{l}}; \sqrt{l}\lambda_{k_{l}})).
\end{equation*}
Since $x \in \partial B(c_{k_{l}}; l\lambda_{k_{l}})$ is arbitrary, we conclude that 
\begin{equation*}
\left(l\lambda_{k_{l}}\right)^{3} \sup_{ \partial B(c_{k_{l}}; l\lambda_{k_{l}})} |du_{l}|^{3} \leq E(u_{l}; B(c_{k_{l}}; l^{2}\lambda_{k_{l}}) \setminus B(c_{k_{l}}; \sqrt{l}\lambda_{k_{l}})) = o(1).
\end{equation*}

To prove~\eqref{eq:innerboundary-b3}, we take an arbitrary $y \in G(Rl)\setminus G(l)$ and estimate
\begin{align*}
|v_{l}(y) - v(p^{-}) | & \leq |v_{l}(y) - v(y)| + |v(y) - v(p^{-})|\\
& \leq |v_{l} - v|_{0; G(Rl) \setminus G(l)} + |v - v(p^{-})|_{0; S^{3} \setminus G(l)}.
\end{align*}
Thus we obtain, for all $l > R$, that 
\begin{equation*}
|v_{l} - v(p^{-})|_{0; G(Rl) \setminus G(l)} \leq |v_{l} - v|_{0; G(l^{2}) \setminus G(\sqrt{l})} + |v - v(p^{-})|_{0; S^{3} \setminus G(l)}.
\end{equation*}
To finish the proof, we recall~\eqref{innerboundary3} and also note that $ |v - v(p^{-})|_{0; S^{3} \setminus G(l)} = o(1)$ as $l \to \infty$, since $v \in C^{1}(S^{3}; M)$ and $S^{3} \setminus G(l)$ converges to $\{p^{-}\}$ in Hausdorff distance as $l \to \infty$.
\end{proof}

We are now ready for the proof of Theorem~\ref{thm:no-energy-loss}.
\begin{proof}[Proof of Theorem~\ref{thm:no-energy-loss}] 
We maintain the notation of Lemma~\ref{innerboundary}. In particular, we continue to write
\begin{equation*}
\rho_{l} = \ep_{k_{l}}, \qquad \sigma_{l} = l \lambda_{k_{l}}.
\end{equation*}
(Note that $\rho_{l}, \sigma_{l} \to 0$ while $\frac{\rho_l}{\sigma_l} \to \infty$ as $l \to \infty$.) Furthermore, we define the following annular regions
\begin{equation*}
A_{l} = B(c_l; \rho_l) \setminus B(c_l; \sigma_l).
\end{equation*}
Then, by Lemma~\ref{lemma:rescaling}(c) and Lemma~\ref{innerboundary}(b), we have
\begin{equation} \label{boundaryenergysmall}
\rho_{l}\int_{\partial B(c_l; \rho_l)} |du_{l}|^{3}dS_{g} + \sigma_{l}\int_{\partial B(c_l; \sigma_l)}|du_{l}|^{3}dS_{g} = o(1) \text{ as }l \to \infty.
\end{equation}
This estimate then allows us to apply Proposition~\ref{annularenergy} to suitable rescalings of the maps $u_{l}$. Specifically, we show that the rescaled maps
\begin{equation*}
\tilde{u}_{l}(x) = u_{l}(c_{l} + \rho_{l}x)
\end{equation*}
eventually satisfy the hypotheses~\eqref{smallboundarylength} and~\eqref{smallannularenergy}.

To see this, note that we have by equation \eqref{eq:rescaling-c2} and equations~\eqref{eq:innerboundary-b1}--\eqref{eq:innerboundary-b2} that, 
\begin{equation*}
\lim_{l \to \infty}\Big(E(u_{l}; A_{l}) + K\rho_{l}\int_{\partial B(c_l; \rho_l)}|du_{l}|^{3}dS_{g} + K\sigma_{l}\int_{\partial B(c_l; \sigma_l)} |du_{l}|^{3}dS_{g}\Big) = \tau,
\end{equation*}
where $K$ is as in Lemma~\ref{extension}. Since $\tau \leq \eta_{0} < \frac{1}{16} \gamma_{1}$, we see that there exists $l_{0}$ such that for all $l \geq l_{0}$, the rescaled maps $\tilde{u}_{l}$ satisfy~\eqref{smallannularenergy} with $r = \frac{\sigma_l}{\rho_l}$. Moreover, increasing $l_{0}$ if necessary, we see that~\eqref{eq:rescaling-c3} and~\eqref{eq:innerboundary-b3} yield~\eqref{smallboundarylength} for all $l \geq l_{0}$. Therefore we may invoke Proposition~\ref{annularenergy} to conclude that 
\begin{equation*}
E(u_{l}; A_{l}) \leq C \Big( \rho_{l}\int_{\partial B(c_l; \rho_l)}|du_{l}|^{3}dS_{g} + \sigma_{l}\int_{\partial B(c_l; \sigma_l)} |du_{l}|^{3}dS_{g} \Big),
\end{equation*}
with $C$ independent of $l$. Letting $l \to \infty$ and recalling~\eqref{eq:innerboundary-b1} and~\eqref{boundaryenergysmall}, we get $\tau = 0$ as desired.
\end{proof}

\subsection{Zero neck length} \label{sec:zero-neck-length}

In this section we continue to use notation from $\S$\ref{sec:energy-annuli} and $\S$\ref{sec:no-energy-loss}. Our goal here is to prove Theorem~\ref{thm:zero-neck-length} (Zero neck length), which in the present notation reads 
\[u_{\infty}(0) = v(p^{-}).
\]
In fact, it is enough to show that
\begin{equation} \label{neckdiameter}
\lim_{l \to \infty} \diam\left( u_{l}\left( B(c_l; \tfrac{1}{4} \rho_l) \setminus B(c_l; 4 \sigma_l) \right) \right) = 0.
\end{equation}
To see this, note that in view of~\eqref{eq:rescaling-c3} and~\eqref{eq:innerboundary-b3} with $R = 4$, the above implies that
\begin{equation*}
\lim_{l \to \infty} \diam \left( u_{l}\left( A_{l} \right) \right) = 0,
\end{equation*}
which is~\eqref{eq:ZeroNeckLength2}. Then as indicated in \S\ref{sec:bubble-overview}, with the help of~\eqref{eq:rescaling-c3} and~\eqref{eq:innerboundary-b3} again, we may finish the proof of Theorem~\ref{thm:zero-neck-length}. That is, we choose points $y_{l} \in \partial B(c_{l}; \rho_{l})$ and $z_{l} \in \partial B(c_{l}; \sigma_{l})$ and estimate
\begin{align*}
|u_\infty(0) - v(p^-)| & \leq |u_\infty(0) - u_l(y_l)| + |u_l(y_l) - u_l(z_l)| + |u_{l}(z_{l}) - v(p^-)| \\
& \leq |u_\infty(0) - u_l|_{0; \partial B(c_{l}; \rho_{l})} + \diam(u_l(A_l)) + |u_{l} - v(p^-)|_{0; \partial B(c_{l}; \sigma_{l})}.
\end{align*}
By~\eqref{eq:rescaling-c3},~\eqref{eq:ZeroNeckLength2} and~\eqref{eq:innerboundary-b3}, all three terms in the last line tend to zero as $l \to \infty$, and we are done. Therefore, it remains to prove~\eqref{neckdiameter}.

As a first step towards proving~\eqref{neckdiameter}, we establish the following pointwise gradient estimates.
\begin{lemma} \label{pointwiseannular} For $l$ sufficiently large, we have
\begin{equation} \label{pointwiseannularbound}
\sup_{x \in B(c_{l}; r_{2})\setminus B(c_{l}; r_{1})} |x - c_{l}|^{3} |Du_{l}(x)|^{3} \leq CE \big(u_{l}; B(c_{l}; 2r_{2}) \setminus B(c_{l}; \tfrac{1}{2} r_1) \big),
\end{equation}
whenever $2\sigma_{l} \leq r_{1} < r_{2} \leq \frac{1}{2} \rho_l$. In particular,
\begin{equation} \label{pointwiseboundlimit}
\lim_{l \to \infty} \Big( \sup_{x \in B(c_{l}; \frac{1}{2} \rho_l )\setminus B(c_{l}; 2\sigma_{l}) } |x - c_{l}| \, |Du_{l}(x)| \Big) = 0.
\end{equation}
\end{lemma}
\begin{proof}
To begin, we note two rather obvious facts. First of all, whenever $r_{1} \leq r \leq r_{2}$, we have 
\begin{equation*}
\min\{ 2r_{2} - r,\ r - \tfrac{1}{2} r_1 \} \geq \tfrac{1}{2} r.
\end{equation*}
Secondly, if $2\sigma_{l} \leq r_{1} < r_{2} \leq \frac{1}{2} \rho_l$, then 
\begin{equation*}
B(c_{l}; 2r_{2}) \setminus B(c_{l}; \tfrac{1}{2} r_1) \subseteq A_{l}.
\end{equation*}
It follows that, for any $l$, if $2\sigma_{l} \leq r_{1} < r_{2} \leq \frac{1}{2} \rho_l$, then for all $r \in [r_{1}, r_{2}]$ and $x \in \partial B(c_{l}; r)$ we have
\begin{equation*}
E\big( u_{l}; B(x; \tfrac{1}{2} r) \big) \leq E \big( u_{l}; B(c_{l}; 2r_{2}) \setminus B(c_{l}; \tfrac{1}{2} r_1) \big) \leq E(u_{l}; A_{l}).
\end{equation*}
Since we have shown in $\S$\ref{sec:no-energy-loss} that
\begin{equation} \label{noenergyloss}
\lim_{l \to \infty}E(u_{l}; A_{l}) = 0,
\end{equation}
we infer that there exists an $L_{0}$ such that, for all $l \geq L_{0}$, and $r_{1}, r_{2}, r, x$ as above, we may apply the mean value inequality (Theorem~\ref{thm:MVI}) on $B(x; \frac{1}{2} r)$ to get
\begin{equation*}
r^{3} |Du_{l}(x)|^{3} \leq CE \big( u_{l}; B(x; \tfrac{1}{2}r) \big) \leq CE \big( u_{l}; B(c_{l}; 2r_{2}) \setminus B(c_{l}; \tfrac{1}{2}r_1) \big).
\end{equation*}
Taking the supremum over $x \in \partial B(c_l; r)$ and $r \in [r_{1}, r_{2}]$ yields~\eqref{pointwiseannularbound}. The second conclusion follows by taking $r_{1} = 2\sigma_{l}$ and $r_{2} = \tfrac{1}{2} \rho_l$ in~\eqref{pointwiseannularbound} and recalling~\eqref{noenergyloss}.
\end{proof}

We introduce some additional notation before we continue. Below, we let $T_{0} = \log 2$. For all $l$ such that $\frac{\rho_l}{\sigma_l} > 4$ (recall that $\frac{\rho_l}{\sigma_l} \to \infty$), we let
\begin{equation*}
T_{l} = \log \sqrt{\frac{\rho_{l}}{\sigma_{l}}}.
\end{equation*}
Moreover, we define $f_{l}: [T_{0}, T_{l}] \to \R$ by 
\begin{equation*}
f_{l}(t) = E(u_{l}; B(c_{l}; e^{-t} \rho_{l}) \setminus B(c_{l}; e^{t}\sigma_{l})).
\end{equation*}
The following estimate is motivated by the proofs of~\cite[Lemma 4.7.3]{MS} and~\cite[Lemma A.4]{HS} and is the key to establishing~\eqref{neckdiameter}.
\begin{prop} \label{neckenergydecay}
There exists a constant $b > 0$ such that 
\begin{equation} \label{powerdecay}
f_{l}(t) \leq e^{-3b(t - T_{0})}f_{l}(T_{0}),
\end{equation}
for all $l$ sufficiently large and $t \in [T_{0}, T_{l}]$.
\end{prop}
\begin{proof}
As the conclusion suggests, we prove the proposition by showing that $f_{l}$ satisfies some first order differential inequality. To begin, suppose $2\sigma_{l} \leq r_{1} < r_{2} \leq \frac{1}{2} \rho_l$. Then by~\eqref{pointwiseboundlimit} and~\eqref{noenergyloss} we see that, for $l$ sufficiently large, the rescaled maps
\begin{equation*}
x \mapsto u_{l}(c_{l} + r_{2}x)
\end{equation*}
satisfy both~\eqref{smallboundarylength} and~\eqref{smallannularenergy}, with $r = \frac{r_{1}}{r_{2}}$. Therefore we apply Proposition~\ref{annularenergy} to deduce that 
\begin{equation*}
E(u_{l}; B(c_{l}; r_{2}) \setminus B(c_{l}; r_{1})) \leq C \Big( r_{2}\int_{\partial B(c_{l}; r_{2})}|du_{l}|^{3}dS_{g} + r_{1}\int_{\partial B(c_{l}; r_{1})}|du_{l}|^{3}dS_{g} \Big).
\end{equation*}
From this it is not hard to see that there exists some $b$ independent of $l$ such that 
\begin{equation*}
f_{l}'(t) \leq -3 b f_{l}(t) \quad \text{ for all $l$ sufficiently large and $t \in [T_{0}, T_{l}]$}.
\end{equation*}
Integrating from $T_{0}$ to $t$ gives~\eqref{powerdecay}.
 \end{proof}
 
At last we come to the proof of Theorem~\ref{thm:zero-neck-length}.
\begin{proof}[Proof of Theorem~\ref{thm:zero-neck-length}]
As remarked in the beginning of this section it suffices to prove~\eqref{neckdiameter}. Since $\frac{\rho_l}{\sigma_l} \to \infty$ as $l \to \infty$, we may without loss of generality assume that 
\begin{equation*}
\rho_{l} > 16\sigma_{l} \text{ for all }l.
\end{equation*}
Now suppose $4\sigma_{l} \leq r \leq (\rho_{l}\sigma_{l})^{\frac{1}{2}}$ and define $t$ by the equation
\begin{equation*}
e^{t}\sigma_{l} = \tfrac{1}{2}r.
\end{equation*}
Then $t \in [T_{0}, T_{l}]$ and moreover $r \leq \frac{1}{2} e^{-t} \rho_l$. Hence by~\eqref{pointwiseannularbound} (with $r_{2} = \frac{1}{2} e^{-t} \rho_l$ and $r_{1} = 2e^{t}\sigma_{l} = r$) and~\eqref{powerdecay}, we have
\begin{align}
\nonumber r^{3}\sup_{\partial B(c_l; r)}|Du_{l}|^{3} & \leq Cf_{l}(t) \leq Ce^{3bT_{0}}e^{-3bt}f_{l}(T_{0}) \\
\label{powerdecay2} & \leq Ce^{3bT_{0}} \Big( \frac{r}{2\sigma_{l}} \Big)^{-3b}f_{l}(T_{0}).
\end{align}
Therefore, in terms of polar coordinates around $c_{l}$, we deduce that for any $4\sigma_{l} \leq s \leq r \leq (\rho_{l}\sigma_{l})^{\frac{1}{2}}$ and $\xi, \eta \in S^{2}$, there holds
\begin{align*}
|u_{l}(s\xi) - u_{l}(r\eta)| & \leq |u_{l}(s\xi) - u_{l}(r\xi)| + |u_{l}(r\xi) - u_{l}(r\eta)|\\
& \leq \int_{4\sigma_{l}}^{r}|Du_{l}(t\xi)|dt + Cr\sup_{\partial B(c_{l}; r)}|Du_{l}|.
\end{align*}
We use~\eqref{powerdecay2} to estimate the first term on the last line by
\begin{align*}
\int_{4\sigma_{l}}^{r}|Du_{l}(t\xi)|dt & \leq C(2\sigma_{l})^{b} \Big(\int_{4\sigma_{l}}^{r}t^{-b-1}dt \Big) f_{l}(T_{0})^{\frac{1}{3}} \\
& \leq Cf_{l}(T_{0})^{\frac{1}{3}} \leq CE(u_{l}; A_{l})^{\frac{1}{3}}.
\end{align*}
Combining this with the previous string of inequalities, we infer that 
\begin{equation*}
\diam\left( u_{l}\left( B(c_{l}; \sqrt{\rho_{l}\sigma_{l}})\setminus B(c_{l}; 4\sigma_{l})\right) \right) \leq CE(u_{l}; A_{l})^{\frac{1}{3}} + C \sup_{x \in B(c_{l}; \frac{1}{2} \rho_l)\setminus B(c_{l}; 2\sigma_{l}) } |x - c_{l}| \, |Du_{l}(x)|,
\end{equation*}
which tends to zero as $l\to \infty$. A similar argument shows that 
\begin{equation*}
\diam\left( u_{l} \big( B(c_{l}; \tfrac{1}{4} \rho_l) \setminus B(c_{l}; \sqrt{\rho_{l}\sigma_{l}}) \big) \right) \to 0 \text{ as }l \to \infty,
\end{equation*}
and hence we have proven~\eqref{neckdiameter}.
\end{proof}

\begin{appendices}

\section{Harmonic Analysis} \label{sec:harmonic-analysis-appendix}

In this appendix we gather some results from harmonic analysis that are used in $\S$\ref{sec:int-reg-cpt}. We first introduce the function spaces before stating the main estimates. 

\begin{defn} \label{hardydefi} The Hardy space $\mathcal{H}^{1}(\R^{n})$ is defined as follows. Consider the class
\begin{equation*}
\mathcal{T} = \{\zeta \in C^{\infty}_{c}(B(1))\ |\ \| D \zeta \|_{\infty} \leq 1\}.
\end{equation*}
For $\zeta \in \mathcal{T}$ and $\ep > 0$, we write $\zeta_{\ep}(x) = \ep^{-n}\zeta(\frac{x}{\ep})$. Then, for any $g \in L^{1}(\R^{n})$, we define
\begin{equation*}
g^{\ast}(x) = \sup\limits_{\zeta \in \mathcal{T}}\sup\limits_{\ep > 0} \left| \int_{\R^{n}}\zeta_{\ep}(x - y) g(y)dy \right|.
\end{equation*}
By definition, $g$ belongs to $\mathcal{H}^{1}$ if $g^{\ast}$ belongs to $L^{1}(\R^{n})$, in which case we define $\|g\|_{\mathcal{H}^{1}} = \|g^{\ast}\|_{1}$.
\end{defn}

\begin{defn}
A function $f \in L^{1}_{\loc}(\R^{n})$ is in $BMO(\R^{n})$, which is called the space of functions with \textit{bounded mean oscillation}, if 
\begin{equation*}
[f]_{BMO} := \sup_{x \in \R^{n}, r > 0} \fint_{B(x; r)} |f - (f)_{B(x; r)}|dx < \infty.
\end{equation*}
\end{defn}

\begin{rmk} We make two remarks concerning these definitions.
\begin{enumerate}[(i)]
\item The Hardy space is a Banach space and is strictly contained in $L^{1}(\R^{n})$. 
\item Note that $[f]_{BMO} = [g]_{BMO}$ when $f, g$ differ by a constant. Thus $[\cdot]_{BMO}$ descends to the quotient of $BMO(\R^{n})$ by the constant functions and in fact makes it a Banach space.
\end{enumerate}
\end{rmk}

As a consequence of a deep result due to Fefferman--Stein~\cite{FS}, which identifies $BMO$ with the dual space of $\cH^{1}$, we have the following inequality.
\begin{thm}[{\cite[Theorem 2]{FS}}] 
\label{BMOdual}
Suppose $f \in L^{\infty}(\R^{n})$ and $g \in \mathcal{H}^{1}(\R^{n})$. Then
\begin{equation*}
\int_{\R^{n}} fg dx \leq C\|g\|_{\mathcal{H}^{1}} [f]_{BMO},
\end{equation*}
where $C$ depends only on the dimension $n$.
\end{thm}

The following two propositions can be found in~\cite{E} and contain estimates which, together with the Fefferman--Stein theorem above, are key to the regularity results in $\S$\ref{sec:int-reg-cpt}. 
\begin{prop} \label{cutoffBMO} Suppose $f \in W^{1, n}(B(2))$ and that $\int_{B(2)}f dx = 0$. Let $\zeta$ be a cutoff function which is identically $1$ on $B(1)$ and vanishes outside of $B(\frac{3}{2})$. Then $\zeta f \in BMO(\R^{n})$ with
\begin{equation*}
[\zeta f]_{BMO} \leq C\|Df\|_{n; B(2)}.
\end{equation*}
\end{prop}
\begin{prop} \label{divcurlhardy}
Let $f \in W^{1, n}_{0}(B(1))$ and let $X \in L^{\frac{n}{n-1}}(B(2); \R^{n})$ be a vector field which is divergence-free on $B(2)$ in the sense of distributions. Then $X \cdot D f \in \cH^{1}(\R^{n})$, and
\begin{equation*}
\|X \cdot D f\|_{\cH^{1}} \leq C\|D f\|_{n}\|X\|_{\frac{n}{n-1}}, 
\end{equation*}
where $C$ depends only on the dimension $n$.
\end{prop}
\begin{proof}
For the sake of completeness we include a proof of Proposition~\ref{divcurlhardy}. We would like to estimate the quantity
\begin{equation*}
\left| \int_{\R^{n}} \zeta_{\ep}(x - y) X(y) \cdot D f(y) dy\right|,
\end{equation*}
for any given $\zeta \in \cT$ and $\ep > 0$. To do that we need to treat a few different cases depending on where $x$ is and how large $\ep$ is.

\textbf{Case 1.} Suppose that $x \in B(\frac{3}{2})$ and $\ep < \frac{1}{2}$. Then $B(x; \ep) \subseteq B(2)$, so for any $\lambda \in \R$, the function
\begin{equation*}
y \, \mapsto \, (f(y) - \lambda)\zeta_{\ep}(x - y)
\end{equation*}
has compact support in $B(2)$. Since $\Div X$ vanishes on $B(2)$ in the sense of distributions, this implies
\begin{align*}
0 & = \int_{\R^{n}}X(y) \cdot D \left[(f(y) - \lambda) \zeta_{\ep}(x - y) \right]dy\\
& = \int_{\R^{n}} \zeta_{\ep}(x - y) X(y) \cdot D f(y)dy - \int_{\R^{n}}(f(y) - \lambda) X(y) \cdot D \zeta_{\ep}(x - y)dy.
\end{align*}
In other words, recalling the definition of $\zeta_{\ep}$
\begin{equation} \label{Xdivfree}
\int_{\R^{n}}\zeta_{\ep}(x - y) X(y) \cdot D f(y)dy = \ep^{-1}\fint_{B(x; \ep)}D \zeta \left(\ep^{-1}(x - y)\right)\cdot (f(y) - \lambda) X(y) dy.
\end{equation}
Next we choose an exponent $p > n$ and define 
\begin{equation*}
q = \frac{p}{p-1}, \qquad p_{\ast} = \frac{np}{p + n},
\end{equation*}
so that $W^{1, p_{\ast}}$ controls $L^{p}$ by the Sobolev embedding theorem, and $\frac{1}{p} + \frac{1}{q} = 1$. For later use, we note that $q < \frac{n}{n-1}$ and $p_{\ast} < n$. Now we choose $\lambda = \fint_{B(x; \ep)}f(y)dy$ and use the H\"older inequality, the fact that $\|D \zeta\|_{\infty} \leq 1$, and the Poincar\'e inequality to estimate the right hand side of~\eqref{Xdivfree} as follows:
\begin{align*}
& \left|\ep^{-1}\fint_{B(x; \ep)}D \zeta (\ep^{-1}(x - y))\cdot (f(y) - \lambda) X(y) dy \right| \\
& \qquad \leq \ep^{-1} \Big( \fint_{B(x; \ep)} |X|^{q} \Big)^{\frac{1}{q}} \Big( \fint_{B(x; \ep)} |f - \lambda|^{p} \Big)^{\frac{1}{p}} \\
& \qquad \leq \Big( \fint_{B(x; \ep)} |X|^{q} \Big)^{\frac{1}{q}} \Big( \fint_{B(x; \ep)} |D f|^{p_{\ast}} \Big)^{\frac{1}{p_{\ast}}} \\
& \qquad \leq \left[M(|X|^{q})(x)\right]^{\frac{1}{q}} \left[M(|D f|^{p_{\ast}})(x)\right]^{\frac{1}{p_{\ast}}},
\end{align*}
where $M(|X|^{q})$ and $M(|Df|^{p_{\ast}})$ in the last line denote \emph{maximal functions}. This means that
\begin{equation*}
M(F)(x) = \sup_{\ep > 0} \, \fint_{B(x; \ep)} F.
\end{equation*}

\textbf{Case 2.} Suppose that $x \in B(\frac{3}{2})$ and $\ep > \frac{1}{2}$. We observe that since 
\begin{equation*}
\int_{\R^{n}} X(y) \cdot D f(y) dy = 0,
\end{equation*}
we have
\begin{equation*}
\int_{\R^{n}} \zeta_{\ep}(x - y) X(y)\cdot D f(y) dy = \int_{B(1)} (\zeta_{\ep}(x - y) - \zeta_{\ep}(x)) X(y)\cdot D f(y) dy,
\end{equation*}
and hence, recalling that $\|D \zeta_{\ep}\|_{\infty} \leq \ep^{-n-1}$, we deduce that
\begin{align}
\nonumber & \left| \int_{\R^{n}} \zeta_{\ep}(x - y) X(y)\cdot D f(y) dy\right| \\
\label{largeepsilon}
& \qquad \leq \ep^{-n-1} \int_{B(1)} |X(y)| | D f(y)| dy \leq 2^{n + 1}\|X\|_{\frac{n}{n-1}}\|D f\|_{n}.
\end{align}

Combining Cases 1 and 2, we get that for $x \in B(\frac{3}{2})$, there holds
\begin{equation*}
(X \cdot D f)^{\ast}(x) \leq \left[M(|X|^{q})(x)\right]^{\frac{1}{q}} \left[M(|D f|^{p_{\ast}})(x)\right]^{\frac{1}{p_{\ast}}} + 2^{n + 1}\|X\|_{\frac{n}{n-1}} \|D f\|_{n}.
\end{equation*}
Consequently 
\begin{equation} \label{case1and2}
\|(X \cdot D f)^{\ast}\|_{1; B(\frac{3}{2})} \leq \| M(|X|^{q})^{\frac{1}{q}} M(|Df|^{p_{\ast}})^{\frac{1}{p_{\ast}}} \|_{1} + C\|X\|_{\frac{n}{n-1}} \|D f\|_{n}.
\end{equation}
To continue, note that the function $|X|^{q}$ lies in $L^{\alpha}(\R^{n})$ for $\alpha = \frac{n}{q(n-1)} > 1$, and that the function $|D f|^{p_{\ast}}$ lies in $L^{\beta}(\R^{n})$ for $\beta = \frac{n}{p_{\ast}} > 1$. Hence 
\begin{align*}
\| M(|X|^{q}) \|_{\alpha} & \leq C \| \, |X|^{q} \, \|_{\alpha} = C\|X\|_{\frac{n}{n-1}}^{q}, \\
\|M(|D f|^{p_{\ast}})\|_{\beta} & \leq C\| \, |D f|^{p_{\ast}} \, \|_{\beta} = C\|D f\|_{n}^{p_{\ast}}.
\end{align*}
Combining these estimates with the H\"older inequality and the fact that $\frac{1}{\alpha q} + \frac{1}{\beta p_{\ast}} = 1$, we get
\begin{align*}
\| M(|X|^{q})^{\frac{1}{q}} M(|D f|^{p_{\ast}})^{\frac{1}{p_{\ast}}} \|_{1} & \leq \|M(|X|^{q})^{\frac{1}{q}}\|_{q\alpha} \|M(|D f|^{p_{\ast}})^{\frac{1}{p_{\ast}}}\|_{p_{\ast}\beta} \\
& = \|M(|X|^{q})\|_{\alpha}^{\frac{1}{q}} \|M(|D f|^{p_{\ast}})\|_{\beta}^{\frac{1}{p_{\ast}}} \\
& \leq C\|X\|_{\frac{n}{n-1}} \|D f\|_{n}.
\end{align*}
Putting this back into~\eqref{case1and2}, we get
\begin{equation} \label{hardybound1}
\|(X \cdot D f)^{\ast}\|_{1; B(\frac{3}{2})} \leq C\|X\|_{\frac{n}{n-1}} \|D f\|_{n}.
\end{equation}

\textbf{Case 3.} Suppose that $x \in \R^{n} \setminus B(\frac{3}{2})$. Then since $D f$ is nonzero only in $B(1)$ and since $B(x; \ep) \cap B(1) = \varnothing$ if $\ep < \frac{1}{3}|x|$, we have in this case that 
\begin{equation*}
(X \cdot Df)^{\ast}(x) = \sup\limits_{\zeta}\sup\limits_{\ep \geq \frac{1}{3}|x|} \ep^{-n}\bigg| \int_{B(1)}\zeta(\ep^{-1}(x - y)) X(y) \cdot Df(y)dy \bigg|.
\end{equation*}
Then, the same observation that lead to~\eqref{largeepsilon} allows us to conclude that in fact
\begin{equation*}
(X \cdot Df)^{\ast}(x) \leq 3^{n + 1}|x|^{-n-1} \|X\|_{\frac{n}{n-1}}\|D f\|_{n}.
\end{equation*}
Integrating this over $\R^{n}\setminus B(\frac{3}{2})$ gives
\begin{equation} \label{hardybound2}
\|(X \cdot D f)^{\ast}\|_{1; \R^{n}\setminus B(\frac{3}{2})} \leq C\|X\|_{\frac{n}{n-1}}\|D f\|_{n}.
\end{equation}
We complete the proof upon combining~\eqref{hardybound1} and~\eqref{hardybound2}, and recalling the definition of $\cH^{1}$.
\end{proof}

\section{Higher regularity and the mean value inequality} \label{sec:regularity-appendix}

In this appendix we give proofs of Theorems~\ref{thm:ep-reg} and~\ref{thm:MVI}. The arguments are largely based on the work of Duzaar--Fuchs~\cite{DF} and Duzaar--Mingione~\cite{DM} on $p$-harmonic maps, with some necessary modifications due to the fact that we have a different system of equations, and because we are working with a potentially nonflat metric on the domain. 

\begin{proof}[Proof of Theorem~\ref{thm:ep-reg}]
We begin with part (a), whose proof actually resembles the proof of Proposition~\ref{energydecay}, but we estimate things slightly differently. Take any $x_{0} \in B(1)$ and $r < \frac{1}{8}$, and again introduce the rescalings
\begin{equation*}
\tilde{u}(x) = u(x_{0} + rx), \qquad \tilde{g}(x) = g(x_{0} + rx).
\end{equation*}
Then Corollary~\ref{holderregularity} is applicable to $\tilde{u}$, so we have~\eqref{morreyestimate} and~\eqref{holderestimate} at our disposal. Next, we let $g^{0}$ denote the constant metric given by 
\begin{equation*}
g^{0}_{ij}(x) := \tilde{g}_{ij}(0) = g_{ij}(x_{0}) \text{ for all }x \in B(2).
\end{equation*}
Similarly to the proof of Proposition~\ref{energydecay}, we consider the unique function $\tilde{v} \in W^{1, 3}(B(1); \R^{d})$ which minimizes 
\begin{equation*}
\int_{B(1)}|dw|_{g^{0}}^{3}d\mu_{g^{0}}
\end{equation*}
amongst all functions $w \in W^{1, 3}(B(1); \R^{d})$ with $w - \tilde{u} \in W^{1, 3}_{0}(B(1); \R^{d})$. Then $\tilde{v}$ satisfies
\begin{equation*}
\int_{B(1)} \langle |d\tilde{v}|_{g^{0}} d\tilde{v}, d\eta \rangle_{g^{0}}d\mu_{g^{0}} = 0 \quad \text{ for all }\eta \in W^{1, 3}_{0}(B(1); \R^{d}),
\end{equation*}
and
\begin{equation} \label{vg0minimizing}
\int_{B(1)} |d\tilde{v}|_{g^{0}}^{3}d\mu_{g^{0}} \leq \int_{B(1)}|d\tilde{u}|^{3}_{g^{0}}d\mu_{g^{0}}.
\end{equation}
Following~\cite{DM}, we introduce the function $V: \R^{3 \times d} \to \R^{3 \times d}$ defined by 
\begin{equation*}
V(\xi) = |\xi|^{\frac{1}{2}}\xi.
\end{equation*}
Then by a refinement of Uhlenbeck's work due to Giaquinta--Modica~\cite[Theorem, 3.1]{GiMa}, we know that $\tilde{v}$ has H\"older continuous first derivatives, and furthermore there exist universal constants $C$ and $\gamma$ such that 
\begin{equation} \label{Giaquinta}
\int_{B(\lambda)} \left| V(D\tilde{v}) - (V(D\tilde{v}))_{B(\lambda)} \right|^{2}dx \leq C\lambda^{3 + 2\gamma}\int_{B(1)}|V(D\tilde{v}) - (V(D\tilde{v}))_{B(1)}|^{2}dx \quad \text{ for all }\lambda \in (0, \tfrac{1}{4}).
\end{equation}
Also, the argument leading to~\eqref{esssupbound} and~\eqref{essinfbound} is still valid, and for all $x \in B(1)$ and $i = 1, \ldots, d$, we get
\begin{equation*}
\min_{B(1)} \tilde{u}^{i} - \tilde{u}^{i}(x) \leq \tilde{v}^{i}(x) - \tilde{u}^{i}(x) \leq \max_{B(1)}\tilde{u}^{i} - u^{i}(x).
\end{equation*}
In particular, by~\eqref{holderestimate}, we have
\begin{equation*}
|\tilde{v} - \tilde{u}|_{0; B(1)} \leq C[\tilde{u}]_{\alpha; B(1)} \leq C\|D\tilde{u}\|_{3; B(2)}.
\end{equation*}
Therefore, using $\tilde{u} - \tilde{v}$ as a test function in~\eqref{2ndorderSmithequation2}, we get
\begin{align*}
\int_{B(1)} \langle |d\tilde{u}|d\tilde{u}, d(\tilde{u} - \tilde{v}) \rangle d\mu_{g} & \leq C\|D\tilde{u}\|_{3; B(2)}^{4} \\
& \leq C\|Du\|_{3; B(x_0; 2r)}^{4} \leq Cr^{4\alpha}\|Du\|_{3; B(2)}^{4},
\end{align*}
where we used~\eqref{morreyestimate} in the last inequality. On the other hand, repeating the calculation in~\eqref{modified3harmonic} and~\eqref{errortermestimate}, and using~\eqref{vg0minimizing}, we deduce that
\begin{align*}
\int_{B(1)} \langle |d\tilde{v}|_{g}d\tilde{v}, d(\tilde{u} - \tilde{v}) \rangle_{g}d\mu_{g} & \leq C|\tilde{g} - g^{0}|_{0; B(1)} \|D\tilde{v}\|_{3; B(1)}^{2}\|D\tilde{u} - D\tilde{v}\|_{3; B(1)} \\
& \leq C|g - g(x_{0})|_{0; B(x_0; r)}\|D\tilde{u}\|^{3}_{3; B(1)} \\
& \leq Cr|Dg|_{0; B(2)} r^{3\alpha}\|Du\|_{3; B(2)}^{3} \\
& \leq C\ep_{0} r^{1 + 3\alpha} \|Du\|_{3; B(2)}^{3}.
\end{align*}
Letting $v(x) = \tilde{v}(r^{-1}(x - x_{0}))$, the above gives
\begin{align*}
\int_{B(x_0; r)} \left| V(Du) - V(Dv) \right|^{2}dx & = \int_{B(1)} \left| V(D\tilde{u}) - V(D\tilde{v}) \right|^{2} dx\\
 & \leq C\int_{B(1)} \langle |d\tilde{u}|_{g}d\tilde{u} - |d\tilde{v}|_{g}d\tilde{v}, d\tilde{u} - d\tilde{v} \rangle d\mu_{g} \\
 & \leq Cr^{4\alpha}(1 + \|Du\|_{3; B(2)})\|Du\|_{3; B(2)}^{3}.
\end{align*}
Note that we used~\cite[equation (11) on page 240]{DM} in the second line above.

Scaling $\tilde{v}$ to $v$ in~\eqref{Giaquinta} and combining it with estimate just obtained, we find that for all $x_{0} \in B(1)$, $r < \frac{1}{8}$, and $\lambda \in (0, \frac{1}{4})$, we have 
\begin{align*}
\int_{B(x_0; \lambda r)} \left| V(Du) - (V(Du))_{B(x_{0}; \lambda r)} \right|^{2}dx \leq\ &C\lambda^{3 + 2\gamma}\int_{B(x_0; r)} \left| V(Du) - (V(Du))_{B(x_{0}; r)} \right|^{2}dx\\
&+ Cr^{4\alpha} \left( 1 + \|Du\|_{3; B(2)} \right)\|Du\|_{3; B(2)}^{3}.
\end{align*}
Letting $K = C\left( 1 + \|Du\|_{3; B(2)} \right)\|Du\|_{3; B(2)}^{3}$ and 
\begin{equation*}
\Phi(x_{0}, t) = \int_{B(x_0; t)} \left| V(Du) - (V(Du))_{B(x_{0}; t)} \right|^{2}dx,
\end{equation*}
we see that for all $x_{0} \in B(1)$ and $s \leq r < \frac{1}{8}$, we have
\begin{equation*}
\Phi(x_{0}, s) \leq C\left( \frac{s}{r} \right)^{3 + 2\gamma}\Phi(x_0, r) + Kr^{4\alpha}.
\end{equation*}
Recalling that $4\alpha > 3$, we obtain by~\cite[Lemma 3.4]{Han-Lin} that 
\begin{equation*}
\Phi(x_{0}, s) \leq C s^{3 + 2\gamma'}
\end{equation*}
for all $x_{0} \in B(1)$ and $s \in (0, \frac{1}{9})$, where $\gamma' = \min\{\gamma, \frac{1}{2}(4\alpha - 3) \}$. Thus $V(Du) \in C^{0, \gamma'}(B(1); \R^{3 \times d})$ with $[V(Du)]_{\gamma'; B(1)}$ depending on $M, J$, the embedding $M \to \R^{d}$, and $\|Du\|_{3; B(2)}$. Then~\cite[Lemma 3]{DM} yields the assertion in part (a).

For (b) we only indicate the main steps. Fix any ball $B$ with compact closure in the set
\[
\{x \in B(1)\ |\ Du(x) \neq 0 \} \quad (= B(1) \setminus \crit_u).
\] Via a difference quotient argument (compare with~\cite[Lemma 2.2]{DF}) using~\eqref{2ndorderSmithequation2}, we can show that 
\begin{equation*}
\int_{B} |Du| \, |D^{2}u|^{2} dx < \infty,
\end{equation*}
so that $u \in W^{2, 2}(B)$, because $|Du|$ is bounded away from zero on $\overline{B}$. This $W^{2, 2}$-regularity then allows us to integrate by parts on the left hand side of~\eqref{2ndorderSmithequation2} to see that $u$ satisfies an elliptic system of equations of the form
\begin{equation*}
A^{\gamma\lambda}_{ij}u^{j}_{\gamma\lambda} + B_{ij}^{\lambda}u^{j}_{\lambda} = F_{i},
\end{equation*}
almost everywhere in $B$, where $A^{\gamma\lambda}_{ij}, B_{ij}^{\lambda}$ and $F_{i}$ are rational expressions involving $g, Dg, J \circ u$ and $Du$ and hence all lie in $C^{0, \beta}(B)$ because $u \in C^{1, \beta}(B)$, and because $g$ is smooth by assumption. The regularity theory for $W^{2, 2}$-solutions to elliptic systems of the above type (see~\cite[Theorem 5.22]{GM}) tells us that $u \in C_{\loc}^{2, \beta}(B)$, and a bootstrapping argument yields smoothness locally in $B$. Since $B$ is an arbitrary ball with compact closure in $\{x \in B(1)\ |\ Du(x) \neq 0 \}$, we are done.
\end{proof}

\begin{proof}[Proof of Theorem~\ref{thm:MVI}]
Throughout this proof, we use $d^{\ast}$ to denote the dual, with respect to $g$, of the exterior derivative. The operators $\Delta$ and $\nabla$, as well as the volume form $\mu$, are with respect to the metric $g$, and $\nabla du$ refers to the covariant derivative of $du$ as a $\R^{d}$-valued $1$-form, as opposed to a section of $\R^{3} \otimes u^{\ast}TM$.

We first require $\ep_{1} < \ep_{0}$, so that by Theorem~\ref{thm:ep-reg} we have $u \in C^{1, \beta}(B(1); M)$ and $u$ is smooth on the set $\{x \in B(1)\ |\ Du(x) \neq 0\}$, which we denote by $U_{+}$ for this proof.

Next, we note that the difference quotient argument in the proof of~\cite[Lemma 2.2]{DF} carries over to our case to give $|du|^{\frac{1}{2}}du \in W^{1, 2}_{\loc}(B(1))$. Since $u \in C^{1}(B(1))$, we deduce that $|du|^{3} \in W^{1, 2}_{\loc}(B(1))$. Next we show, as in~\cite[Lemma 2.3]{DF}, that the function $w := |du|^{3}$ satisfies a second order, uniformly elliptic equation in $U_{+}$.

The computation in the proof of~\cite[Lemma 2.3]{DF} can then be modified as follows. (See also~\cite[pages 222--223]{U}.) At any point in $U_{+}$ we have
\begin{align*}
\Delta w & = \Delta \langle |du|du, du \rangle\\
& = \langle \Delta \left( |du|du \right), du \rangle + 2\langle \nabla \left( |du|du \right), \nabla du \rangle + \langle |du|du, \Delta du \rangle\\
& = -\langle (d^{\ast}d + dd^{\ast}) \left( |du|du \right), du \rangle + \Ric(|du|du, du) + 2\langle \nabla \left( |du|du \right), \nabla du \rangle + \langle |du|du, \Delta du \rangle\\
& = -\langle (d^{\ast}d + dd^{\ast})\left( |du|du \right), du \rangle + \Ric(|du|du, du) + \langle \nabla \left( |du|du \right), \nabla du \rangle + \tfrac{1}{3}\Delta w.
\end{align*}
(Note that we used the Weitzenb\"ock formula for differential $1$-forms in going from the second line to the third.) Hence we get
\begin{align} \label{bochnerprecursor}
\nonumber \tfrac{2}{3}\Delta w + \langle d^{\ast}d\left( |du|du \right), du \rangle & = -\langle dd^{\ast}\left( |du|du \right), du \rangle + \Ric(|du|du, du) + \langle \nabla (|du|du), \nabla du \rangle\\
\nonumber & \geq -\langle dd^{\ast}\left( |du|du \right), du \rangle - K|du|^{3} + |du|^{-1}\langle du, \nabla du \rangle^{2} + |du| |\nabla du|^{2} \\
& \geq -\langle dd^{\ast}\left( |du|du \right), du \rangle - K|du|^{3} + |du| |\nabla du|^{2}.
\end{align}
To continue, note that we have
\begin{equation*}
 \langle d^{\ast}d\left( |du|du \right), du \rangle = -\tfrac{1}{3}\Div\left( dw - |du|^{-2}\langle du, dw \rangle du \right),
\end{equation*}
and thus the left hand side of~\eqref{bochnerprecursor} becomes
\begin{equation} \label{bochnerprincipalpart}
\tfrac{2}{3}\Delta w + \langle d^{\ast}d\left( |du|du \right), du \rangle = \tfrac{1}{3}\Div ( dw + |du|^{-2}\langle du, dw \rangle du ),
\end{equation}
where the right hand side can be understood as a uniformly elliptic operator of second order acting on $w$. For the right hand side of~\eqref{bochnerprecursor}, the first term can be estimated with the help of~\eqref{2ndorderSmithequation2} as follows:
\begin{align}
-\langle dd^{\ast}\left( |du|du \right), du \rangle & = \tfrac{\sqrt{3}}{2} \big\langle d \big(\tfrac{1}{\sqrt{g}} \perm^{\alpha \beta \gamma} (J^{i}_{jk}\circ u)_{\gamma} u^{j}_{\alpha}u^{k}_{\beta} \big), du^{i} \big\rangle \nonumber \\
& \geq -C|du|^{4} - C|du|^{5} - C|du|^{3}|D^{2}u|\nonumber \\
& \geq -C|du|^{3} - C|du|^{5} - C|du|^{3}|D^{2}u| \label{eq:nonlinearity-bound},
\end{align}
where we used Young's inequality to estimate $|du|^{4} \leq C(|du|^3 + |du|^5)$ in the last inequality above.

Next, by the coordinate expression for $\nabla du$, we have 
\[
|D^{2}u| \leq C(|\nabla du| + |du|).
\]
Using this and Young's inequality again, we bound the last term in~\eqref{eq:nonlinearity-bound} as follows.
\begin{align}
|du|^{3}|D^{2}u| & \leq C|du|^{3}(|\nabla du| + |du|) \nonumber \\
&\leq C \big( \tfrac{1}{4 \ep} |du|^{5} + \ep |du||\nabla du|^{2} \big) + C \big( \tfrac{1}{2}|du|^{3} + \tfrac{1}{2}|du|^{5} \big).
\end{align}
Combining this with~\eqref{eq:nonlinearity-bound} and choosing a small enough $\ep$, we arrive at
\[
-\langle dd^{\ast}\left( |du|du \right), du \rangle \geq -C|du|^{3} - C|du|^{5} - \tfrac{1}{2}|du||\nabla du|^{2}.
\]
Putting this back into the right hand side of~\eqref{bochnerprecursor} and using~\eqref{bochnerprincipalpart}, we obtain
\begin{equation} \label{bochner}
\tfrac{1}{3}\Div ( dw + |du|^{-2}\langle du, dw \rangle du ) \geq -Cw^{\frac{5}{3}} - Cw,
\end{equation}
pointwise everywhere on $U_{+}$. Hence, for all $\zeta \in C^{1}_{c}(U_{+})$ with $\zeta \geq 0$, we have
\begin{equation*}
\int_{U_{+}} \big( \langle dw, d\zeta \rangle + |du|^{-2}\langle du, dw \rangle \langle du, d\zeta \rangle \big)d\mu \leq C\int_{U_{+}} (w^{\frac{5}{3}} + w)\zeta d\mu.
\end{equation*}

Now the proof of~\cite[Lemma 2.4]{DF} carries over to show that the above inequality actually holds for all $\zeta \in C^{1}_{0}(B(1))$ with $\zeta \geq 0$. Consequently $w \in W^{1, 2}_{\loc}(B(1))$ is a weak solution to~\eqref{bochner} on all of $B(1)$. We may then follow~\cite[pages 394--395]{DF} to complete the proof, shrinking $\ep_{1}$ if necessary.
\end{proof}

\end{appendices}

\end{document}